\documentclass[11pt]{article}
\usepackage{amsmath}
\usepackage{amssymb}
\usepackage{amsthm}
\usepackage{mathrsfs}
\usepackage{geometry}
\usepackage{stmaryrd}
\usepackage{tikz}
\usepackage{xcolor}
\usepackage{appendix}
\usepackage{titlesec}
\usepackage{authblk}

\usepackage{url}
\usepackage[colorlinks,
            linkcolor=blue,
            urlcolor=green,
            citecolor=red]{hyperref}

\numberwithin{equation}{section}
\numberwithin{figure}{section}

\newtheorem{definition}{Definition} 
\newtheorem{theorem}{Theorem}
\newtheorem{lemma}[theorem]{Lemma} 
\newtheorem{corollary}[theorem]{Corollary}
\newtheorem{proposition}[theorem]{Proposition}
\newtheorem{example}{Example}

\newtheorem*{remark}{Remark}

\DeclareMathOperator{\spn}{Span}
\DeclareMathOperator{\rank}{rank}
\DeclareMathOperator{\diag}{diag}
\DeclareMathOperator{\conv}{Conv}

\newcommand{\Address}{{
\bigskip
\footnotesize
\textsc{Jianghao Zhang: Mathematical Institute, University of Bonn, Endenicher Allee 60, 53115 Bonn, Germany}\par\nopagebreak
\textit{Email address}: \texttt{Jianghao@uni-bonn.de}
}}

\begin{document}
\title{\textbf{Non-Degenerate Multilinear Singular Multipliers with Fractional Rank}}
\author{Jianghao Zhang}
\date{}
\maketitle

\begin{abstract}
We establish $L^p$ estimates for multilinear multipliers acting on $(n-1)$-tuples of functions on $\mathbb{R}^d$. We assume that the multiplier satisfies symbol estimates outside a linear subspace of dimension $m$. The difficulty of proving $L^p$ bounds increases with the rank $\frac{m}{d}$. Such bounds are unachievable using existing tools when $\frac{m}{d} \geq \frac{n}{2}$. And when $\lceil \frac{m}{d}\rceil<\frac{n}{2}$, our results do not improve upon those in \cite{MTTMultilinearSingularMultipliers}. Therefore, our focus is on the fractional rank case $\frac{m}{d}<\frac{n}{2}\leq \lceil \frac{m}{d}\rceil$ proposed in \cite{DPTFractionalRank}. We extend the results of \cite{DPTFractionalRank} by covering a broader range of Lebesgue exponents. Additionally, we develop the tree counting argument, thereby including cases with larger $\frac{m}{d}$ compared to \cite{DPTFractionalRank}. In doing so, we also identify and correct an inaccuracy in \cite{DPTFractionalRank}.
\end{abstract}

\section{Introduction} \label{Section1}
We use $\xi_j=(\xi_{j, 1}, \cdots, \xi_{j, d})\in \mathbb{R}^d$ and $\xi:=(\xi^{'}, \xi_n):=(\xi_1, \cdots, \xi_n)\in (\mathbb{R}^d)^n$ for frequency variables. We will consider the boundedness of multilinear multipliers of the form:
\begin{equation}  \label{Section1SymmetricFormToBeDefined}
\int_{\Gamma_0} \mathfrak{m}(\xi_1, \cdots, \xi_n)\widehat{f_1}(\xi_1)\cdots \widehat{f_n}(\xi_n).
\end{equation}
Here the function $\mathfrak{m}$, defined on $\Gamma_0:=\{\xi\in (\mathbb{R}^d)^n: \sum_{j=1}^n \xi_j=0\}$, is singular on a subspace $\Gamma\subset \Gamma_0$. The precise definition is given as follows:
\begin{definition}
Let $\Gamma\subset \Gamma_0$ be a linear subspace and $\mathfrak{m}: \Gamma_0\rightarrow \mathbb{C}$ be a function satisfying
\begin{equation}  \label{Section1MultiplierDecay1}
|\partial^{\alpha}\mathfrak{m} (\xi)|\lesssim_{\alpha} d(\xi, \Gamma)^{-|\alpha|},\ \forall \alpha: |\alpha|\leq N_2
\end{equation}
for some large $N_2$. Let $f_1, \cdots, f_{n-1}\in \mathcal{S}(\mathbb{R}^d)$. Define
\begin{equation*}
\begin{split}
T_{\mathfrak{m}} &(f_1, \cdots, f_{n-1})(x) \\
&:=\int_{(\mathbb{R}^d)^{n-1}} e^{2\pi i x\cdot(\sum_{j=1}^{n-1} \xi_j)} \mathfrak{m}(\xi_1, \cdots, \xi_{n-1}, -\sum_{j=1}^{n-1} \xi_j)\widehat{f_1}(\xi_1)\cdots \widehat{f_{n-1}}(\xi_{n-1}) d\xi_1\cdots d\xi_{n-1}.
\end{split}
\end{equation*}
The associated multilinear form is given by
\begin{equation*}
\begin{split}
\Lambda_{\mathfrak{m}} (f_1, \cdots, f_{n}) &:=\langle T_{\mathfrak{m}} (f_1, \cdots, f_{n-1}), f_n \rangle \\
&=\int_{(\mathbb{R}^d)^{n-1}} \mathfrak{m}(\xi_1, \cdots, \xi_{n-1}, -\sum_{j=1}^{n-1} \xi_j)\widehat{f_1}(\xi_1)\cdots \widehat{f_{n-1}}(\xi_{n-1}) \widehat{f_n}(-\sum_{j=1}^{n-1} \xi_j) d\xi_1\cdots d\xi_{n-1},
\end{split}
\end{equation*}
where the last line also serves as the definition of \eqref{Section1SymmetricFormToBeDefined}.
\end{definition}
\begin{remark}
When integrating over $\Gamma_0$, we always use the normalized Lebesgue measure as above.
\end{remark}

There has been a long interest in studying the boundedness of $T_{\mathfrak{m}}$. The decay condition \eqref{Section1MultiplierDecay1} makes $T_{\mathfrak{m}}$ have modulation symmetry along $\Gamma$, so tools from time-frequency analysis are particularly useful. This field can be traced back to the classical works on pointwise convergence of Fourier series \cite{CarlesonCarleson}\cite{FeffermanCarleson}. Then seminal breakthroughs \cite{LTBHT1}\cite{LTBHT2}, made by Lacey and Thiele, established boundedness of the bilinear Hilbert transform. It is the first non-trivial example of $T_{\mathfrak{m}}$ with $\dim(\Gamma)>0$. Their method also led to a new short proof of Carleson's theorem \cite{LTCarleson}. See \cite{DTL^pOuterMeasure}, \cite{UWBilinearHilbertUniform}, \cite{DTTwoDimensionalBHT}, \cite{KovacTwistedParaproduct}\cite{DSTMultilinearCubicalStructure}, \cite{BMTheHelicoidalMethod}, \cite{LiePolynomialCarleson}\cite{Zorin-KranichPolynomialCarleson}, and references therein for some recent developments.

The systematic study for general $T_{\mathfrak{m}}$ with $d=1$ originated in \cite{MTTMultilinearSingularMultipliers}. For general $d$, the quantity $\frac{m}{d}$ measures the complexity of our operators. We call this quantity rank, and it can be a fractional number when $d\geq 2$. We only consider non-degenerate $\Gamma$ such that the rank is strictly smaller than $\frac{n}{2}$, because this is the range where our existing time-frequency tools work well.

Now let's give some definitions about the non-degeneracy of $\Gamma$. There are two types of non-degenerate conditions. The first can be traced back to \cite{MTTMultilinearSingularMultipliers}, which is sufficient for the case $\lceil\frac{m}{d}\rceil<\frac{n}{2}$. The second was introduced in \cite{DPTFractionalRank} to handle fractional rank operators such that $\frac{m}{d}<\frac{n}{2}\leq \lceil\frac{m}{d}\rceil$. Since the $\lceil\frac{m}{d}\rceil<\frac{n}{2}$ case was essentially solved in \cite{MTTMultilinearSingularMultipliers}, we will mainly consider the range $\frac{m}{d}<\frac{n}{2}\leq \lceil\frac{m}{d}\rceil$. This case requires both non-degenerate conditions.
\begin{definition}
Given $\Theta\subset[n]\times [d]$, consider the projection $\Pi_{\Theta}$ from $\Gamma_0$ to the coordinates $(\xi_{j, k})_{(j, k)\in \Theta}$. Let $v\in \Gamma_0$. Write $(v)_{\Theta}$ for $\Pi_{\Theta}v$. For $A\subset [n]$, consider $\Pi_{\Theta}$ with $\Theta:=A\times [d]$, let $P_{A}$ denote $\Pi_{\Theta}$ restriced on $\Gamma$.
\end{definition}
\begin{remark}
Using this notation, We interpret \eqref{Section1MultiplierDecay1} as follows: Regard $\mathfrak{m}$ as a function of $\xi^{'}:=(\xi_1, \cdots, \xi_{n-1})$, and \eqref{Section1MultiplierDecay1} means
\begin{equation}  \label{Section1MultiplierDecay2}
|\partial^{\alpha}\mathfrak{m} (\xi^{'})|\lesssim_{\alpha} d(\xi^{'}, P_{[n-1]}\Gamma)^{-|\alpha|},\ \forall \alpha: |\alpha|\leq N.  
\end{equation}
In fact, assume \eqref{Section1MultiplierDecay2} holds. Then the same type of estimates holds if we regard $\mathfrak{m}$ as the function of any $n-1$ variables. Thus we use $\Gamma_0, \Gamma$ and write \eqref{Section1MultiplierDecay2} as \eqref{Section1MultiplierDecay1} to emphasize that $\xi_1, \cdots, \xi_n$ are at symmetric positions.
\end{remark}

\begin{definition}[Type \uppercase\expandafter{\romannumeral1}] \label{Section1DefinitionType1NoN-Degenerate}
Assume $\frac{m}{d}<\frac{n}{2}$. We say $\Gamma$ satisfies Type \uppercase\expandafter{\romannumeral1} non-degenerate condition if for any $A\subset [n], \#A=\lceil \frac{m}{d}\rceil$, we have
\begin{equation*}
\dim \Big(P_A(\Gamma) \Big)=m.    
\end{equation*}
\end{definition}
\begin{remark}
This means that $\Gamma$ is the graph over the coordinates $(\xi_{j, k})_{(j, k)\in \Theta}$ for some $\Theta\subset A\times [d]$ such that $\# \Theta=m$.
\end{remark}

When $\frac{m}{d}<\frac{n}{2}\leq \lceil\frac{m}{d}\rceil$, there are two possibilities. The first is the case $\frac{m}{d}\leq \frac{n-1}{2}$. This situation was handled in \cite{DPTFractionalRank}, and we call it the medium rank case. The second is the case $\frac{n-1}{2}<\frac{m}{d}$. Call this situation the large rank case. The large rank case was claimed to be solved in \cite{DPTFractionalRank}, but the non-degenerate condition there holds for no $\Gamma$. For example, consider (41) in \cite{DPTFractionalRank}. Let all terms on the right to be zero and $\overrightarrow{\xi}^{(1)}=\cdots=\overrightarrow{\xi}^{(d)}$. We see the number of remaining variables are more than the number of remaining equations, so there exist non-zero solutions. We will extend the result for the medium rank case in \cite{DPTFractionalRank} and also tackle the large rank case. However, for the large rank case, Our argument only works well when $n$ is large in terms of $d$. If $n$ is small, the operator can satisfy a large group of symmetry like the triangular Hilbert transform. See Example \ref{Section1ExampleFractionalRankSmalln} for more details.

\begin{definition}[Type \uppercase\expandafter{\romannumeral2}] \label{Section1Non-DegenerateCondition2}
Assume $\frac{m}{d}<\frac{n}{2}\leq \lceil\frac{m}{d}\rceil$. Consider the linear map $\mathfrak{L}$ defined on $\Gamma^L$ which maps $(v^{(1)}, \cdots, v^{(L)})$ to
\begin{equation}  \label{Section1EquationTheCrucialMap}
\Big( \big(P_A(v^{(1)})-P_A(v^{(k)})\big)_{2\leq k\leq L}, \big( P_{B^{(k)}}(v^{(k)})\big)_{1\leq k\leq L} \Big).
\end{equation}
We define $L$, $A, B^{(1)}, \cdots, B^{(L)}$ and Type \uppercase\expandafter{\romannumeral2} non-degeneracy as follows:
\begin{itemize}
\item When $\frac{m}{d}\leq \frac{n-1}{2}$, let $L:=2$. Note that $n$ must be an even integer. Write $n=2n^{'}+2$. Let $A, B^{(1)}, B^{(2)}$ be any disjoint subsets of $[n]$ such that $\#A=1, \#B^{(1)}=\#B^{(2)}=n^{'}$. We say $\Gamma$ satisfies Type \uppercase\expandafter{\romannumeral2} non-degenerate condition if the linear map $\mathfrak{L}$ satisfies $\ker (\mathfrak{L})=\{0\}$ for any possible choice of $A, B^{(1)}, \cdots, B^{(L)}$.
\item When $\frac{n-1}{2}<\frac{m}{d}$, let $L:=d$. Write $m=ld+r, 0\leq r\leq d-1$. Further, write $l+1-(d-r)=ad+b, 0\leq b\leq d-1$. Let $A, U_1, \cdots, U_{2d-1}, W$ be any disjoint subsets of $[n]$ such that $\#A=d-r, \#U_1=\cdots=\#U_{2d-1}=a, \#W=b$. Define
\begin{equation*}
B^{(k)}:=(U_k\cup \cdots\cup U_{k+d-1})\cup W,\ 1\leq k\leq d.
\end{equation*}
We say $\Gamma$ satisfies Type \uppercase\expandafter{\romannumeral2} non-degenerate condition if $l+1-(d-r)\geq 0$ and the linear map $\mathfrak{L}$ satisfies $\ker (\mathfrak{L})=\{0\}$ for any possible choice of $A, B^{(1)}, \cdots, B^{(L)}$.
\end{itemize}
\end{definition}
These conditions look complicated, but they hold for almost all $\Gamma$ in the following sense.
\begin{definition}
Fix $m\in \mathbb{N}_+$. Consider $\Gamma_0^m$ and the product of normalized Lebesgue measures on it. If a property about tuples $(v_1, \cdots, v_m)\in \Gamma_0^m$ is true outside a zero-measure set, then we say it holds for generic tuples.
\end{definition}

\begin{proposition} \label{SectionPropositionConditionsGeneric}
Let $n, d, m\in \mathbb{N}_+, n\geq 3$ such that $\frac{m}{d}<\frac{n}{2}$. Then for generic tuples $(v_1, \cdots, v_m)\in \Gamma_0^m$, the linear space
\begin{equation*}
\Gamma:=\spn\{v_1, \cdots, v_m\} 
\end{equation*}
has dimension $m$ and satisfies Type \uppercase\expandafter{\romannumeral1} non-degenerate condition. Moreover, assume $\frac{m}{d}<\frac{n}{2}\leq \lceil\frac{m}{d}\rceil$. When $\frac{m}{d}\leq \frac{n-1}{2}$, $\Gamma$ also satisfies Type \uppercase\expandafter{\romannumeral2} non-degenerate condition for generic tuples $(v_1, \cdots, v_m)\in \Gamma_0^m$. When $\frac{n-1}{2}<\frac{m}{d}$ and $n>10d^3$, $\Gamma$ also satisfies Type \uppercase\expandafter{\romannumeral2} non-degenerate condition for generic tuples $(v_1, \cdots, v_m)\in \Gamma_0^m$.
\end{proposition} 
We will prove Proposition \ref{SectionPropositionConditionsGeneric} later. Now we state our main theorem.
\begin{theorem}  \label{MainTheorem1}
Let $n, d, m\in \mathbb{N}_+, n\geq 3$ such that $\frac{m}{d}<\frac{n}{2}$ and $1<p_1, \cdots, p_{n-1}\leq \infty, 0<q<\infty$ such that $\frac{1}{q}=\frac{1}{p_1}+\cdots+\frac{1}{p_{n-1}}$. Denote $q^{'}$ by $p_n$. Let $\Gamma\subset \Gamma_0$ be an $m\text{-dimensional}$ linear subspace and $\mathfrak{m}: \Gamma_0\rightarrow \mathbb{C}$ be a function satisfying \eqref{Section1MultiplierDecay1} for some large $N_2$ depending on $n, d, m$ and $p_j, j\in [n-1]$. Consider the estimate
\begin{equation} \label{MainEstimate}
\|T_{\mathfrak{m}}(f_1, \cdots, f_{n-1})\|_q \lesssim \prod_{j=1}^{n-1} \|f_j\|_{p_j},\ \forall f_1, \cdots, f_n\in \mathcal{S}(\mathbb{R}^d).
\end{equation}
\begin{enumerate}
\item When $\lceil \frac{m}{d} \rceil<\frac{n}{2}$, assume $\Gamma$ satisfies Type \uppercase\expandafter{\romannumeral1} non-degenerate condition. Then \eqref{MainEstimate} holds whenever $(p_1, \cdots, p_{n-1}, q)$ satisfies
\begin{equation} \label{Section1RangeOfExponents}
\begin{split}
\sum_{j\in \Upsilon} \frac{1}{p_j}<n-2\alpha+\frac{1}{2},\ \forall \Upsilon\subset[n],\ \#\Upsilon=n-2\alpha+1, \\
\sum_{j\in \Upsilon} \frac{1}{p_j}<n-2\alpha+1,\ \forall \Upsilon\subset[n],\ \#\Upsilon=n-1,
\end{split}
\end{equation}
with $\alpha:=\lceil \frac{m}{d} \rceil$.
\item When $\frac{n}{2}\leq \lceil \frac{m}{d} \rceil$ and $\frac{m}{d}\leq \frac{n-1}{2}$, assume $\Gamma$ satisfies Type \uppercase\expandafter{\romannumeral1} and Type \uppercase\expandafter{\romannumeral2} non-degenerate conditions. Then \eqref{MainEstimate} holds whenever $(p_1, \cdots, p_{n-1}, q)$ satisfies \eqref{Section1RangeOfExponents} with $\alpha=\frac{n-1}{2}$.
\item When $\frac{n}{2}\leq \lceil \frac{m}{d} \rceil$ and $\frac{n-1}{2}<\frac{m}{d}$, assume $\Gamma$ satisfies Type \uppercase\expandafter{\romannumeral1} and Type \uppercase\expandafter{\romannumeral2} non-degenerate conditions. Then \eqref{MainEstimate} holds whenever $(p_1, \cdots, p_{n-1}, q)$ satisfies
\begin{align*}
\sum_{j\in \Upsilon} \frac{1}{p_j}<\frac{n+1}{2}-\alpha,\ \forall \Upsilon\subset[n],\ \#\Upsilon=1, \\
\sum_{j\in \Upsilon} \frac{1}{p_j}<\frac{n+2}{2}-\alpha,\ \forall \Upsilon\subset[n],\ \#\Upsilon=2, \\
\sum_{j\in \Upsilon} \frac{1}{p_j}<n-2\alpha+1,\ \forall \Upsilon\subset[n],\ \#\Upsilon=n-1,
\end{align*}
with $\alpha:=\frac{m}{d}$.
\end{enumerate}
\end{theorem}
\begin{remark}
In the first two cases, we get boundedness in the range including $1<p_j\leq \infty, j\in [n-1], 1<q<\infty$. The first one is obtained in \cite{MTTMultilinearSingularMultipliers}. In fact, they proved boundedness in a larger range of $q$, but the proof given here does not rely on induction. The second one extends the local $L^2$ range ($2<p_j<\infty, 1<q<2$) in \cite{DPTFractionalRank}. Bounedness of the third case is new.
\end{remark}

One can also consider multilinear operators in the kernel formulation. Let $\mathfrak{L}_1, \cdots, \mathfrak{L}_{n-1}: \mathbb{R}^{d(n-1)-m}\rightarrow \mathbb{R}^d$ be linear operators such that
\begin{equation} \label{Section1EquationKernelRankCondition}
\rank{
\begin{pmatrix}
\mathfrak{L}_1 \\
\vdots \\
\mathfrak{L}_{n-1} 
\end{pmatrix} }=d(n-1)-m.
\end{equation}
Assume $K: \mathbb{R}^{d(n-1)-m}\rightarrow \mathbb{C}$ is a C-Z kernel in some suitable sense. Then we can write
\begin{equation}  \label{Section1EquationKernelRepresentation}
\int_{\mathbb{R}^{d(n-1)-m}} \prod_{j=1}^{n-1} f_j(x+\mathfrak{L}_j t)K(t) dt    
\end{equation}
into the form of $T_{\mathfrak{m}} (f_1, \cdots, f_{n-1})(x)$ with
\begin{equation} \label{SectionEquationDefineMultiplierfromKernel}
\mathfrak{m}(\xi_1, \cdots, \xi_n):=\widehat{K}\big(-(\mathfrak{L}_1^{*}\xi_1+\cdots+\mathfrak{L}_{n-1}^{*}\xi_{n-1})\big).
\end{equation}
Here
\begin{equation*}
\Gamma:=\{(\xi_1, \cdots, \xi_n)\in \Gamma_0: \mathfrak{L}_1^{*}\xi_1+\cdots+\mathfrak{L}_{n-1}^{*}\xi_{n-1}=0\}    
\end{equation*}
has dimension $m$. However, not all $T_{\mathfrak{m}}$'s have such a kernel representation.

Now we come to examples of $T_{\mathfrak{m}}$ in the kernel representation \eqref{Section1EquationKernelRepresentation}. For simplicity, let $K(t):=\frac{t_1}{|t|^{d(n-1)-m+1}}$ be the Riesz kernel in the following examples.
\begin{example} \label{Section1ExampleBHT}
Let $n=3, d=1, m=1$. Take $\mathfrak{L}_1, \mathfrak{L}_2: \mathbb{R}\rightarrow \mathbb{R}$ such that $\mathfrak{L}_1 t=-t, \mathfrak{L}_2 t=t$. Then
\begin{equation*}
T_{\mathfrak{m}} (f_1, f_2)(x):=\int_{\mathbb{R}} f_1(x-t)f_2(x+t)K(t) dt.
\end{equation*}
It belongs to the first case of Theorem \ref{MainTheorem1}. Since
\begin{equation*}
\Gamma=\{(\tau, \tau, -2\tau): \tau\in \mathbb{R}\}
\end{equation*}
satisfies Type \uppercase\expandafter{\romannumeral1} non-degenerate condition, we get it is bounded in a certain range. In fact, $T_{\mathfrak{m}}$ is the bilinear Hilbert transform handled in \cite{LTBHT1}\cite{LTBHT2}.
\end{example}

\begin{example} \label{Section1ExampleFractionalRankNon-Degenerate}
Let $n=4, d=2, m=3$. Take $\mathfrak{L}_1, \mathfrak{L}_2, \mathfrak{L}_3: \mathbb{R}^3\rightarrow \mathbb{R}^2$ such that $\mathfrak{L}_1 t=(t_2, t_3), \mathfrak{L}_2 t=(t_3, 2t_1), \mathfrak{L}_3 t=(t_1, 3t_2)$. Then
\begin{equation*}
T_{\mathfrak{m}} (f_1, f_2)(x):=\int_{\mathbb{R}^3} f_1 \big(x+(t_2, t_3) \big) f_2 \big(x+(t_3, 2t_1) \big) f_3 \big(x+(t_1, 3t_2) \big)K(t) dt.
\end{equation*}
It belongs to the second case of Theorem \ref{MainTheorem1}. Since
\begin{equation*}
\Gamma=\{ \big( (\tau, \tau^{'}), (-\tau^{'}, \tau^{''}), (-2\tau^{''}, -\frac{\tau}{3}), (\tau^{'}+2\tau^{''}-\tau, \frac{\tau}{3}-\tau^{'}-\tau^{''})\big): \tau, \tau^{'}, \tau^{''}\in \mathbb{R}\}
\end{equation*}
satisfies Type \uppercase\expandafter{\romannumeral1} and Type \uppercase\expandafter{\romannumeral2} non-degenerate conditions, we get it is bounded in a certain range. The boundedness of $T_{\mathfrak{m}}$ in the range $2<p_j<\infty, 1<q<2$ was first established in \cite{DPTFractionalRank}.
\end{example}

\begin{example} \label{Section1ExampleFractionalRankDegenerateMildly}
Let $n=4, d=2, m=3$. Take $\mathfrak{L}_1, \mathfrak{L}_2, \mathfrak{L}_3: \mathbb{R}^3\rightarrow \mathbb{R}^2$ such that $\mathfrak{L}_1 t=(t_2, t_3), \mathfrak{L}_2 t=(t_3, t_1), \mathfrak{L}_3 t=(t_1, t_2)$. Then
\begin{equation*}
T_{\mathfrak{m}} (f_1, f_2)(x):=\int_{\mathbb{R}^3} f_1 \big(x+(t_2, t_3) \big) f_2 \big(x+(t_3, t_1) \big) f_3 \big(x+(t_1, t_2) \big)K(t) dt.
\end{equation*}
It does not fit the second case of Theorem \ref{MainTheorem1} because the singular space
\begin{equation}  \label{Section1EquationMildlyDegenerateSingularSpace}
\Gamma=\{ \big( (\tau, \tau^{'}), (-\tau^{'}, \tau^{''}), (-\tau^{''}, -\tau), (\tau^{'}+\tau^{''}-\tau, \tau-\tau^{'}-\tau^{''})\big): \tau, \tau^{'}, \tau^{''}\in \mathbb{R}\}
\end{equation}
does not satisfy Type \uppercase\expandafter{\romannumeral2} non-degenerate condition. However, the degeneracy is mild, and our method can obtain its boundedness in a smaller range. See Section \ref{Section7} for more details.
\end{example}

\begin{example} \label{Section1ExampleVeryLargeRankTHT}
Let $n=3, d=2, m=3$. Take $\mathfrak{L}_1, \mathfrak{L}_2: \mathbb{R}\rightarrow \mathbb{R}^2$ such that $\mathfrak{L}_1 t=(0, t), \mathfrak{L}_2 t=(t, 0)$. Then
\begin{equation*}
T_{\mathfrak{m}} (f_1, f_2)(x):=\int_{\mathbb{R}} f_1 \big(x+(0, t) \big)f_2 \big(x+(t, 0) \big)K(t) dt.
\end{equation*}
This is the famous triangular Hilbert transform. It has rank $\frac{m}{d}=\frac{3}{2}=\frac{n}{2}$ and have a large group of symmetry:
\begin{equation*}
|T_{\mathfrak{m}} (e^{i g_1(x_1)}f_1, e^{ig_2(x_2)}f_2)|=|T_{\mathfrak{m}} (f_1, f_2)|.
\end{equation*}
for any real-valued functions $g_1, g_2$. No bound is known for this operator. Another famous example with rank $\frac{n}{2}$ is the trilinear Hilbert transform \cite{TaoMultilinearHilbertCancellation}. See \cite{TaoMultilinearHilbertCancellation}\cite{Zorin-KranichSimplexHilbertCancellation}\cite{DKTSimplexHilbertPowerCancellation}\cite{KTZKTriangularHilbertDyadicPartialResult} for some progress tackling these operators.
\end{example}

\begin{example} \label{Section1ExampleFractionalRankSmalln}
Let $n=3, d=3, m=4$. Take $\mathfrak{L}_1, \mathfrak{L}_2: \mathbb{R}^2\rightarrow \mathbb{R}^3$ such that $\mathfrak{L}_1 t=(0, t), \mathfrak{L}_2 t=(t, 0)$. Then
\begin{equation*}
T_{\mathfrak{m}} (f_1, f_2)(x):=\int_{\mathbb{R}^2} f_1 \big(x+(0, t) \big)f_2 \big(x+(t, 0) \big)K(t) dt.
\end{equation*}
This corresponds to the block matrix $\mathbf{T}_2$ in \cite{BDLClassificationOfTrilinear}. The relation $\frac{3-1}{2}<\frac{4}{3}$ fits the third case of Theorem \ref{MainTheorem1}. However, $n=3$ is too small. We can see $l=1, r=1$ and $B^{(k)}=\emptyset$ ($l+1-(d-r)=0$) in Definition \ref{Section1Non-DegenerateCondition2}. Thus $(v, \cdots, v)\in \ker(\mathfrak{L}), \forall v\in \Gamma$, and Type \uppercase\expandafter{\romannumeral2} non-degenerate condition fails. In fact, $T_{\mathfrak{m}}$ has a large group of symmetry as Example \ref{Section1ExampleVeryLargeRankTHT}, so the techniques in this paper are not able to handle it. See the last remark in Section \ref{Section8} for more discussions.
\end{example}

The proof of Theorem \ref{MainTheorem1} builds on the classical framework \cite{MTTMultilinearSingularMultipliers}\cite{DPTFractionalRank}\cite{MS2}: First use Whitney decomposition to exploit the decay condition \eqref{Section1MultiplierDecay1}, tensorize $\mathfrak{m}$, and discretize the operator. Then consider subcollections called trees to apply C-Z type arguments. After this, we gather trees with comparable densities together using orthogonality. Finally sum the geometric series of essentially different densities. In the second and third cases of our theorems, the classical method for gathering trees with comparable densities is inefficient, making the geometric series divergent. The focus of \cite{DPTFractionalRank} and our paper is to count trees more efficiently. One also needs to adjust the classical definition of trees to make the new counting strategy work. To tackle the large rank case, we construct the new non-degenerate condition in Definition \ref{Section1Non-DegenerateCondition2}. We also modify and develop other parts of the argument in \cite{DPTFractionalRank} to go outside the local $L^2$ range.

\paragraph{Outline of the Paper.}
\begin{itemize}
\item In Section \ref{Section2}, we show Proposition \ref{SectionPropositionConditionsGeneric}.
\item In Section \ref{Section3}, we do the discretization step following the classical framework.
\item In Section \ref{Section4}, we define trees and prove related propositions. Some parts of this section are classical, but we include them for completeness.
\item In Section \ref{Section5}, we develop the more efficient tree counting strategy.
\item In Section \ref{Section6}, we prove Theorem \ref{MainTheorem1}.
\item In Section \ref{Section7}, we show how to adapt our method to tackle the mildly degenerate operator in Example \ref{Section1ExampleFractionalRankDegenerateMildly}.
\item In Section \ref{Section8}, we give some remarks including the simplification of the proof of Theorem \ref{MainTheorem1} for $(p_1, \cdots, p_{n-1}, q)\in (1, \infty)^n$.
\end{itemize}
\paragraph{Notations and Conventions.}
\begin{itemize}
\item Write $[n]:=\{1, \cdots, n\}$. In Section \ref{Section2}, we also use the notation $[n_1, n_2]:=\{n_1, \cdots, n_2\}$ for $n_1, n_2\in \mathbb{Z}, n_1\leq n_2$.
\item For a matrix $A=(a_{j, k})_{j, k\in [n]}$, we write
\begin{equation*}
\diag (A):=\prod_{j=1}^n a_{j, j}.
\end{equation*}
\item For each $j\in [n]$, let $I_j\subset\mathbb{R}^d$ denote dyadic cubes in the physical space.
\item For each $j\in [n]$, we use $\Xi_j$ for shifted dyadic cubes in the frequency space. Here a shifted dyadic cube has the form $2^{l}\big([0, 1)^d+z+\alpha \big), l\in \mathbb{Z}, z\in \mathbb{Z}^d, \alpha\in \{0, \frac{1}{3}, \frac{2}{3}\}^d$.
\item Assume $Q$ is a cube. Let $\mathfrak{c}(Q)$ denote its center. The side length of $Q$ is sometimes called its scale and denoted by $s(Q)$. For $C>0$, let $CQ$ denote the cube with center $\mathfrak{c}(Q)$ and side length $Cs(Q)$. Denote the projection of $Q$ onto the $k\text{-th}$ coordinate by $Q_k$.
\item Let $I_j$ be one dyadic cube in the physical space and $\Xi_j$ be one shifted dyadic cube in the frequency space. Assume $s(I_j)\cdot s(\Xi_j)=1$. We call the product $I_j\times \Xi_j$ a rectangle and denote it by $\mathcal{R}_j$. Use $I(\mathcal{R}_j), \Xi(\mathcal{R}_j)$ to denote the physical and frequency cubes of a given rectangle $R_j$. We also write $\mathcal{R}:=(\mathcal{R}_1, \cdots, \mathcal{R}_n)$ if $I(\mathcal{R}_1)=\cdots=I(\mathcal{R}_n)$. In this case, call $\mathcal{R}$ a vector rectangle and let $I(\mathcal{R}):=I(\mathcal{R}_1), \Xi(\mathcal{R}):=\Xi(\mathcal{R}_1)\times \cdots \times \Xi(\mathcal{R}_n)$. Sometimes we may use the letter $\mathcal{P}$ instead of $\mathcal{R}$.  
\item There will be three large but fixed constants: $1\ll C_1\ll C_2\ll C_3$. Here $C_1$ is for Whitney decomposition, $C_2$ is for dilation of frequency cubes when defining trees, and $C_3$ is for sparse conditions. We will also use the large but fixed integer constants $N_{\epsilon}\ll N_1\ll N_2$ for rapid decay, which are irrelevant to $C_1, C_2, C_3$. Here $\epsilon>0$ is a small constant appearing in Section \ref{Section6}, and we choose the constant $N_{\epsilon}$ according to the value of $\epsilon$.
\item For $X, Y>0$, the notation $X\sim Y$ means there exists a (not important) constant $C>0$ such that $\frac{1}{C}Y\leq X\leq CY$.
\end{itemize}
\paragraph{Acknowledgement.}
The author would like to thank Prof. Christoph Thiele for suggesting the topic, giving valuable advice, and reading the draft of this paper. The author also appreciates Dr. Rajula Srivastava's help during the author's study and application.


\section{About the Non-Degenerate Conditions} \label{Section2}
The proof of Proposition \ref{SectionPropositionConditionsGeneric} relies on the following lemma.
\begin{lemma}  \label{Section2LemmaPolynomialVarietyisGeneric}
Let $P(x_1, \cdots, x_n)$ be a non-zero polynomial of $n$ variables. Then we have
\begin{equation*}
|\{(z_1, \cdots, z_n)\in \mathbb{R}^n: P(z_1, \cdots, z_n)=0\}|=0.
\end{equation*}
\end{lemma}
\begin{proof}
We prove by induction on $n$. If $n=1$, the variety it finite. Assume the conclusion for any $n^{'}\leq n-1$ is true, we show it is also true for $n$. Write
\begin{equation*}
P(x_1, \cdots, x_n)=\sum_{l=1}^{L} P_l(x_1, \cdots, x_{n-1})x_n^{l},
\end{equation*}
where $P_L(x_1, \cdots, x_{n-1})$ is not identically zero. Then
\begin{equation*}
\{P=0\}\subset \{P_L=0\}\cup \{P_L\neq 0, P=0\}.    
\end{equation*}
Now apply Fubini. For the first term on the right, integrate $x_1, \cdots, x_{n-1}$ before $x_n$. For the second term on the right, integrate $x_n$ before $x_1, \cdots, x_{n-1}$. We get both sets have measure zero by the induction hypothesis.
\end{proof}
Recall the normalized Lebesgue measure $\mu$ on $\Gamma_0$ satisfies
\begin{equation} \label{Section2EquationBasicPropertyOfMeasureOnGamma0}
\mu(E)=\int_{\mathbb{R}^{d(n-1)}} 1_{P_{[j]^c\times [d]}E}\ dx_1\cdots dx_{j-1}dx_{j+1}\cdots dx_{n}
\end{equation}
for any $E\subset \Gamma_0$ and $j\in [n]$.

We first show a generic tuple $(v_1, \cdots, v_m)$ makes $\Gamma$ have dimension $m$. Note that $\frac{m}{d}<\frac{n}{2}\leq n-1$, so we can take $\Theta\subset [n-1]\times[d]$ such that $\#\Theta=m$. Fix $\Theta$. It suffices to show
\begin{equation*}
\mu^m\Big( \{(v_1, \cdots, v_m)\in \Gamma_0^m: \det\big( (v_1)_{\Theta}, \cdots, (v_m)_{\Theta} \big)=0\}\Big)=0.
\end{equation*}
Use \eqref{Section2EquationBasicPropertyOfMeasureOnGamma0} to expand LHS into a integral. Applying Fubini and Lemma \ref{Section2LemmaPolynomialVarietyisGeneric} gives the result.

Then we prove Type \uppercase\expandafter{\romannumeral1} non-degenerate condition holds for generic tuples. Fix $\Theta\subset A\times [d]$ for some $\#A=\lceil \frac{m}{d} \rceil$. We only need to prove $\Pi_{\Theta}$ restricted on $\Gamma$ has a non-trivial kernel only when $(v_1, \cdots, v_m)$ is in a zero-measure set. This is equivalent to $\det\big( (v_1)_{\Theta}, \cdots, (v_m)_{\Theta} \big)=0$, so the previous argument implies $\big{\{}\{(v_1, \cdots, v_m)\in \Gamma_0^m: \ker (P_\Theta)\neq \{0\} \big{\}}$ has measure zero.

Finally, we prove Type \uppercase\expandafter{\romannumeral2} non-degenerate condition also holds for generic tuples if $\frac{m}{d}<\frac{n}{2}\leq \lceil\frac{m}{d}\rceil$. Let's tackle the $\frac{m}{d}\leq \frac{n-1}{2}$ case first. Recall $\mathfrak{L}$ is defined in \eqref{Section1EquationTheCrucialMap}. It suffices to show the set of tuples making $\ker (\mathfrak{L})\neq \{0\}$ has measure zero for any fixed choice of $A, B^{(1)}, B^{(2)}$. By permutation of the indices if necessary, we can further assume $A=\{2n^{'}+1\}$ and $B^{(1)}=[1, n^{'}], B^{(2)}=[n^{'}+1, 2n^{'}]$. Consider the system of equations for $(w^{(1)}, w^{(2)})\in \Gamma^2$:
\begin{equation}  \label{Section2EquationType2Non-Degenerate1}
\begin{cases}
P_{B^{(1)}}w^{(1)}=0, \\
P_{B^{(2)}}w^{(2)}=0, \\
P_{A}(w^{(1)}-w^{(2)})=0.
\end{cases}
\end{equation}
Then $\ker (\mathfrak{L})\neq \{0\}$ is equivalent to the statement that \eqref{Section2EquationType2Non-Degenerate1} has a non-trivial solution. Write
\begin{equation}  \label{Section2EquationExpressw}
w^{(\cdot)}=\sum_{i=1}^m t^{(\cdot)}_i v_i,
\end{equation}
and plug it into \eqref{Section2EquationType2Non-Degenerate1}. Then we get a system of equations about $(t_1^{(1)}, \cdots, t_m^{(1)}, t_1^{(2)}, \cdots, t_m^{(2)})\in \mathbb{R}^{2m}$. The coefficient matrix is
\begin{equation}  \label{Section2EquationType2Non-Degenerate1Matrix}
\left( \begin{array}{c|c|c|c|c|c}
(v_1)_{[n^{'}]\times [d]} & \cdots & (v_m)_{[n^{'}]\times [d]} & 0 & \cdots & 0 \\
\hline
0 & \cdots & 0 & (v_1)_{[n^{'}+1, 2n^{'}]\times [d]} & \cdots & (v_m)_{[n^{'}+1, 2n^{'}]\times [d]}  \\
\hline
(v_1)_{\{2n^{'}+1\}\times [d]} & \cdots & (v_m)_{\{2n^{'}+1\}\times [d]} & -(v_1)_{\{2n^{'}+1\}\times [d]} & \cdots & -(v_m)_{\{2n^{'}+1\}\times [d]}
\end{array} \right).
\end{equation}
Here each small block on the first two big rows is of $dn^{'}\times 1$, and each small block on the third big row is of $d\times 1$. It suffices to show this $(d(2n^{'})+1)\times 2m$ matrix has rank less than $2m$ only in a zero-measure set of $(v_1, \cdots, v_m)$. Note that $\frac{m}{d}\leq \frac{n-1}{2}$ implies $2m\leq d(n-1)=d(2n^{'}+1)$. Write $m=ld+r$. Then $l=n^{'}$ and $2r\leq d$. Delete the last $d-2r$ rows of \eqref{Section2EquationType2Non-Degenerate1Matrix}. The remaining matrix has the form:
\begin{equation*}
F:=\left( \begin{array}{c|c}
X & 0 \\
\hline
0 & Y \\
\hline
Z & -Z
\end{array} \right).
\end{equation*}
Here each of $X, Y, Z$ has distinct entries, and any two of them have disjoint sets of entries. Now use \eqref{Section2EquationBasicPropertyOfMeasureOnGamma0} to calculate. Note that all involved variables $v_{j, k}$'s satisfy $j\leq n-1$. It suffices to prove $\det (F)$ is a non-zero polynomial by Lemma \ref{Section2LemmaPolynomialVarietyisGeneric}.

Write
\begin{equation*}
X=(X^{'}, X^{''}),\ Y=(Y^{''}, Y^{'}),
\end{equation*}
where $X^{'}, Y^{'}$ are of $dn^{'}\times dn^{'}$, and $X^{''}, Y^{''}$ are of $dn^{'}\times r$. In the expansion of $\det (F)$, consider those monomials containing
\begin{equation*}
\diag (X^{'})\cdot \diag (Y^{'}).
\end{equation*}
To get such terms, one must use the entries at the positions $(1, 1), (2, 2), \cdots, (dn^{'}, dn^{'}); (dn^{'}+1, dn^{'}+2r+1), \cdots, (2dn^{'}, dn^{'}+2r+dn^{'})$. It then suffices to show the determinant of the remaining matrix is a non-zero polynomial because this implies $\det(F)$ has certain non-zero monomials. The remaining matrix consists of the last $r$ columns of $Z$ and the first $r$ columns of $-Z$. But $2r\leq 2(d-1)\leq m$ implies all its entries are distinct, so its determinant is a non-zero polynomial.

Now we come to the more complicated case: $\frac{n-1}{2}<\frac{m}{d}$ and $n>10d^3$. The idea is similar to before. We can fix $A=[(2d-1)a+b+1, (2d-1)a+b+(d-r)], U_1=[1, a], \cdots, U_{2d-1}=[(2d-2)a+1, (2d-1)a], W=[(2d-1)a+1, (2d-1)a+b]$. We prove the system equations for $(w^{(1)}, \cdots, w^{(d)})\in \Gamma^d$:
\begin{equation}  \label{Section2EquationType2Non-Degenerate2}
\begin{cases}
P_{B^{(k)}} w^{(k)}=0,\ \forall 1\leq k\leq d, \\
P_A(w^{(1)}- w^{(k)})=0,\ \forall 2\leq k\leq d.
\end{cases}
\end{equation}
has a non-trivial solution only in a zero-measure set of $(v_1, \cdots, v_m)\in \Gamma_0^m$. Plug \eqref{Section2EquationExpressw} into \eqref{Section2EquationType2Non-Degenerate2} like above. Then we get a system of equations about $(t_1^{(1)}, \cdots, t_m^{(1)}, \cdots, t_1^{(d)}, \cdots, t_m^{(d)})\in \mathbb{R}^{dm}$. The coefficient matrix has the form:
\begin{equation}  \label{Section2EquationType2Non-Degenerate2Matrix}
G:=\left( \begin{array}{c|c|c|c|c}
X_1 & 0 & 0 & \cdots & 0 \\
\hline
X_2 & 0 & 0 & \cdots & 0  \\
\hline
\cdots & \cdots & \cdots & \cdots & \cdots  \\
\hline
X_d & 0 & 0 & \cdots & 0  \\
\hline
0 & X_2 & 0 & \cdots & 0  \\
\hline
0 & X_3 & 0 & \cdots & 0  \\
\hline
\cdots & \cdots & \cdots & \cdots & \cdots  \\
\hline
0 & X_{d+1} & 0 & \cdots & 0  \\
\hline
\cdots & \cdots & \cdots & \cdots & \cdots  \\
\hline
0 & 0 & 0 & \cdots & X_d \\
\hline
0 & 0 & 0 & \cdots & X_{d+1} \\
\hline
\cdots & \cdots & \cdots & \cdots & \cdots  \\
\hline
0 & 0 & 0 & \cdots & X_{2d-1} \\
\hline
Y & 0 & 0 & \cdots & 0  \\
\hline
0 & Y & 0 & \cdots & 0  \\
\hline
\cdots & \cdots & \cdots & \cdots & \cdots  \\
\hline
0 & 0 & 0 & \cdots & Y \\
\hline
Z & -Z & 0 & \cdots & 0  \\
\hline
Z & 0 & -Z & \cdots & 0  \\
\hline
\cdots & \cdots & \cdots & \cdots & \cdots  \\
\hline
Z & 0 & 0 & \cdots & -Z 
\end{array} \right).
\end{equation}
Here each of $X_1, \cdots, X_{2d-1}$ is of $ad\times m$, and $Y, Z$ are of $bd\times m, (d-r)d\times m$ respectively. There are $ad\cdot d\cdot d+bd\cdot d+(d-r)d\cdot(d-1)=md$ rows and $md$ columns. Moreover, each of $X_1, \cdots, X_{2d-1}, Y, Z$ has distinct entries, and any two of them have disjoint sets of entries. Note that all involved variables $v_{j, k}$'s satisfy $j\leq (2d-1)a+b+(d-r)\leq n-1$. It suffices to prove $\det (G)$ is a non-zero polynomial.

In the expansion of $\det (G)$, we shall require the monomials to contain a series of entries and delete certain rows and columns as above. There are $d$ big columns in \eqref{Section2EquationType2Non-Degenerate2Matrix}, each of which consists of $m$ successive columns. In each big column, we number the $m$ columns using $1, 2, \cdots, m$. During the deleting procedures below, we use $\Upsilon_k$ to record which columns are left in the $k\text{-th}$ big column. Thus each $\Upsilon_k=[1, m]$ at the beginning.
\begin{enumerate}
\item Write each $Y$ into
\begin{equation*}
(Y^{'}, Y^{''}),
\end{equation*}
where $Y^{'}, Y^{''}$ have $bd, m-bd$ columns respectively. Consider those monomials in $\det(G)$ containing 
\begin{equation*}
\diag(Y^{'})^d.    
\end{equation*}
After generating it, there is no $Y$'s block left in \eqref{Section2EquationType2Non-Degenerate2Matrix}, and each $\Upsilon_k$ becomes $[bd+1, m]$.
\item Then write $X_i, 1\leq i\leq d$ into
\begin{equation*}
(X_i^{'}, X_i^{''}, X_i^{'''}),
\end{equation*}
where $X_i^{'}, X_i^{''}, X_i^{'''}$ have $bd, ad, m-bd-ad$ columns respectively. The previous step deleted $X_i^{'}, 1\leq i\leq d$. Further require the monomials to contain
\begin{equation*}
\prod_{i=1}^d \diag(X_i^{''}).    
\end{equation*}
Note that $\diag(X_1^{''})$ can only come from the first big column. After generating $\diag(X_1^{''})$, $\diag(X_2^{''})$ can only come from the second big column. The pattern works for $\diag(X_i^{''}), i=3, \cdots, d$. In this way, we delete the rows in $[1, ad], [ad^2+1, ad^2+ad], \cdots, [(d-1)ad^2+1, (d-1)ad^2+ad]$, and each $\Upsilon_k$ becomes $[bd+ad+1, m]$. The remaining matrix has the form (We omit $Z$'s blocks.):
\begin{equation*}
\left( \begin{array}{c|c|c|c|c}
X_2^{'''} & 0 & 0 & \cdots & 0 \\
\hline
X_3^{'''} & 0 & 0 & \cdots & 0  \\
\hline
\cdots & \cdots & \cdots & \cdots & \cdots  \\
\hline
X_d^{'''} & 0 & 0 & \cdots & 0  \\
\hline
0 & X_3^{'''} & 0 & \cdots & 0  \\
\hline
0 & X_4^{'''} & 0 & \cdots & 0  \\
\hline
\cdots & \cdots & \cdots & \cdots & \cdots  \\
\hline
0 & X_{d+1}^{'''} & 0 & \cdots & 0  \\
\hline
\cdots & \cdots & \cdots & \cdots & \cdots  \\
\hline
0 & 0 & 0 & \cdots & X_{d+1}^{'''} \\
\hline
0 & 0 & 0 & \cdots & X_{d+2}^{'''} \\
\hline
\cdots & \cdots & \cdots & \cdots & \cdots  \\
\hline
0 & 0 & 0 & \cdots & X_{2d-1}^{'''} \\
\hline
\cdots & \cdots & \cdots & \cdots & \cdots 
\end{array} \right).
\end{equation*}
\item Now we find that $X_2^{'''}, \cdots, X_{d+1}^{'''}$ are at the same position as $X_1, \cdots, X_d$ in the previous step. Do the previous procedure again to $X_2^{'''}, \cdots, X_{d+1}^{'''}$: Write $X_i^{'''}, 2\leq i\leq d+1$ into
\begin{equation*}
(X_i^{''''}, X_i^{'''''}),
\end{equation*}
where $X_i^{''''}, X_i^{'''''}$ have $ad, m-bd-2ad$ columns respectively. Require the monomials to contain
\begin{equation*}
\prod_{i=2}^{d+1} \diag(X_i^{''''}),  
\end{equation*}
so that each $\Upsilon_k$ becomes $[bd+2ad+1, m]$. After this, we see $X_3^{'''''}, \cdots, X_{d+2}^{'''''}$ are at the same position as $X_1, \cdots, X_d$ in the previous step. Iterate this argument $d-1$ times in total. Our $\Upsilon_k$ finally becomes $[bd+(d-1)ad+1, m]$. And the remaining matrix is:
\begin{equation*}
\left( \begin{array}{c|c|c|c|c}
X_d^{'\cdots'} & 0 & \cdots & 0 \\
\hline
0 & X_{d+1}^{'\cdots'} & 0 & \cdots & 0  \\
\hline
\cdots & \cdots & \cdots & \cdots & \cdots  \\
\hline
0 & 0 & 0 & \cdots & X_{2d-1}^{'\cdots'}  \\
\hline
\cdots & \cdots & \cdots & \cdots & \cdots  
\end{array} \right).
\end{equation*}
\item Now we do a similar removal. This time each $X_i^{'\cdots'}, d\leq i\leq 2d-1$ only appear once, so we can remove any $V_k \subset [bd+(d-1)ad+1, m], \#V_k=ad, 1\leq k\leq d$ using the above argument. Since $n>10d^3$, we have $m-bd-(d-1)ad=ad+(d-1)(d-r)>d(d-1)(d-r)=d(m-bd-ad^2)$. Thus we can make the new $\Upsilon_k:=[bd+(d-1)ad+1, m]\backslash V_k, 1\leq k\leq d$ pairwise disjoint. Only $Z$'s blocks in \eqref{Section2EquationType2Non-Degenerate2} are left, and the remaining matrix is:
\begin{equation}  \label{Section2EquationType2Non-Degenerate2LastStep}
\left( \begin{array}{c|c|c|c|c}
Z_{\Upsilon_1} & -Z_{\Upsilon_2} & 0 & \cdots & 0 \\
\hline
Z_{\Upsilon_1} & 0 & -Z_{\Upsilon_3} & \cdots & 0  \\
\hline
\cdots & \cdots & \cdots & \cdots & \cdots  \\
\hline
Z_{\Upsilon_1} & 0 & 0 & \cdots & Z_{\Upsilon_d}  \\
\end{array} \right).
\end{equation}
Here each of $Z_{\Upsilon_1}, \cdots, Z_{\Upsilon_d}$ is of $d(d-r)\times (d-1)(d-r)$. Moreover, each of them has distinct entries, and any two of them have disjoint sets of entries.
\item At last, there are $d-1$ big rows in \eqref{Section2EquationType2Non-Degenerate2LastStep} numbered by $2\leq j\leq d$. We will remove $(d-1)(d-r)$ rows in each of them. For the $j\text{-th}$ big row, we number its $d(d-r)$ rows using $1, 2, \cdots, d(d-r)$ and use $\widetilde{\Upsilon}_j$ to record which rows are left in it. Like before, each $Z_{\Upsilon_2}, \cdots, Z_{\Upsilon_d}$ only appear once. Hence we can remove any $\widetilde{V}_j \subset[1, d(d-r)], \#\widetilde{V}_j=(d-1)(d-r), 2\leq j\leq d$ by requiring suitable entries of $Z_{\Upsilon_2}, \cdots, Z_{\Upsilon_d}$. Making $\widetilde{\Upsilon}_j:=[1, d(d-r)]\backslash \widetilde{V}_j, 2\leq j\leq d$ pairwise disjoint, we obtain the matrix:
\begin{equation*}
\left( \begin{array}{c}
(Z_{\Upsilon_1})_{\widetilde{\Upsilon}_2} \\
(Z_{\Upsilon_1})_{\widetilde{\Upsilon}_3} \\
\cdots \\
(Z_{\Upsilon_1})_{\widetilde{\Upsilon}_d} \\
\end{array} \right)
\end{equation*}
with all entries distinct. Thus its determinant is a non-zero polynomial. This finishes the proof of Proposition \ref{SectionPropositionConditionsGeneric}.
\end{enumerate}


\section{Discretization} \label{Section3}
The discretization procedure, modified from \cite{DPTFractionalRank}\cite{DTTMultilinearMaximal}, consists of three steps.
\begin{description}
\item[Whitney decomposition and tensorization]
Fix a large constant $C_1$. Apply the Whitney decomposition to $\mathbb{R}^{d(n-1)}\backslash (P_{[n-1]}\Gamma)$. We get
\begin{equation*}
\mathbb{R}^{d(n-1)}\backslash (P_{[n-1]}\Gamma):=\bigcup Q^{'},
\end{equation*}
such that
\begin{equation}  \label{Section3EquationWhitneyConditionForn-1}
d(Q^{'}, P_{[n-1]}\Gamma)\sim C_1s(Q^{'}),\ \forall Q^{'}.
\end{equation}
Take a partition of unity $\{\chi_{Q^{'}}\}_{Q^{'}}$ associated with $\{2Q^{'}\}$ in the sense that $\chi_{Q^{'}}$ is supported in $2Q^{'}$ and
\begin{equation*}
|\partial^{\alpha} \chi_{Q^{'}}(\xi^{'})|\lesssim s(Q^{'})^{-|\alpha|}, \forall \alpha
\end{equation*}
holds for any $Q^{'}$. Regard $\mathfrak{m}$ as a function of $\xi^{'}$ in the natural way. Write
\begin{equation*}
\mathfrak{m}(\xi^{'}):=\sum_{Q^{'}} m(\xi^{'}) \chi_{Q^{'}}(\xi^{'}):=\sum_{Q^{'}} \mathfrak{m}_{Q^{'}}(\xi^{'}).
\end{equation*}
We can assume there are finite $Q^{'}$'s by the dominated convergence theorem and Fatou's lemma. Apply the Fourier expansion to each $\mathfrak{m}_{Q^{'}}$ by mapping $2Q^{'}$ to $[-\frac{1}{4}, \frac{1}{4})^{d(n-1)}$. Then add a bump function $\rho(\xi^{'}):=\rho_1(\xi_1)\cdots \rho_{n-1}(\xi_{n-1})$ such that each $\rho_j$ is $1$ on $[-\frac{1}{4}, \frac{1}{4}]^d$ and supported in $(-\frac{1}{2}, \frac{1}{2})^d$. We can write
\begin{equation*}
\mathfrak{m}_{Q^{'}}(\xi^{'})=\sum_{w\in \mathbb{Z}^{d(n-1)}} a_{Q^{'}, w} e^{2\pi i w\cdot \frac{\xi^{'}-\mathfrak{c}(Q^{'})}{4s(Q^{'})}} \rho \big(\frac{\xi^{'}-\mathfrak{c}(Q^{'})}{4s(Q^{'})} \big),\ \forall \xi^{'}\in \mathbb{R}^{d(n-1)}.
\end{equation*}
Here $a_{Q^{'}, w}$ satisfies
\begin{equation*}
|a_{Q^{'}, w}|\lesssim_N (1+|w|)^{-\frac{N_2}{2}}
\end{equation*}
uniformly in $Q^{'}$. Unfold $Q^{'}=Q_1\times \cdots \times Q_{n-1}, w=(w_1, \cdots, w_{n-1})$. We can then write $\sum_{Q^{'}} \mathfrak{m}_{Q^{'}}(\xi^{'})$ into
\begin{align*}
\sum_{Q^{'}} & \sum_{w\in \mathbb{Z}^{d(n-1)}} a_{Q^{'}, w} e^{2\pi i w\cdot \frac{\xi^{'}-\mathfrak{c}(Q^{'})}{4s(Q^{'})}} \prod_{j=1}^{n-1} \rho_{j}\big(\frac{\xi_j-\mathfrak{c}(Q_j)}{4s(Q^{'})} \big) \\
& =\sum_{w\in \mathbb{Z}^{d(n-1)}} (1+|w|)^{-100d(n-1)} \sum_{Q^{'}} \prod_{j=1}^{n-1} \Big( a_{Q^{'}, w}^{\frac{1}{n-1}}(1+|w|)^{100d} e^{2\pi i w_j\cdot \frac{\xi_j-\mathfrak{c}(Q_j)}{4s(Q^{'})}} \rho_{j}\big(\frac{\xi_j-\mathfrak{c}(Q_j)}{4s(Q^{'})} \big) \Big)  \\
& =:\sum_{w\in \mathbb{Z}^{d(n-1)}} (1+|w|)^{-100d(n-1)}\sum_{Q^{'}} \prod_{j=1}^{n-1} \varphi_{j}^{Q^{'}, w}\big(\frac{\xi_j-\mathfrak{c}(Q_j)}{4s(Q^{'})} \big),
\end{align*}
where $a_{Q^{'}, w}^{\frac{1}{n-1}}$ is an arbitrary complex number such that $(a_{Q^{'}, w}^{\frac{1}{n-1}})^{n-1}=a_{Q^{'}, w}$. It suffices to handle the multiliear multiplier in the form:
\begin{equation}  \label{Section3EquitionTensorizedm1}
\sum_{Q^{'}} \prod_{j=1}^{n-1} \varphi_{Q^{'}, j} \big(\frac{\xi_j-\mathfrak{c}(Q_j)}{4s(Q^{'})} \big)
\end{equation}
for any collection of functions $\{\varphi_{Q^{'}, j}\}_{Q^{'}, j}$ such that $\varphi_{Q^{'}, j}$ is supported in $(-\frac{1}{2}, \frac{1}{2})^d$ and satisfies
\begin{equation}  \label{Section3EquitionTensorizedmDerivativeBounds}
|\partial^{\alpha} \varphi_{Q^{'}, j}|\lesssim_{\alpha} 1,\ \forall \alpha: |\alpha|\leq \frac{N_2}{10n}
\end{equation}
uniformly in $Q^{'}$. In fact, assume we can estimate multipliers having the form \eqref{Section3EquitionTensorizedm1} with a constant only depending on the above derivative bounds. Then we can sum over $w$ and get the estimate for $\mathfrak{m}$. For each $Q^{'}$, fix a cube $Q_n\subset \mathbb{R}^d$ such that $s(Q_n)\sim s(Q^{'})$ and
\begin{equation*}
-4(Q_1+\cdots+Q_{n-1})\subset Q_n.
\end{equation*}
Denote $Q^{'}\times Q_n$ by $Q$. Pick a bump function $\varphi_{n}$ such that it is $1$ on $[-\frac{1}{4}, \frac{1}{4}]^d$ and supported in $(-\frac{1}{2}, \frac{1}{2})^d$. We keep such localization for $\xi_n:=-(\xi_1+\cdots +\xi_{n-1})$ by adding $\varphi_n$ to \eqref{Section3EquitionTensorizedm1}:
\begin{equation}  \label{Section3EquitionTensorizedm2}
\sum_{Q^{'}} \prod_{j=1}^{n-1} \varphi_{Q^{'}, j} \big(\frac{\xi_j-\mathfrak{c}(Q_j)}{4s(Q^{'})} \big)=\sum_{Q^{'}} \Big( \prod_{j=1}^{n-1} \varphi_{Q^{'}, j} \big(\frac{\xi_j-\mathfrak{c}(Q_j)}{4s(Q^{'})} \big) \Big)\cdot \varphi_n \big(\frac{\xi_n-\mathfrak{c}(Q_n)}{4s(Q_n)} \big).
\end{equation}
\item[Regularization]
Now we replace these $Q_j$'s with shifted dyadic cubes.
\begin{definition}  \label{Section3DefinitionGrid}
Let $\mathscr{G}$ be a collection of sets in $\mathbb{R}^d$. We say $\mathscr{G}$ is a grid if
\begin{equation*}
G\cap G^{'}\neq \emptyset \Longrightarrow G\subset G^{'} \text{\ or\ } G^{'}\subset G, \forall G, G^{'} \in \mathscr{G}.
\end{equation*}
Moreover, let $W>1$ be a given constant. We say $\mathscr{G}$ is a $W\text{-central}$ grid if it is a grid and
\begin{equation*}
G\subsetneqq G^{'}\Longrightarrow WG\subset G^{'}, \forall G, G^{'} \in \mathscr{G}.
\end{equation*}
\end{definition}
Note that we can divide the collection of shifted dyadic cubes $\big{\{}2^{l} \big([0, 1)^d+z+\alpha \big): l\in \mathbb{Z}, z\in \mathbb{Z}^d, \alpha\in \{0, \frac{1}{3}, \frac{2}{3}\}^d\big{\}}$ into $3^d$ groups so that each group is a grid. To show this, it suffices to consider the $d=1$ case. We can divide it into the following three groups:
\begin{gather*}
\{2^{l} \big([0, 1)+z\big): l\in \mathbb{Z}, z\in \mathbb{Z}\}, \\
\{2^{2l+1}\big([0, 1)+z+\frac{1}{3}\big), 2^{2l} \big([0, 1)+z+\frac{2}{3}\big): l\in \mathbb{Z}, z\in \mathbb{Z}\}, \\
\{2^{2l+1} \big([0, 1)+z+\frac{2}{3}\big), 2^{2l} \big([0, 1)+z+\frac{1}{3}\big): l\in \mathbb{Z}, z\in \mathbb{Z}\}.
\end{gather*}
This is why we regard shifted dyadic cubes as more regular objects. Each $Q_j$ can be replaced by a slightly larger shifted dyadic cube using the Lemma \ref{Section3LemmaShiftedDyadic} below.
\begin{lemma}  \label{Section3LemmaShiftedDyadic}
Let $Q_j\subset \mathbb{R}^d$ be a cube. Then there exists a shifted dyadic cube $\Xi_j$ such that $s(\Xi_j)\sim s(Q_j)$ and $\Xi_j\supset Q_j$.
\end{lemma}
\begin{proof}
We can assume $d=1$ and $\frac{1}{4}\leq s(Q_j)\leq \frac{1}{2}$. Write $Q_j=[a, b)$. Let $z$ be the largest integer such that $\frac{z}{3}\leq a$. Then we must have $b<\frac{z}{3}+1$, otherwise
\begin{equation*}
b-a\geq \frac{z}{3}+1-\frac{z+1}{3}=\frac{2}{3},
\end{equation*}
a contradiction. Thus $Q_j\subset [\frac{z}{3}, \frac{z}{3}+1)$.
\end{proof}
For each $Q_j$ in $Q=Q_1\times \cdots\times Q_n$, we take $\Xi_j$ using Lemma \ref{Section3LemmaShiftedDyadic}. We can also make $s(\Xi_1)=\cdots=s(\Xi_n)$ by replacing some $\Xi_j$'s with their parents. Let $\Xi:=\Xi_1\times \cdots \times \Xi_n$. Further write those bump functions in \eqref{Section3EquitionTensorizedm2} as follows:
\begin{align*}
\varphi_{Q^{'}, j} \big(\frac{\xi_j-\mathfrak{c}(Q_j)}{4s(Q^{'})} \big) & =:\varphi_{\Xi, j} \big(\frac{\xi_j-\mathfrak{c}(\Xi_j)}{4s(\Xi)} \big),\ 1\leq j\leq n-1 \\
\varphi_n \big(\frac{\xi_n-\mathfrak{c}(Q_n)}{4s(Q_n)} \big) & =:\varphi_{\Xi, n} \big(\frac{\xi_n-\mathfrak{c}(\Xi_n)}{4s(\Xi)} \big).
\end{align*}
These new $\varphi_{\Xi, j}$'s still satisfy $\varphi_{Q^{'}, j}$'s support condition and derivative bounds \eqref{Section3EquitionTensorizedmDerivativeBounds}. And RHS of \eqref{Section3EquitionTensorizedm2} becomes
\begin{equation}  \label{Section3EquitionTensorizedm3}
\sum_{Q^{'}} \prod_{j=1}^n \varphi_{\Xi, j} \big(\frac{\xi_j-\mathfrak{c}(\Xi_j)}{4s(\Xi)} \big),
\end{equation}
where $\Xi$ on the right depends on $Q^{'}$. It may happen that two different $Q^{'}, \widetilde{Q^{'}}$ give the same $\Xi$, but one can divide $\{Q^{'}\}$ into $O(1)$ groups to avoid this. In fact, we have the following stronger result.
\begin{lemma}  \label{Section3LemmaRankOfCubes}
Let $j_1, \cdots, j_{\lceil \frac{m}{d}\rceil}\in [n]$ be distinct indices. Then for any $Q^{'}$, there are at most $O(1)$ many $\widetilde{Q^{'}}$'s which give the same $(\Xi_{j_1}, \cdots, \Xi_{j_{\lceil \frac{m}{d}\rceil}})$ produced by $Q^{'}$.
\end{lemma}
\begin{proof}
Recall $Q:=Q^{'}\times Q_n$. Define
\begin{gather*}
\mathfrak{c}(Q^{'}):=\big(\mathfrak{c}(Q_1), \cdots, \mathfrak{c}(Q_{n-1}) \big)\in Q^{'}, \\
\mathfrak{c}(Q):=\big(\mathfrak{c}(Q_1), \cdots, \mathfrak{c}(Q_{n-1}), -\mathfrak{c}(Q_1)-\cdots-\mathfrak{c}(Q_{n-1}) \big)\in \Gamma_0\cap Q.
\end{gather*}
Then
\begin{equation*}
d\big (\mathfrak{c}(Q), \Gamma \big)\sim d\big(\mathfrak{c}(Q^{'}), P_{[n-1]} \Gamma \big)\sim C_1s(Q^{'}).
\end{equation*}
Take $\xi\in \Gamma$ such that $d(\mathfrak{c}(Q), \xi)\sim C_1s(Q^{'})$. By Type \uppercase\expandafter{\romannumeral1} non-degenerate condition, we can choose $\Theta\subset \{j_1, \cdots, j_{\lceil \frac{m}{d}\rceil}\}\times [d], \#\Theta=m$ such that for any $(j, k)\in [n]\times [d]$, there exists linear operator $\mathfrak{L}_{j, k}$ satisfying
\begin{equation*}
\xi_{j, k}=\mathfrak{L}_{j, k}\xi_{\Theta}, \ \forall \xi\in \Gamma.    
\end{equation*}
Thus
\begin{equation}  \label{Section3EquationRankOfCubes1}
d \big(\mathfrak{c}(Q)_{j, k}, \mathfrak{L}_{j, k}\mathfrak{c}(Q)_{\Theta} \big)\leq d \big(\mathfrak{c}(Q)_{j, k}, \xi_{j, k} \big)+d(\xi_{j, k}, \mathfrak{L}_{j, k}\xi_{\Theta})+d \big(\mathfrak{L}_{j, k}\xi_{\Theta}, \mathfrak{L}_{j, k}\mathfrak{c}(Q)_{\Theta} \big)\lesssim C_1s(Q^{'}).
\end{equation}
The same argument for $\widetilde{Q^{'}}$ gives
\begin{equation}  \label{Section3EquationRankOfCubes2}
d \big(\mathfrak{c}(\widetilde{Q})_{j, k}, \mathfrak{L}_{j, k}\mathfrak{c}(\widetilde{Q})_{\Theta}\big)\lesssim C_1s(\widetilde{Q^{'}}).
\end{equation}
Now the condition that $Q^{'}, \widetilde{Q^{'}}$ give the same $(\Xi_{j_1}, \cdots, \Xi_{j_{\lceil \frac{m}{d}\rceil}})$ implies $s(Q^{'})\sim s(\widetilde{Q^{'}})$ and
\begin{equation}  \label{Section3EquationRankOfCubes3}
d \big(\mathfrak{c}(Q)_{\Theta}, \mathfrak{c}((\widetilde{Q}))_{\Theta}\big)\lesssim C_1 s(Q^{'}).
\end{equation}
Combining \eqref{Section3EquationRankOfCubes1}, \eqref{Section3EquationRankOfCubes2}, and \eqref{Section3EquationRankOfCubes3} gives
\begin{equation}  \label{Section3EquationRankOfCubes4}
d(Q^{'}, \widetilde{Q^{'}})\leq d\big( \mathfrak{c}(Q^{'}), \mathfrak{c}(\widetilde{Q^{'}})\big)\lesssim C_1 s(Q^{'}).
\end{equation}
Note that $Q^{'}$'s are disjoint. Relation \eqref{Section3EquationRankOfCubes4} together with a common volume-packing argument immediately gives the desired result.
\end{proof}
Combining Lemma \ref{Section3LemmaRankOfCubes} and the pigeon hole principle, we can divide $\{Q^{'}\}$ into $O(1)$ groups such that $\Xi(Q^{'})$'s are distinct inside each group. Thus it suffices to study
\begin{equation}  \label{Section3EquitionTensorizedm4}
\sum_{\Xi} \prod_{j=1}^n \varphi_{\Xi, j} \big(\frac{\xi_j-\mathfrak{c}(\Xi_j)}{4s(\Xi)} \big)
\end{equation}
to handle our multiplier \eqref{Section3EquitionTensorizedm3}. Here $\Xi$ ranges in a finite collection of shifted cubes satisfying $\Xi\cap \Gamma_0\neq \emptyset, d(\Xi, \Gamma)\sim C_1 s(\Xi)$, and $\varphi_{\Xi, j}$ is supported in $(-\frac{1}{2}, \frac{1}{2})^d$ and satisfies
\begin{equation} \label{Section3EquitionTensorizedmDerivativeBoundsForXi}
|\partial^{\alpha}\varphi_{\Xi, j}|\lesssim_{\alpha} 1,\ \forall \alpha: |\alpha|\leq \frac{N_2}{10n}
\end{equation} 
uniformly in $\Xi$.
\item[Fulfilling the uncertainty principle]
Plug \eqref{Section3EquitionTensorizedm4} into $\Lambda_{\mathfrak{m}}(f_1, \cdots, f_n)$. Using
\begin{equation}  \label{Section3EquationBasicIntegralFormula}
\int_{(\mathbb{R}^d)^{n-1}} \widehat{F_1}(\xi_1)\cdots \widehat{F_{n-1}}(\xi_{n-1}) \widehat{F_n}(-\sum_{j=1}^{n-1} \xi_j) d\xi_1\cdots d\xi_{n-1}=\int_{\mathbb{R}^d} F_1\cdots F_{n},
\end{equation}
we get
\begin{equation}  \label{Section3EquationTensorizedMultilinearForm}
\Lambda_{\mathfrak{m}}(f_1, \cdots, f_n)=\sum_{\Xi} \int_{\mathbb{R}^d} \prod_{j=1}^n \Big(\varphi_{\Xi, j} \big(\frac{D-\mathfrak{c}(\Xi_j)}{4s(\Xi)} \big)f_j \Big).
\end{equation}
The functions in brackets have localized Fourier supports. And we have the following uncertainty principle: Let $g$ be a function $\mathbb{R}^d$. If $\hat{g}$ is supported in a cube $\Xi_j\subset \mathbb{R}^d$, then we expect
\begin{equation*}
g(x)\approx \sum_{I_j: s(I_j)\cdot s(\Xi_j)=1} a_{I_j}e^{2\pi i\mathfrak{c}(\Xi_j)} 1_{I_j}(x),
\end{equation*}
where $I_j$ ranges in a tiling of $\mathbb{R}^d$. Lemma \ref{Section3LemmaUncertaintyRigorousVersion} below is a rigorous version of this heuristic.
\begin{lemma}  \label{Section3LemmaUncertaintyRigorousVersion}
Let $\rho\in \mathcal{S}(\mathbb{R}^d)$ satisfy $\hat{\rho}$ is supported in $(-\frac{1}{2}, \frac{1}{2})^{d}$ and
\begin{equation*}
\sum_{z\in \mathbb{Z}^d} |\hat{\rho}(\xi-\frac{z}{3})|^2=1,\ \forall \xi\in \mathbb{R}^d. 
\end{equation*}
For dyadic cube $I_j$ and shifted dyadic cube $\Xi_j$ such that $s(I_j)\cdot s(\Xi_j)=1$, define
\begin{equation*}
\varphi_{\mathcal{R}_j}(x):=|I_j|^{-\frac{1}{2}}e^{2\pi i\mathfrak{c}(\Xi_j) \big(x-\mathfrak{c}(I_j) \big)}\rho \big( \frac{x-\mathfrak{c}(I_j)}{s(I_j)}\big)
\end{equation*}
for $\mathcal{R}_j:=I_j\times \Xi_j$. Then for any dyadic scale $s$, we have
\begin{equation*}
g=\sum_{\mathcal{R}_j: s\big(\Xi_j(R_j)\big)=s} \langle g | \varphi_{\mathcal{R}_j} \rangle \varphi_{\mathcal{R}_j}, 
\end{equation*}
where the convergence holds pointwise uniformly for $g\in \mathcal{S}(\mathbb{R}^d)$ and in $L^2$ for $g\in L^2(\mathbb{R}^d)$.
\end{lemma}
\begin{proof}
In this proof, we omit the index $j$ for simplicity. For fixed $s$ and $g\in \mathcal{S}(\mathbb{R}^d)$, we have
\begin{equation}  \label{Section3EquationUncertaintyRigorousVersion0}
|\langle g | \varphi_{\mathcal{R}} \rangle|=|\langle \hat{g} | \widehat{\varphi_{\mathcal{R}}} \rangle|\lesssim \Big(1+|\mathfrak{c}\big(I(R) \big)|+|\mathfrak{c}\big(\Xi(R) \big)| \Big)^{-100d}.
\end{equation}
Hence the pointwise convergence is uniform. To show the limit is $g$, we prove
\begin{equation}  \label{Section3EquationUncertaintyRigorousVersion1}
(|\hat{\rho} \big( \frac{x-\mathfrak{c}(\Xi)}{s(\Xi)}\big)|^2\hat{g})^{\vee}=\sum_{\mathcal{R}: \Xi(\mathcal{R})=\Xi} \langle g | \varphi_{\mathcal{R}} \rangle \varphi_{\mathcal{R}}
\end{equation}
holds pointwise for any shifted dyadic cube $\Xi$. Assume \eqref{Section3EquationUncertaintyRigorousVersion1} is true. Then summing over $\Xi: s(\Xi)=s$ gives the desired result because the Fourier transform of LHS converges to $\hat{g}$ in $L^1$. To prove \eqref{Section3EquationUncertaintyRigorousVersion1}, we only need to show the Fourier transforms of both sides agree in $L^1$. And it suffices to show they agree in $L^2$ since both have compact supports. By translation and dilation, our goal becomes
\begin{equation*}
|\hat{\rho}(\xi)|^2\hat{g}(\xi)=\sum_{z\in \mathbb{Z}^d} \langle \hat{g} | e^{-2\pi i z\cdot \xi} \hat{\rho}\rangle e^{-2\pi i z\cdot \xi} \hat{\rho}(\xi).
\end{equation*}
Both sides are $L^2$ functions on $(-\frac{1}{2}, \frac{1}{2})^{d}$. It suffices to check that their Fourier coefficients agree, which is a direction calculation.

By \eqref{Section3EquationUncertaintyRigorousVersion0}, we also get the $L^2$ convergence for $g\in \mathcal{S}(\mathbb{R}^d)$. Let $\mathscr{R}\subset \{\mathcal{R}_j: s\big(\Xi_j(R_j)\big)=s\}$ be a finite collection. The partial sum operator
\begin{equation}  \label{Section3EquationPartialSumOperator}
A_{\mathscr{R}} h:=\sum_{\mathcal{R}\in \mathscr{R}}\langle h | \varphi_{\mathcal{R}} \rangle \varphi_{\mathcal{R}},\ h\in L^2(\mathbb{R}^d)
\end{equation}
is $L^2$ bounded uniformly in $\mathscr{R}$. This is proved by the common argument combining Schur's test and orthogonality (as in the proof of Theorem \ref{Section4TheoremLittlewood-PaleyDyadic} below). Then we see the $L^2$ convergence for $g\in L^2(\mathbb{R}^d)$ via approximation, which completes the proof.
\end{proof}
Let $f_1, \cdots, f_{n-1}\in \mathcal{S}(\mathbb{R}^d)$ and $f_n$ be a simple function. Consider each $\Xi$ in the sum of \eqref{Section3EquationTensorizedMultilinearForm}. We replace $f_j$ with $A_{\mathscr{R}_j} f_j$ defined by \eqref{Section3EquationPartialSumOperator}, where $s=s(\Xi)$ and $\mathscr{R}_j$ is a finite collection associated with $s$. The new multilinear form is as follows:
\begin{equation}  \label{Section3EquationTensorizedMultilinearFormNew}
\widetilde{\Lambda_{\mathfrak{m}}}(f_1, \cdots, f_n):=\sum_{\Xi} \sum_{\mathcal{R}_j\in \mathscr{R}_j} \Bigg(\int_{\mathbb{R}^d} \prod_{j=1}^n \Big(\varphi_{\Xi, j} \big(\frac{D-\mathfrak{c}(\Xi_j)}{4s(\Xi)} \big)\varphi_{\mathcal{R}_j} \Big) \Bigg) \prod_{j=1}^n \langle f_j | \varphi_{\mathcal{R}_j} \rangle.
\end{equation}
Denote the associated multilinear operator acting on $f_1, \cdots, f_{n-1}$ by $\widetilde{T_{\mathfrak{m}}}$. When all involved finite collections $\mathscr{R}_j$ goes to their full collections, we have
\begin{equation}  \label{Section3EquationTensorizedMultilinearFormNewApproximateOld}
\widetilde{\Lambda_{\mathfrak{m}}}(f_1, \cdots, f_n)\to \Lambda_{\mathfrak{m}}(f_1, \cdots, f_n)
\end{equation}
by Lemma \ref{Section3LemmaUncertaintyRigorousVersion}. Note that $\widetilde{\Lambda_{\mathfrak{m}}}$ is also well-defined when $f_1, \cdots, f_n\in L^2(\mathbb{R}^d)$. Therefore, as in \cite{MTTMultilinearSingularMultipliers}\cite{MS2}\cite{DOPBilinearVariational}, we only need to prove bounds of the following form:
\begin{equation}  \label{Section3EquationTensorizedMultilinearFormNewBounds}
|\widetilde{\Lambda_{\mathfrak{m}}}(f_1, \cdots, f_{n-1}, 1_{E_n\backslash \Omega})|\lesssim (\prod_{j=1}^{n-1} |E_j|^{\frac{1}{p_j}}) |E_n|^{\frac{1}{q^{'}}}
\end{equation}
for $|f_j|\leq 1_{E_j}, 1\leq j\leq n-1$ and $(p_1, \cdots, p_{n-1}, q)$ in the range of Theorem \ref{MainTheorem1}. Here $E_j, j\in [n]$ are bounded sets, and $\Omega$ is a suitable exceptional set such that $ \Omega\subset E_n, |\Omega|\leq \frac{1}{2}|E_n|$. More precisely, we deduce Theorem \ref{MainTheorem1} from \eqref{Section3EquationTensorizedMultilinearFormNewBounds} step by step:
\begin{enumerate}
\item The estimate \eqref{Section3EquationTensorizedMultilinearFormNewBounds} is equivalent to the restricted weak bound for $\widetilde{T_{\mathfrak{m}}}$. (We will explain more details on this equivalence in Section \ref{Section6}.) We can use multilinear interpolation to get $\widetilde{T_{\mathfrak{m}}}$ is bounded for $(p_1, \cdots, p_{n-1}, q)$ in the full range of Theorem \ref{MainTheorem1}. In this multilinear interpolation, we fix those $f_j$ such that $p_j=\infty$ and regard $\widetilde{T_{\mathfrak{m}}}$ is a multilinear operator on other functions. This gives the bounds in full range under the additional assumption that $f_j$ is compactly supported if $p_j=\infty$. A further approximation removes this assumption.
\item Let $f_1, \cdots, f_{n-1}\in \mathcal{S}(\mathbb{R}^d)$ and $f_n$ be a simple function. Apply \eqref{Section3EquationTensorizedMultilinearFormNewApproximateOld} to obtain the bounds for $\Lambda_{\mathfrak{m}}$, which imply $T_{\mathfrak{m}}$'s strong bounds when $q\geq 1$ and weak bounds when $q<1$. Then we can define $T_{\mathfrak{m}}$ for simple functions.
\item Let $f_j\in \mathcal{S}(\mathbb{R}^d)$ approximate simple functions if $p_j\neq \infty$. We can extend the weak bounds for $T_{\mathfrak{m}}$ to simple functions.
\item Apply multilinear interpolation to get $T_{\mathfrak{m}}$ is bounded in the full range of Theorem \ref{MainTheorem1}. Here we still fix $f_j\in \mathcal{S}(\mathbb{R}^d)$ if $p_j=\infty$ and let other $f_j$ be simple functions in the multilinear interpolation.
\end{enumerate}
Back to \eqref{Section3EquationTensorizedMultilinearFormNew}, we can further simplify the summation. Apply \eqref{Section3EquationBasicIntegralFormula}. We see
\begin{equation*}
\int_{\mathbb{R}^d} \prod_{j=1}^n \Big(\varphi_{\Xi, j} \big(\frac{D-\mathfrak{c}(\Xi_j)}{4s(\Xi)} \big)\varphi_{\mathcal{R}_j} \Big)=0
\end{equation*}
unless $\Xi(\mathcal{R}_j)\cap \Xi_{j}\neq \emptyset, \forall j\in [n]$. Moreover, write $I_j:=I(\mathcal{R}_j)$ and $s:=s(I_j)$. We have
\begin{equation}  \label{Section3EquationTranslatePhisicalIntervalsEstimate}
\begin{split}
|\int_{\mathbb{R}^d} & \prod_{j=1}^n \Big(\varphi_{\Xi, j} \big(\frac{D-\mathfrak{c}(\Xi_j)}{4s(\Xi)} \big)\varphi_{\mathcal{R}_j} \Big)| \\
&\lesssim \int_{\mathbb{R}^d}\int_{\mathbb{R}^{dn}} \prod_{j=1}^n \big( s^{-d}(1+|\frac{t_j}{s}|)^{-\frac{N_2}{10n}}|I_j|^{-\frac{1}{2}}(1+|\frac{x-t_j-\mathfrak{c}(I_j)}{s}|)^{-\frac{N_2}{10n}} \big) dt_1 \cdots dt_n dx  \\
&\lesssim s^{d(1-\frac{n}{2})} (1+\max_{j, \tilde{j}} \{|\frac{\mathfrak{c}(I_j)-\mathfrak{c}(I_{\tilde{j}})}{s}|\})^{-\frac{N_2}{20n}}.
\end{split}    
\end{equation}
This suggests we only need to consider the diagonal terms satisfying $I_1=\cdots=I_n$. Now we make this intuition rigorous.
\begin{definition} \label{Section3DefinitionWavePacket}
Let $I_j$ be a dyadic cube and $\Xi_j$ be a shifted dyadic cube  such that $s(I_j)\cdot s(\Xi_j)=1$. Define $\mathcal{R}_j:=I_j\times \Xi_j$. A function $\varphi_{\mathcal{R}_j}$ is called a wave packet adapted to $\mathcal{R}_j$ if it can be written into the following form:
\begin{equation*}
\varphi_{\mathcal{R}_j}(x)=|I_j|^{-\frac{1}{2}}e^{2\pi i\mathfrak{c}(\Xi_j)\big(x-\mathfrak{c}(I_j) \big)}\zeta_{\mathcal{R}_j} \big( \frac{x-\mathfrak{c}(I_j)}{s(I_j)}\big),
\end{equation*}
where $\zeta_{\mathcal{R}_j}\in \mathcal{S}(\mathbb{R}^d)$ satisfies
\begin{enumerate}
\item The Fourier transform $\widehat{\zeta_{\mathcal{R}_j}}$ is supported in $(-\frac{1}{2}, \frac{1}{2})^{d}$.
\item For any $\alpha: |\alpha|\leq 1$, we have
\begin{equation}  \label{Section3EquationWavePacketDecay}
|\partial^{\alpha} \zeta_{\mathcal{R}_j}(x)|\lesssim (1+|x|)^{-N_1}.
\end{equation}
The implicit constant in \eqref{Section3EquationWavePacketDecay} does not depend on $\mathcal{R}_j$.
\end{enumerate}
\end{definition}
For example, those $\varphi_{\mathcal{R}_j}$'s in Lemma \ref{Section3LemmaUncertaintyRigorousVersion} are wave packets. Let $\varphi_{\mathcal{R}_j}$ be a wave packet and $z_j\in \mathbb{Z}^d$. We can write
\begin{equation*}
\begin{split}
\varphi_{\mathcal{R}_j}(x) & =|I_j|^{-\frac{1}{2}}e^{2\pi i\mathfrak{c}(\Xi_j)\big(x-\mathfrak{c}(I_j) \big)}\zeta \big( \frac{x-\mathfrak{c}(I_j)}{s(I_j)}\big) \\
& =(1+|z_j|)^{N_1}\Big( (1+|z_j|)^{-N_1} |I_j|^{-\frac{1}{2}}e^{2\pi i\mathfrak{c}(\Xi_j)\big(x-\mathfrak{c}(I_j) \big)}\zeta \big( \frac{x-\mathfrak{c}(I_j)-z_js(I_j)}{s(I_j)}+z_j\big)\Big) \\
& =:c(1+|z_j|)^{N_1} \varphi_{\big( I_j+z_js(I_j) \big)\times \Xi_j}(x),
\end{split}
\end{equation*}
where $|c|=1$ and $\varphi_{\big( I_j+z_js(I_j) \big)\times \Xi_j}$ is another wave packet adapted to ${\big( I_j+z_js(I_j) \big)\times \Xi_j}$. This means translating physical intervals of wave packets costs a polynomial factor of the distance. If $\Xi(\mathcal{R}_j)\cap \Xi_{j}\neq \emptyset, \forall j\in [n]$, then $\frac{\mathfrak{c}\big( \Xi(\mathcal{R})\big)-\mathfrak{c}(\Xi)}{s(\Xi)}\in (\{-\frac{2}{3}, -\frac{1}{3}, 0, \frac{1}{3}, \frac{2}{3}\}^d)^n=:W$. Classifying the value of this ratio and using the above translation property together with \eqref{Section3EquationTranslatePhisicalIntervalsEstimate}, we can write RHS of \eqref{Section3EquationTensorizedMultilinearFormNew} into
\begin{equation}  \label{Section3EquationAlmostFinalDiscreteModelSum}
\sum_{w\in W}\sum_{z^{'}\in (\mathbb{Z}^d)^{n-1}} \sum_{\mathcal{R}\in \mathscr{R}^{w, z^{'}}} c_{\mathcal{R}}^{w, z^{'}} |I(\mathcal{R})|^{1-\frac{n}{2}}\prod_{j=1}^n \langle f_j | \varphi_{\mathcal{R}_j}^{z^{'}} \rangle.
\end{equation}
Here $\mathscr{R}^{w, z^{'}}$ is a finite collection of vector rectangles, $c_{\mathcal{R}}^{w, z^{'}}$ is a constant satisfying
\begin{equation*}
|c_{\mathcal{R}}^{w, z^{'}}|\lesssim (1+|z^{'}|)^{-100d(n-1)},
\end{equation*}
and $\varphi_{\mathcal{R}_j}^{z^{'}}$ is a wave packet adapted to $\mathcal{R}_j$. Fixing $w, z^{'}$, we only need to estimate
\begin{equation*}
\sum_{\mathcal{R}\in \mathscr{R}} |I(\mathcal{R})|^{1-\frac{n}{2}}\prod_{j=1}^n |\langle f_j | \varphi_{\mathcal{R}_j} \rangle|.
\end{equation*}
\item[Sparsification]
We need a final reduction to facilitate later geometric arguments.
\begin{definition}
Let $d\in \mathbb{N}_+$. We say a collection $\mathscr{G}$ of cubes in $\mathbb{R}^d$ is $L\text{-sparse}$ if the following properties hold:
\begin{enumerate}
\item For any $G, \widetilde{G}\in \mathscr{G}$, $s(G)>s(\widetilde{G})$ implies $s(G)>L s(\widetilde{G})$.
\item For any $G, \widetilde{G}\in \mathscr{G}$, $s(G)=s(\widetilde{G}), G\neq \widetilde{G}$ implies $d(G, \widetilde{G})>L s(G)$.  
\end{enumerate}
\end{definition}
By the pigeon hole principle, we can divide $\mathscr{R}$ into $O(1)$ groups such that for each group $\mathscr{P}$, the collections $\{\Xi(\mathcal{R}_j): \mathscr{R}\in \mathscr{P}\}, j\in [n]$ are all $C_3\text{-sparse}$ grids.
\end{description}
After all these reductions, Theorem \ref{MainTheorem1} can be deduced from Theorem \ref{MainTheorem1Discrete} below.
\begin{theorem}  \label{MainTheorem1Discrete}
Let $n, d, m, p_1, \cdots, p_{n-1}, q$ and $\Gamma$ be as in Theorem \ref{MainTheorem1}. Suppose $\mathscr{R}$ is a finite collections of vector rectangles such that:
\begin{enumerate}
\item We have $\Xi(\mathcal{R})\cap \Gamma_0\neq \emptyset$ for any $\mathcal{R}\in \mathscr{R}$.
\item The relation $d \big(\Xi(\mathcal{R}), \Gamma \big)\sim C_1 s\big( \Xi(\mathcal{R}) \big)$ holds for any $\mathcal{R}\in \mathscr{R}$.
\item Let $j_1, \cdots, j_{\lceil \frac{m}{d}\rceil}\in [n]$ be distinct indices. Then $\Xi(\mathcal{R}_{j_1}), \cdots, \Xi(\mathcal{R}_{j_{\lceil \frac{m}{d}\rceil}})$ uniquely determine all $\Xi(\mathcal{R}_j), j\in [n]$ for any $\mathcal{R}\in \mathscr{R}$.
\item The collections $\{\Xi(\mathcal{R}_j): \mathcal{R}\in \mathscr{R}\}, j\in [n]$ are all $C_3\text{-sparse}$ grids.
\end{enumerate}
For any $\mathcal{R}\in \mathscr{R}$ and $j\in [n]$, let $\varphi_{\mathcal{R}_j}$ be a wave packet adapted to $\mathcal{R}_j$. Then for any $(c_{\mathcal{R}})_{\mathcal{R}\in \mathscr{R}}: |c_{\mathcal{R}}|=1, \forall \mathcal{R}\in \mathscr{R}$, the operator
\begin{equation*}
T_{\mathscr{R}}^{dis}(f_1, \cdots, f_{n-1}):=\sum_{\mathcal{R}\in \mathscr{R}} c_{\mathcal{R}}|I(\mathcal{R})|^{1-\frac{n}{2}}\prod_{j=1}^{n-1} \langle f_j | \varphi_{\mathcal{R}_j} \rangle \varphi_{\mathcal{R}_n}
\end{equation*}
satisfies \eqref{MainEstimate} with $T_{\mathfrak{m}}$ replaced by $T_{\mathscr{R}}^{dis}$ (in each of the three cases). The implicit constant is independent with $\mathscr{R}$ and $(c_{\mathcal{R}})_{\mathcal{R}\in \mathscr{R}}$.
\end{theorem}
\begin{remark}
We can assume the third property of $\mathscr{R}$ because of Lemma \ref{Section3LemmaRankOfCubes}, the dividing procedure after that, and the fact that all $\mathcal{R}\in \mathscr{R}$ belong to the same $\mathscr{R}^{w, z^{'}}$.
\end{remark}
In fact, consider the dual form
\begin{equation}  \label{Section3EquationMainGoalDiscreteVersion}
\Lambda_{\mathscr{R}}^{dis}(f_1, \cdots, f_n):=\sum_{\mathcal{R}\in \mathscr{R}} |I(\mathcal{R})|^{1-\frac{n}{2}}\prod_{j=1}^n |\langle f_j | \varphi_{\mathcal{R}_j} \rangle|.
\end{equation}
The estimates for $T_{\mathscr{R}}^{dis}$ in Theorem \ref{MainTheorem1Discrete} are equivalent to \eqref{Section3EquationTensorizedMultilinearFormNewBounds} with $\widetilde{\Lambda_{\mathfrak{m}}}$ replaced by $\Lambda_{\mathscr{R}}^{dis}$. They imply \eqref{Section3EquationTensorizedMultilinearFormNewBounds}, and therefore Theorem \ref{MainTheorem1}. Below we will focus on proving \eqref{Section3EquationTensorizedMultilinearFormNewBounds} with $\widetilde{\Lambda_{\mathfrak{m}}}$ replaced by $\Lambda_{\mathscr{R}}^{dis}$. Fix a finite collection $\mathscr{R}$, bounded sets $E_j, j\in [n]$, and $|f_j|\leq 1_{E_j}, j\in [n]$ from now on. We can also assume each term in \eqref{Section3EquationMainGoalDiscreteVersion} is non-zero.


\section{Trees, Mass, and the Selection Algorithm for Scalar Trees} \label{Section4}
We would like to define our trees using $C_2\Xi(\mathcal{R})$, but we will work with slightly different objects to get the central grid structure. Lemma \ref{Section4LemmaCentralGrid} below modifies the centralization lemma in \cite{GLBHTUniform}.
\begin{lemma}  \label{Section4LemmaCentralGrid}
Let $L\gg 1$ and $\mathscr{A}$ be a finite $L\text{-sparse}$ collection of intervals in $\mathbb{R}$. Then there exists a new collection $\mathscr{G}=\{G(A): A\in \mathscr{A}\}$ of intervals such that
\begin{enumerate}
\item The property $A\subset G(A)\subset 2A$ hold for any $A\in \mathscr{A}$.
\item $\mathscr{G}$ is a $L^{\frac{1}{2}}\text{-central}$ grid.
\end{enumerate}
\end{lemma}
\begin{proof}
We prove by induction on $\#\mathscr{A}$. The $\#\mathscr{A}=1$ case is obvious. Assume the conclusion holds for $n-1$, we prove it for $n$. Write the elements in $\{s(A): A\in \mathscr{A}\}$ as $s_1>\cdots>s_{n^{'}}$. Take one $A_1\in \mathscr{A}$ such that $s(A_1)=a_1$. Apply the induction hypothesis to $\mathscr{A}\backslash\{A_1\}$. We get $G(A), A\neq A_1$. Now we define $G(A_1)$. First let $G_1:=A_1$. The define $G_k, 2\leq k\leq n^{'}$, as follows:
\begin{equation*}
G_{k}:=\conv \Big( G_{k-1}\cup\big( \bigcup_{\substack{A: s(A)=s_k,\\d \big(G(A), G_{k-1} \big)\leq 10s_k}} L^{\frac{1}{2}}G(A) \big) \Big).
\end{equation*}
We let $G(A_1):=G_{n^{'}}$. First note that
\begin{equation}  \label{Section4EquationCentralGridLengthInequality}
|G_{k}|\leq |G_{k-1}|+O(W^{\frac{1}{2}}s_{k}),\ \forall 2\leq k\leq n^{'}.
\end{equation}
Thus
\begin{equation*}
|G_{n^{'}}|\leq s_1+\sum_{k=2}^{n^{'}} O(W^{\frac{1}{2}}s_{k})<(1+\frac{1}{100})s_1.
\end{equation*}
This means $G_{n^{'}}\subset 2A_1$ since $A_1\subset G_{n^{'}}$. It then suffices to show $G(A)\cap G(\widetilde{A})\neq \emptyset, s(A)<s(\widetilde{A})$ implies $L^{\frac{1}{2}}G(A)\subset G(\widetilde{A})$. By the induction hypothesis, we only need to tackle the $\widetilde{A}=A_1$ case. Assume $s(A)=s_{k_0}$. It suffices to prove $\big(G(A), G_{k_0-1} \big)\leq 10s_{k_0}$. If this does not hold, then $d\big(G(A), G_{k_0-1} \big)> 10s_{k_0}$. Write $G_{k_0-1}=[a, b)$. Assume $G(A)\subset (b, \infty)$ without loss of generality. Now \eqref{Section4EquationCentralGridLengthInequality} implies
\begin{equation*}
|G_{n^{'}}|\leq |G_{k_0}|+\frac{1}{100}s_{k_0},
\end{equation*}
so $\big(G(A), G_{k_0} \big)<\frac{1}{100}s_{k_0}$. Combine this with $d\big(G(A), G_{k_0-1} \big)> 10s_{k_0}$. We get there exits $A^{'}$ such that $s \big(G(A^{'}) \big)=s_{k_0}$ and
\begin{gather*}
d\big(G(A^{'}), G_{k_0-1} \big)\leq 10s_{k_0}, \\
L^{\frac{1}{2}} G(A^{'})\cap (b, \infty)\neq \emptyset.
\end{gather*}
But then we have
\begin{gather*}
A\subset \big(b, b+O(L^{\frac{1}{2}}s_{k_0}) \big),  \\
A^{'}\subset (b-2L^{\frac{1}{2}}s_{k_0}, b+20s_{k_0}),
\end{gather*}
which means $d(A, A^{'})=O(L^{\frac{1}{2}}s_{k_0})\ll Ls_{k_0}$, a contradiction. This finishes the proof.
\end{proof}
Recall the fourth property in Theorem \ref{MainTheorem1Discrete}. Apply Lemma \ref{Section4LemmaCentralGrid} to $\{C_2\Xi(\mathcal{R}_j): \mathcal{R}\in \mathscr{R}\}, j\in [n]$, with $L:=C_3^{\frac{2}{3}}$. We get the $C_3^{\frac{1}{3}}\text{-central}$ grid $\{G\big(C_2\Xi(\mathcal{R}_j) \big): \mathcal{R}\in \mathscr{R}\}, j\in [n]$.
\begin{definition}
For $\mathcal{R}\in \mathscr{R}, j\in [n]$, denote $G\big(C_2\Xi(\mathcal{R}_j) \big)$ by $C_2\circ \Xi(\mathcal{R}_j)$ and $\Big( C_2\circ \Xi(\mathcal{R}_1) \Big)\times \cdots \times\Big( C_2\circ \Xi(\mathcal{R}_n) \Big)$ by $C_2\circ \Xi(\mathcal{R})$.
\end{definition}
Now we define two kinds of order relation.
\begin{definition}
Let $I_j\subset \mathbb{R}^d$ be a dyadic interval, $\xi_j\in \mathbb{R}^d$ be a point, and $\mathcal{R}_j$ be a rectangle. Then we write
\begin{enumerate}
\item $\mathcal{R}_j\leq (I_j, \xi_j)$ if $I(\mathcal{R}_j)\subset I_j$ and $\xi_j\subset \Xi(\mathcal{R}_j)$,
\item $\mathcal{R}_j\lesssim (I_j, \xi_j)$ if $I(\mathcal{R}_j)\subset I_j$ and $\xi_j\subset C_2\circ \Xi(\mathcal{R}_j)$.
\end{enumerate}
Moreover, let $\mathcal{R}_j, \widetilde{\mathcal{R}}_j$ be two rectangles. Write
\begin{enumerate}
\item $\mathcal{R}_j\leq \widetilde{\mathcal{R}}_j$ if $I(\mathcal{R}_j)\subset I(\widetilde{\mathcal{R}}_j)$ and $\Xi(\widetilde{\mathcal{R}}_j)\subset \Xi(\mathcal{R}_j)$,
\item $\mathcal{R}_j\lesssim \widetilde{\mathcal{R}}_j$ if $I(\mathcal{R}_j)\subset I(\widetilde{\mathcal{R}}_j)$ and $C_2 \circ \Xi(\widetilde{\mathcal{R}}_j)\subset C_2\circ \Xi(\mathcal{R}_j)$.
\end{enumerate}
We also need the vector version of these relations. Let $I\subset \mathbb{R}^d$ be a dyadic interval, $\xi\in \Gamma_0$ be a point, and $\mathcal{R}$ be a vector rectangle. Define
\begin{enumerate}
\item $\mathcal{R}\leq (I, \xi)$ if $\mathcal{R}_j\leq (I, \xi_j), \forall j\in[n]$,
\item $\mathcal{R}\lesssim (I, \xi)$ if $\mathcal{R}_j\lesssim (I, \xi_j), \forall j\in[n]$.
\end{enumerate}
Similarly, for two vector rectangles $\mathcal{R}=(\mathcal{R}_1, \cdots, \mathcal{R}_n), \widetilde{\mathcal{R}}=(\widetilde{\mathcal{R}}_1, \cdots, \widetilde{\mathcal{R}}_n)$, define
\begin{enumerate}
\item $\mathcal{R}\leq \widetilde{\mathcal{R}}$ if $\mathcal{R}_j\leq \widetilde{\mathcal{R}}_j, \forall j\in[n]$,
\item $\mathcal{R}\lesssim \widetilde{\mathcal{R}}$ if $\mathcal{R}_j\lesssim \widetilde{\mathcal{R}}_j, \forall j\in[n]$.
\end{enumerate}
\end{definition}
We will work with the relation $\lesssim$ most of the time including Definition \ref{Section4DefinitionTrees} below, but the relation $\leq$ will also be useful in Section \ref{Section5}.
\begin{definition}  \label{Section4DefinitionTrees}
Let $j\in [n]$. A non-empty collection $\mathscr{T}_j\subset \{\mathcal{R}_j: \mathcal{R}\in \mathscr{R}\}$ is called a $j\text{-tree}$ if there exists a dyadic interval $I_j\subset \mathbb{R}^d$ and a point $\xi_j\in \mathbb{R}^d$ such that $\mathcal{R}_j\lesssim (I_j, \xi_j), \forall \mathcal{R}_j\in \mathscr{T}_j$. All such $I_j$'s are intersect with each other, so the smallest one exists. Denote it by $I_{\mathscr{T}_j}$. If there exists $\xi_j$ such that $\xi_j\subset C_2\circ \Xi(\mathcal{R}_j)\backslash 10\Xi(\mathcal{R}_j), \forall \mathcal{R}_j\in \mathscr{T}_j$, then we say $\mathscr{T}_j$ is lacunary. Moreover, A non-empty collection $\mathscr{T}\subset \mathscr{R}$ is called a vector tree if there exists a dyadic interval $I\subset \mathbb{R}^d$ and a point $\xi\in \Gamma_0$ such that $\mathcal{R}\lesssim (I, \xi), \forall \mathcal{R}\in \mathscr{T}$. Denote the smallest $I$ by $I_{\mathscr{T}}$.
\end{definition}
Lacunary $j\text{-trees}$ provide us with the classical configuration to apply the C-Z theory. More precisely, we shall estimate sums over trees by $\mathbf{M}_j$ defined below and estimate $\mathbf{M}_j$ using a localized C-Z type argument.
\begin{definition}
Let $j\in [n]$ and $\mathscr{P}_j\subset \{\mathcal{R}_j: \mathcal{R}\in \mathscr{R}\}$ be a collection. Define
\begin{equation*}
\mathbf{M}_j(\mathscr{P}_j):=\sup_{\substack{\mathscr{T}_j\subset \mathscr{P}_j,\\\mathscr{T}_j \text{\ is a lacunary $j\text{-tree}$}}} \big(\frac{1}{|I_{\mathscr{T}_j}|} \sum_{\mathcal{R}_j\in \mathscr{T}_j} |\langle f_j|\varphi_{\mathcal{R}_j} \rangle|^2 \big)^{\frac{1}{2}}.
\end{equation*}
\end{definition}
\begin{remark}
Recall we have fixed $f_j, j\in [n]$. We let $\mathbf{M}_j(\mathscr{P}_j):=0$ if $\mathscr{P}_j=\emptyset$.
\end{remark}
The geometry of Whitney decomposition guarantees every vector tree has at least two lacunary indices so that we can bound sums over trees using $\mathbf{M}_j$.
\begin{lemma}  \label{Section4LemmaTreeEstimate}
Let $\mathscr{T}$ be a vector tree. Then
\begin{equation*}
\sum_{\mathcal{R}\in \mathscr{T}} |I(\mathcal{R})|^{1-\frac{n}{2}} \prod_{j=1}^n |\langle f_j | \varphi_{\mathcal{R}_j} \rangle|\lesssim |I_{\mathscr{T}}|\prod_{j=1}^n \mathbf{M}_j(\mathscr{T}_j),
\end{equation*}
where $\mathscr{T}_j:=\{\mathcal{R}_j: \mathcal{R}\in \mathscr{T}\}$.
\end{lemma}
\begin{proof}
Take $\xi\in \Gamma_0$ such that $\mathcal{R}\lesssim (I_{\mathscr{T}}, \xi), \forall \mathcal{R}\in \mathscr{T}$. Fix a large constant $C_0\ll C_1$. Let $\Upsilon:=\{\Xi(\mathcal{R}): \mathcal{R}\in \mathscr{T}\}$. Here are some facts about $\Upsilon$:
\begin{enumerate}
\item For $\Xi, \widetilde{\Xi}\in \Upsilon$, $s(\Xi)=s(\widetilde{\Xi})$ implies $\Xi=\widetilde{\Xi}$. This is due to $C_2 \circ \Xi\cap C_2 \circ \widetilde{\Xi}\neq \emptyset$ and the $C_3\text{-sparse}$ property.
\item For all but at most one $\Xi\in \Upsilon$, we can find one $j$ such that $\xi_j\in C_2\circ \Xi_j\backslash C_0\Xi_j$. In fact, if there exist two $\Xi, \widetilde{\Xi}\in \Upsilon$ such that $\xi\in C_0\Xi\cap C_0\widetilde{\Xi}$. The first implies $s(\Xi)\neq s(\widetilde{\Xi})$. Assume $s(\Xi)<s(\widetilde{\Xi})$ without loss of generality. Then $d(\Xi, \widetilde{\Xi})\lesssim C_0 s(\widetilde{\Xi})$. Combining this with $d(\Xi, \Gamma)\sim C_1 s(\Xi)$ gives
\begin{equation*}
d(\widetilde{\Xi}, \Gamma)\leq d(\widetilde{\Xi}, \Xi)+d(\Xi, \Gamma)\lesssim C_0 s(\widetilde{\Xi})+C_1 s(\Xi)\ll C_1 s(\widetilde{\Xi}),
\end{equation*}
a contradiction.
\item If $\Xi\in \Upsilon$ satisfy $\Xi_j\in C_2\circ \Xi_j\backslash C_0\Xi_j$ for some $j$, then there must exist $j^{'}\neq j$ such that $\xi_{j^{'}}\in C_2\circ \Xi_{j^{'}}\backslash 10\Xi_{j^{'}}$. Otherwise $\xi_{j^{'}}\in 10\Xi_{j^{'}}, \forall j^{'}\neq j$. Take $\tau\in \Xi\cap \Gamma_0$. Then $d(\xi_j, \tau_j)\gtrsim C_0s(\Xi)$ but $d(\xi_{j^{'}}, \tau_{j^{'}})\lesssim s(\Xi), \forall j^{'}\neq j$. On the other hand, we have
\begin{equation*}
\sum_{j} \xi_j=\sum_{j} \tau_j=0\Longrightarrow \xi_j-\tau_j=\sum_{j^{'}\neq j} (\tau_{j^{'}}-\xi_{j^{'}}).
\end{equation*}
This is a contradiction since $C_0$ is large.
\end{enumerate}
The above properties imply that we can write
\begin{equation*}
\mathscr{T}=\mathscr{T}(\Xi)\cup \Big( \bigcup_{j\neq j^{'}} \mathscr{T}(j, j^{'}) \Big),
\end{equation*}
where $\Xi(\mathcal{R})=\Xi, \forall \mathcal{R}\in \mathscr{T}(\Xi)$, and $\xi_j\in C_2\circ \Xi_j\backslash 10\Xi_j, \xi_{j^{'}}\in C_2\circ \Xi_{j^{'}}\backslash 10\Xi_{j^{'}}, \forall \mathcal{R}\in \mathscr{T}(j, j^{'})$. Note that for $\mathcal{R}=(\mathcal{R}_1, \cdots, \mathcal{R}_n)$, the first property above implies $\mathcal{R}\mapsto \mathcal{R}_j$ is a bijection. We can then estimate the sum over $\mathscr{T}(j, j^{'})$ as follows:
\begin{equation*}
\begin{split}
\sum_{\mathcal{R}\in \mathscr{T}(j, j^{'})} & |I(\mathcal{R})|^{1-\frac{n}{2}} \prod_{j=1}^n |\langle f_j | \varphi_{\mathcal{R}_j} \rangle| \\
& \leq \Big( \prod_{j^{''}\neq j, j^{'}} \sup_{\mathcal{R}_{j^{''}}\in \mathscr{T}_{j^{''}}} \frac{|\langle f_{j^{''}} | \varphi_{\mathcal{R}_{j^{''}}} \rangle|}{|I(\mathcal{R})|^{\frac{1}{2}}} \Big)\Big( \sum_{\mathcal{R}\in\mathscr{T}(j, j^{'})} |\langle f_j | \varphi_{\mathcal{R}_j} \rangle||\langle f_{j^{'}} | \varphi_{\mathcal{R}_{j^{'}}} \rangle| \Big) \\
& \leq \Big(\prod_{j^{''}\neq j, j^{'}} \mathbf{M}_{j^{''}}(\mathscr{T}_{j^{''}})\Big) |I_{\mathscr{T}}| \Big(\frac{1}{|I_{\mathscr{T}}|}\sum_{\mathcal{R}\in\mathscr{T}(j, j^{'})}|\langle f_j | \varphi_{\mathcal{R}_j} \rangle|^2\Big)^{\frac{1}{2}} \Big(\frac{1}{|I_{\mathscr{T}}|}\sum_{\mathcal{R}\in\mathscr{T}(j, j^{'})}|\langle f_{j^{'}} | \varphi_{\mathcal{R}_{j^{'}}} \rangle|^2\Big)^{\frac{1}{2}} \\
& \leq |I_{\mathscr{T}}|\prod_{j=1}^n \mathbf{M}_j(\mathscr{T}_j).
\end{split}
\end{equation*}
Finally consider $\mathscr{T}(\Xi)$. Let $j=1, j^{'}=2$ and take $\tau_1\in C_2\circ \Xi_1\backslash 10\Xi_1, \tau_2\in C_2\circ \Xi_2\backslash 10\Xi_2$. One can estimate exactly as above. Then summing over $\mathscr{T}(\Xi)$ and all $\mathscr{T}(j, j^{'})$ completes the proof.
\end{proof}
Now we estimate $\mathbf{M}_j$. The following theorem is essentially the Littlewood-Paley theory. It's classical, but we include the proof.
\begin{theorem}  \label{Section4TheoremLittlewood-PaleyDyadic}
Let $C>1$ be a constant. Let $\mathscr{I}$ be a finite collection of dyadic intervals in $\mathbb{R}^d$ and $(\xi_I)_{I\in \mathscr{I}}$ be a family of points in $\mathbb{R}^d$ such that $\xi_I\in \big( -Cs(I)^{-1}, Cs(I)^{-1} \big)^d\backslash \big( -s(I)^{-1}, s(I)^{-1} \big)^d$. Let $(\varphi_{I})_{I\in \mathscr{I}}$ be a family of functions such that each $\varphi_{I}$ can be written into the following form:
\begin{equation*}
\varphi_{I}(x)=|I|^{-\frac{1}{2}}e^{2\pi i \xi_I\big(x-\mathfrak{c}(I) \big)}\zeta_{I} \big( \frac{x-\mathfrak{c}(I)}{s(I)}\big),
\end{equation*}
where $\zeta_{I}\in \mathcal{S}(\mathbb{R}^d)$ satisfies
\begin{enumerate}
\item The Fourier transform $\widehat{\zeta_{I}}$ is supported in $(-\frac{1}{2}, \frac{1}{2})^{d}$.
\item For any $\alpha: |\alpha|\leq 1$, we have
\begin{equation*}
|\partial^{\alpha} \zeta_{I}(x)|\lesssim (1+|x|)^{-10d}.
\end{equation*}
The implicit constant does not depend on $I$.
\end{enumerate}
Then the square function
\begin{equation*}
\mathfrak{s}f(x):=\Big( \sum_{I\in \mathscr{I}} \frac{|\langle f|\varphi_{I} \rangle|^2}{|I|}1_{I}(x) \Big)^{\frac{1}{2}}
\end{equation*}
is $L^2\rightarrow L^2$ bounded and $L^1\rightarrow L^{1, \infty}$ bounded.
\end{theorem}
\begin{proof}
We first prove the $L^2$ boundedness. By dividing $\mathscr{I}$ into $O(1)$ groups, we can assume $\langle \varphi_{I}|\varphi_{I^{'}}\rangle=0$ if $s(I)\neq s(I^{'})$. Consider $A: f\mapsto (\langle f|\varphi_{I}\rangle)_{I\in \mathscr{I}}$. It suffices to show $\|A^{*}Af\|_2\lesssim \|Af\|_2$, which means
\begin{equation*}
\sum_{I, I^{'}} \langle f|\varphi_{I}\rangle \overline{\langle f|\varphi_{I^{'}}\rangle} \langle \varphi_{I}|\varphi_{I^{'}}\rangle \lesssim \sum_{I} |\langle f|\varphi_{I}\rangle|^2
\end{equation*}
By Schur's test and orthogonality, it suffices to show
\begin{equation*}
\sum_{I^{'}: s(I^{'})=s(I)} |\langle \varphi_{I}|\varphi_{I^{'}}\rangle|\lesssim 1
\end{equation*}
for any $I$. But this is an easy consequence of the estimate:
\begin{equation}  \label{Dection4EquationSupportPerturbationLemma}
|\langle \varphi_{I}|\varphi_{I^{'}}\rangle|\lesssim 
 \big(1+\frac{d(I, I^{'})}{s(I)} \big)^{-5d}.
\end{equation}
Now we use the C-Z decomposition \cite{Stein} to prove the $L^1\rightarrow L^{1, \infty}$ boundedness. Given $\lambda>0$, write $f=g+b:=g+\sum_{Q} b_Q$. We have
\begin{equation*}
|\{\mathfrak{s}f>\lambda\}|\leq |\{\mathfrak{s}g>\frac{\lambda}{2}\}|+|\{\mathfrak{s}b>\frac{\lambda}{2}\}|.  
\end{equation*}
For the first term on the right, we estimate
\begin{equation*}
|\{\mathfrak{s}g>\frac{\lambda}{2}\}|\lesssim \frac{\|\mathfrak{s}g\|_2}{\lambda^2}\lesssim \frac{\|g\|_2}{\lambda^2}\lesssim \frac{\|f\|_1}{\lambda}.
\end{equation*}
For the second term, it suffices to show
\begin{equation}  \label{Section4EquationC-ZIntermidiateStep1}
|\{x\in \Big(\bigcup_{Q} C_d Q \Big)^c: \mathfrak{s}b>\frac{\lambda}{2}\}\lesssim \frac{\|f\|_1}{\lambda}
\end{equation}
since
\begin{equation*}
\sum_{Q} |Q|\lesssim \frac{\|f\|_1}{\lambda}.
\end{equation*}
The estimate \eqref{Section4EquationC-ZIntermidiateStep1} is implied by
\begin{equation*}
\int_{(C_dQ)^c} |\mathfrak{s}b_Q|\lesssim \|b_Q\|_1,\ \forall Q.
\end{equation*}
To prove this, write
\begin{equation}  \label{Section4EquationC-ZIntermidiateStep2}
\int_{(C_dQ)^c} |\mathfrak{s}b_Q|=\int_{(C_dQ)^c} \Big( \sum_{I\in \mathscr{I}} \frac{|\langle b_Q|\varphi_{I} \rangle|^2}{|I|}1_{I}(x) \Big)^{\frac{1}{2}}\leq \int_{(C_dQ)^c} \sum_{I\in \mathscr{I}} \frac{|\langle b_Q|\varphi_{I} \rangle|}{|I|^{\frac{1}{2}}}1_{I}(x).
\end{equation}
We can restrict $I\cap (C_dQ)^c\neq \emptyset$. First consider the sum over $I: s(I)\leq s(Q)$ in RHS of \eqref{Section4EquationC-ZIntermidiateStep2}. The cancellation property doesn't play a role in this case, so we estimate
\begin{equation*}
|\langle b_Q\big{|} |I|^{\frac{1}{2}}\varphi_{I} \rangle|\lesssim \big(1+\frac{d(I, Q)}{s(I)} \big)^{-10d}\|b_Q\|_1.
\end{equation*}
Plug this into RHS of \eqref{Section4EquationC-ZIntermidiateStep2}. We then obtain
\begin{equation*}
\sum_{\substack{I: I\cap (C_dQ)^c\neq \emptyset,\\ s(I)\leq s(Q)}} \big(1+\frac{d(I, Q)}{s(I)} \big)^{-10d}\lesssim \sum_{2^k\leq s(Q)}\sum_{\substack{I: I\cap (C_dQ)^c\neq \emptyset,\\ s(I)=2^k}} \big(\frac{s(Q)}{2^k} \big)^{-5d}\big(1+\frac{d(I, Q)}{2^k} \big)^{-5d}\lesssim 1.
\end{equation*}
Then consider the sum over $I: s(I)>s(Q)$ in RHS of \eqref{Section4EquationC-ZIntermidiateStep2}. Using $\int b_Q=0$, we estimate
\begin{equation*}
|\langle b_Q\big{|} |I|^{\frac{1}{2}}\varphi_{I} \rangle|=|\langle b_Q\big{|} |I|^{\frac{1}{2}}\varphi_{I}(\cdot)-|I|^{\frac{1}{2}}\varphi_{I} \big(\mathfrak{c}(Q) \big) \rangle|\lesssim \frac{s(Q)}{s(I)} \big(1+\frac{d(I, Q)}{s(I)} \big)^{-10d}\|b_Q\|_1.
\end{equation*}
Plug this into RHS of \eqref{Section4EquationC-ZIntermidiateStep2}. We have
\begin{equation*}
\sum_{\substack{I: I\cap (C_dQ)^c\neq \emptyset,\\ s(I)>s(Q)}} \frac{s(Q)}{s(I)} \big(1+\frac{d(I, Q)}{s(I)} \big)^{-10d}=\sum_{2^k>s(Q)}\sum_{I: s(I)=2^k} \frac{s(Q)}{2^k} \big(1+\frac{d(I, Q)}{2^k}\big)^{-10d} \lesssim 1.
\end{equation*}
This completes the proof.
\end{proof}
\begin{corollary} \label{Section4CorollaryMassEstimate1}
Let $j\in [n]$. Let $\mathscr{T}_j$ be a lacunary $j\text{-tree}$. Then
\begin{equation*}
\frac{1}{|I_{\mathscr{T}_j}|}\big{\|} \big( \sum_{\mathcal{R}_j\in \mathscr{T}_j} \frac{|\langle f_j|\varphi_{\mathcal{R}_j} \rangle|^2}{I(\mathcal{R}_j)}1_{I(\mathcal{R}_j)} \big)^{\frac{1}{2}} \big{\|}_{1, \infty}\lesssim \frac{1}{|I_{\mathscr{T}_j}|}\int_{\mathbb{R}^d} |f_j| \big(1+\frac{d(x, I_{\mathscr{T}_j})}{s(I_{\mathscr{T}_j})} \big)^{-2N_{\epsilon}}\lesssim \inf_{x\in I_{\mathscr{T}_j}} Mf_j(x).
\end{equation*}
\end{corollary}
\begin{proof}
Pick $\xi_j\in \mathbb{R}^d$ such that $\xi_j\subset C_2\circ \Xi(\mathcal{R}_j)\backslash 10\Xi(\mathcal{R}_j), \forall \mathcal{R}_j\in \mathscr{T}_j$. Our $C_3\text{-sparse}$ property implies that the map $\mathcal{R}_j\mapsto I(\mathcal{R}_j)$ is a bijection in $\mathscr{T}_j$. For $I\in \{I(\mathcal{R_j}): \mathcal{R}_j\in \mathscr{T}_j\}$, denote the unique $\mathcal{R}_j\in \mathscr{T}_j$ such that $I(\mathcal{R}_j)=I$ by $\mathcal{R}_I$. Let
\begin{equation*}
\varphi_I(x):=\big(1+\big{|}\frac{x-\mathfrak{c}(I_{\mathscr{T}_j})}{s(I_{\mathscr{T}_j})}\big{|}^2 \big)^{N_{\epsilon}}e^{-2\pi i\xi_j(x-\mathfrak{c}(I))}\varphi_{\mathcal{R}_I}(x).
\end{equation*}
Since $\varphi_{\mathcal{R}_I}$ is a wave packet adapted to $\mathcal{R}_I$, we can write $\varphi_I$ into the form required by Theorem \ref{Section4TheoremLittlewood-PaleyDyadic}. In fact, assume
\begin{equation*}
\varphi_{\mathcal{R}_I}(x)=|I|^{-\frac{1}{2}}e^{2\pi i\mathfrak{c} \big(\Xi(\mathcal{R}_I) \big)\big(x-\mathfrak{c}(I) \big)}\zeta_{\mathcal{R}_I} \big( \frac{x-\mathfrak{c}(I)}{s(I)}\big).
\end{equation*}
We can take
\begin{equation*}
\zeta_I(x):=\big(1+\big{|}\frac{s(I)}{s(I_{\mathscr{T}_j})}x+\frac{\mathfrak{c}(I)-\mathfrak{c}(I_{\mathscr{T}_j})}{s(I_{\mathscr{T}_j})} \big{|}^2 \big)^{N_{\epsilon}}\zeta_{\mathcal{R}_I}(x)
\end{equation*}
with $\xi_I:=\mathfrak{c} \big(\Xi(\mathcal{R}_I) \big)-\xi_j$. It easy to check $\zeta_I$ satisfies all the requirements in Theorem \ref{Section4TheoremLittlewood-PaleyDyadic} since $1\ll N_{\epsilon}\ll N_1$ and $\partial^{\alpha} \widehat{\zeta_{\mathcal{R}_I}}$ is still supported in $(-\frac{1}{2}, \frac{1}{2})^d$. Now let
\begin{equation*}
f(x):=\big(1+\big{|}\frac{x-\mathfrak{c}(I_{\mathscr{T}_j})}{s(I_{\mathscr{T}_j})}\big{|}^2 \big)^{-N_{\epsilon}}e^{-2\pi i\xi_j(x-\mathfrak{c}(I))}f_j(x). 
\end{equation*}
We have $\langle f|\varphi_I \rangle=\langle f_j|\varphi_{\mathcal{R}_I} \rangle$. Applying the $L^1\rightarrow L^{1, \infty}$ bound in Theorem \ref{Section4TheoremLittlewood-PaleyDyadic} gives the first estimate. The second estimate is standard.
\end{proof}
There is still a difference between $\mathbf{M}_j$ and LHS of Corollary \ref{Section4CorollaryMassEstimate1}: One need to replace the $L^{1, \infty}$ norm in Corollary \ref{Section4CorollaryMassEstimate1} with the $L^2$ norm in $\mathbf{M}_j$. This is achieved by the now-standard dyadic John-Nirenberg inequality \cite{MTTBiest1}.
\begin{definition}
Let $\mathscr{I}$ be a finite collection of dyadic intervals in $\mathbb{R}^d$ and $(a_I)_{I\in\mathscr{I}}$ be a family of complex numbers. Define
\begin{align*}
\|(a_I)_{I\in\mathscr{I}}\|_p & :=\big{\|}\big(\sum_{I\in \mathscr{I}} \frac{|a_I|^2}{|I|}1_I\big)^{\frac{1}{2}}\big{\|}_p,\ 0<p<\infty,  \\
\|(a_I)_{I\in\mathscr{I}}\|_{bmo} & :=\sup_{I_0\in \mathscr{I}} \frac{1}{|I_0|^{\frac{1}{2}}}\big{\|}\big(\sum_{I\in \mathscr{I}: I\subset I_0} \frac{|a_I|^2}{|I|}1_I\big)^{\frac{1}{2}}\big{\|}_2.
\end{align*}
\end{definition}
\begin{lemma}[Dyadic John-Nirenberg Inequality]  \label{Section4LemmaDyadicJohn-Nirenberg}
Let $\mathscr{I}$ be a finite collection of dyadic intervals in $\mathbb{R}^d$ and $(a_I)_{I\in\mathscr{I}}$ be a family of complex numbers. Then for any $0<p, q<\infty$, we have
\begin{equation*}
\sup_{I_0\in \mathscr{I}} \frac{1}{|I_0|^{\frac{1}{p}}}\big{\|}\big(\sum_{I\in \mathscr{I}: I\subset I_0} \frac{|a_I|^2}{|I|}1_I\big)^{\frac{1}{2}}\big{\|}_{p}\sim \sup_{I_0\in \mathscr{I}} \frac{1}{|I_0|^{\frac{1}{q}}}\big{\|}\big(\sum_{I\in \mathscr{I}: I\subset I_0} \frac{|a_I|^2}{|I|}1_I\big)^{\frac{1}{2}}\big{\|}_{q, \infty}.
\end{equation*}
\end{lemma}
\begin{remark}
Lemma \ref{Section4LemmaDyadicJohn-Nirenberg} immediately gives
\begin{equation*}
\sup_{I_0\in \mathscr{I}} \frac{1}{|I_0|^{\frac{1}{p}}}\big{\|}\big(\sum_{I\in \mathscr{I}: I\subset I_0} \frac{|a_I|^2}{|I|}1_I\big)^{\frac{1}{2}}\big{\|}_{p}\sim \sup_{I_0\in \mathscr{I}} \frac{1}{|I_0|^{\frac{1}{q}}}\big{\|}\big(\sum_{I\in \mathscr{I}: I\subset I_0} \frac{|a_I|^2}{|I|}1_I\big)^{\frac{1}{2}}\big{\|}_{q}
\end{equation*}
for any $0<p, q<\infty$. Hence the $L^2$ norm used in the definition of $\|\cdot\|_{bmo}$ is not important.
\end{remark}
\begin{proof}
We only need to prove
\begin{equation*}
\sup_{I_0\in \mathscr{I}} \frac{1}{|I_0|^{\frac{1}{p}}}\big{\|}\big(\sum_{I\in \mathscr{I}: I\subset I_0} \frac{|a_I|^2}{|I|}1_I\big)^{\frac{1}{2}}\big{\|}_{p}\lesssim \sup_{I_0\in \mathscr{I}} \frac{1}{|I_0|^{\frac{1}{q}}}\big{\|}\big(\sum_{I\in \mathscr{I}: I\subset I_0} \frac{|a_I|^2}{|I|}1_I\big)^{\frac{1}{2}}\big{\|}_{q, \infty}
\end{equation*}
for $0<q<p<\infty$. Assume $\text{RHS}=1$. Denote the supremum in LHS by $C(p)$, and assume it is obtained at $I_{*}$. We now estimate
\begin{equation}  \label{Section4EquationDyadicJohn-NirenbergIntermidiate1}
\int_{I_{*}} \big(\sum_{I\in \mathscr{I}: I\subset I_{*}} \frac{|a_I|^2}{|I|}1_I\big)^{\frac{p}{2}}.
\end{equation}
Pick a large constant $C$. Define
\begin{equation*}
\Omega:=\{x\in \mathbb{R}^d: \big(\sum_{I\in \mathscr{I}: I\subset I_{*}} \frac{|a_I|^2}{|I|}1_I(x)\big)^{\frac{1}{2}}>C\}.
\end{equation*}
Then the previous assumption implies
\begin{equation}  \label{Section4EquationDyadicJohn-NirenbergIntermidiate7}
|\Omega|\leq \frac{|I_{*}|}{C^q}  
\end{equation}
For $x\notin \Omega$, we estimate
\begin{equation}  \label{Section4EquationDyadicJohn-NirenbergIntermidiate2}
\int_{I_{*}\cap \Omega^c} \big(\sum_{I\in \mathscr{I}: I\subset I_{*}} \frac{|a_I|^2}{|I|}1_I\big)^{\frac{p}{2}}\leq C^p|I_{*}|.
\end{equation}
For $x\in \Omega$, consider
\begin{equation*}
\mathscr{J}:=\{J\in \mathscr{I}: \big(\sum_{I\in \mathscr{I}: I\subset I_{*}} \frac{|a_I|^2}{|I|}1_I(x)\big)^{\frac{1}{2}}>C,\ \forall x\in J\}.
\end{equation*}
Since the square function is constant on those smallest intervals in $\mathscr{I}$, we have $\Omega=\bigcup_{J\in \mathscr{J}} J$. Denote the collection of maximal elements in $\mathscr{J}$ by $\mathscr{J}_{max}$. We also have
\begin{equation}  \label{Section4EquationDyadicJohn-NirenbergIntermidiate3}
\Omega=\bigcup_{J_{max}\in \mathscr{J}_{max}} J_{max}.
\end{equation}
It suffices to estimate the integrals over each $J_{max}$. Write
\begin{equation}  \label{Section4EquationDyadicJohn-NirenbergIntermidiate4}
\int_{J_{max}} \big(\sum_{I\in \mathscr{I}: I\subset I_{*}} \frac{|a_I|^2}{|I|}1_I\big)^{\frac{p}{2}}\lesssim \int_{J_{max}} \big(\sum_{I\in \mathscr{I}: I\subset J_{max}} \frac{|a_I|^2}{|I|}1_I\big)^{\frac{p}{2}}+\int_{J_{max}} \big(\sum_{I\in \mathscr{I}: J_{max}\subsetneqq I\subset I_{*}} \frac{|a_I|^2}{|I|}1_I\big)^{\frac{p}{2}}    
\end{equation}
The second term of RHS can be handled by the definition of $J_{max}$: Let $I_0$ be the smallest element in $\{I\in \mathscr{I}: J_{max}\subsetneqq I\}$. The maximal property of $J_{max}$ implies the existence of $y\in I_0\backslash J_{max}$ satisfying
\begin{equation*}
\sum_{I\in \mathscr{I}: I\subset I_{*}} \frac{|a_I|^2}{|I|}1_I(y)\leq C.
\end{equation*}
Thus
\begin{equation*}
\big(\sum_{I\in \mathscr{I}: J_{max}\subsetneqq I\subset I_{*}} \frac{|a_I|^2}{|I|}1_I(x)\big)^{\frac{1}{2}}=\big(\sum_{I\in \mathscr{I}: J_{max}\subsetneqq I\subset I_{*}} \frac{|a_I|^2}{|I|}1_I(y)\big)^{\frac{1}{2}}\leq C,\ \forall x\in J_{max},
\end{equation*}
and we get
\begin{equation}  \label{Section4EquationDyadicJohn-NirenbergIntermidiate5}
\int_{J_{max}} \big(\sum_{I\in \mathscr{I}: J_{max}\subsetneqq I\subset I_{*}} \frac{|a_I|^2}{|I|}1_I(x)\big)^{\frac{p}{2}}\leq C^p|J_{max}|.
\end{equation}
We can estimate the first term on the right of \eqref{Section4EquationDyadicJohn-NirenbergIntermidiate4} as follows:
\begin{equation}  \label{Section4EquationDyadicJohn-NirenbergIntermidiate6}
\int_{J_{max}} \big(\sum_{I\in \mathscr{I}: I\subset I_{*}} \frac{|a_I|^2}{|I|}1_I\big)^{\frac{p}{2}}\leq C(p)^p|J_{max}|.
\end{equation}
Combining \eqref{Section4EquationDyadicJohn-NirenbergIntermidiate2}, \eqref{Section4EquationDyadicJohn-NirenbergIntermidiate4}\eqref{Section4EquationDyadicJohn-NirenbergIntermidiate5}\eqref{Section4EquationDyadicJohn-NirenbergIntermidiate6}, and \eqref{Section4EquationDyadicJohn-NirenbergIntermidiate3}\eqref{Section4EquationDyadicJohn-NirenbergIntermidiate7}, we obtain
\begin{equation*}
C(p)^p|I_{*}|\leq C^{p}|I_{*}|+C^{'} \big(\frac{C^p|I_{*}|}{C^q}+\frac{C(p)^p|I_{*}|}{C^q} \big),
\end{equation*}
where $C^{'}$ is a constant depending on $p$. Taking $C$ sufficiently large gives $C(p)\lesssim 1$ as desired.
\end{proof}
\begin{corollary}  \label{Section4CorollaryMassEstimate2}
Let $j\in [n]$ and $\mathscr{P}_j\subset \{\mathcal{R}_j: \mathcal{R}\in \mathscr{R}\}$ be a non-empty collection. Then
\begin{equation*}
\mathbf{M}_j(\mathscr{P}_j)\lesssim \sup_{\mathcal{R}_j\in \mathscr{P}_j} \frac{1}{|I(\mathcal{R}_j)|}\int_{\mathbb{R}^d} |f_j| \big(1+\frac{d \big(x, I(\mathcal{R}_j) \big)}{s \big(I(\mathcal{R}_j) \big)} \big)^{-2N_{\epsilon}}\lesssim \sup_{\mathcal{R}_j\in \mathscr{P}_j} \inf_{x\in I(\mathcal{R}_j)} Mf_j(x).
\end{equation*}
\end{corollary}
\begin{proof}
Inside a lacunary $j\text{-tree}$ $\mathscr{T}_j$, the map $\mathcal{R}_j\mapsto I(\mathcal{R}_j)$ is a bijection. And $\{\mathcal{R}_j\in \mathscr{T}_j: I(\mathcal{R}_j)\subset I_0\}$ is also a lacunary $j\text{-tree}$ if $I_0\in \{I(\mathcal{R}_j): \mathcal{R}_j\in \mathscr{T}_j\}$. This implies that the supremum in $\mathbf{M}_j(\mathscr{P}_j)$ can always be obtained at a lacunary $j\text{-tree}$ $\widetilde{\mathscr{T}_j}\subset \mathscr{P}_j$ such that $I_{\widetilde{\mathscr{T}_j}}=I(\mathcal{R}_j)$ for some $\mathcal{R}_j\in \widetilde{\mathscr{T}_j}$. It then suffices to control
\begin{equation*}
\big(\frac{1}{|I_{\widetilde{\mathscr{T}_j}}|} \sum_{\mathcal{R}_j\in \widetilde{\mathscr{T}_j}} |\langle f_j|\varphi_{\mathcal{R}_j} \rangle|^2 \big)^{\frac{1}{2}}=\frac{1}{|I_{\widetilde{\mathscr{T}_j}}|^{\frac{1}{2}}} \big{\|}\big(\sum_{\mathcal{R}_j\in\widetilde{\mathscr{T}_j}} \frac{|\langle f_j|\varphi_{\mathcal{R}_j} \rangle|^2}{|I(\mathcal{R}_j)|}1_{I(\mathcal{R}_j)} \big)^{\frac{1}{2}}\big{\|}_2.
\end{equation*}
Combing Lemma \ref{Section4LemmaDyadicJohn-Nirenberg} and Corollary \ref{Section4CorollaryMassEstimate1} gives the desired estimate.
\end{proof}

We sill also need an interpolation result between $\|\cdot\|_2$ and $\|\cdot\|_{bmo}$. The main idea appears implicitly in \cite{LaceyLittlewood-PaleyForArbitraryIntervals}.
\begin{lemma}  \label{Section4LemmaInterpolation}
Let $\mathscr{I}$ be a finite collection of dyadic intervals in $\mathbb{R}^d$ and $(a_I)_{I\in\mathscr{I}}$ be a family of complex numbers. Then for $2\leq p<\infty$, we have
\begin{equation*}
\|(a_I)_{I\in\mathscr{I}}\|_{p}\lesssim \|(a_I)_{I\in\mathscr{I}}\|_2^{\frac{2}{p}}\|(a_I)_{I\in\mathscr{I}}\|_{bmo}^{1-\frac{2}{p}}. 
\end{equation*}
\end{lemma}
\begin{proof}
We can assume each $a_I\neq 0$, otherwise delete $I$ from $\mathscr{I}$. Let $\lambda_0$ be the smallest dyadic number larger than $|(a_I)_{I\in\mathscr{I}}\|_{bmo}$ and $\mathscr{I}_0:=\mathscr{I}, \mathscr{J}_0\neq \emptyset$ be the initial collections. Run the following algorithm with $\lambda_{n+1}:=\frac{\lambda_n}{2}, n\geq 0$:
\begin{description}
\item[Step 1] Consider all maximal $J_n\in \mathscr{I}_n$ such that
\begin{equation}  \label{Section4EquationInterpolation1}
\frac{1}{|J_n|^{\frac{1}{2}}}\big{\|}\big(\sum_{I\in \mathscr{I}_n: I\subset J_n} \frac{|a_I|^2}{|I|}1_I\big)^{\frac{1}{2}}\big{\|}_2>\lambda_{n+1}.
\end{equation}
Denote the collection of these $J_n$'s by $\mathscr{J}_n$.
\item[Step 2] Renew $\mathscr{I}_{n+1}:=\mathscr{I}_{n}\backslash \{I\in \mathscr{I}_{n}: I\subset J_n \text{\ for some\ }J_n\in \mathscr{J}_n\}$. Then return to Step 1 to proceed the $(n+1)\text{-th}$ iteration.
\end{description}
We get $\mathscr{I}_N=\emptyset$ after finitely many iterations. And the estimate
\begin{equation}  \label{Section4EquationInterpolation2}
\|(a_I)_{I\in\mathscr{I}_{n}}\|_{bmo}\leq \lambda_n,\ \forall 0\leq n\leq N
\end{equation}
holds. Now we have
\begin{equation}  \label{Section4EquationInterpolation3}
\|(a_I)_{I\in\mathscr{I}}\|_{p}\leq \sum_{n=0}^N \|(a_I)_{I\in\mathscr{I}_{n}\backslash \mathscr{I}_{n+1}}\|_{p}=\sum_{n=0}^N \Big(\sum_{J_n\in \mathscr{J}_n}\int_{J_n} \big(\sum_{I\in\mathscr{I}_{n}\backslash \mathscr{I}_{n+1}: I\subset J_n} \frac{|a_I|^2}{|I|}1_I \big)^{\frac{p}{2}} \Big)^{\frac{1}{p}}
\end{equation}
since $J_n$'s are disjoint. Combining Lemma \ref{Section4LemmaDyadicJohn-Nirenberg} and \eqref{Section4EquationInterpolation2} gives
\begin{equation}  \label{Section4EquationInterpolation4}
\sum_{J_n\in \mathscr{J}_n}\int_{J_n} \big(\sum_{I\in\mathscr{I}_{n}\backslash \mathscr{I}_{n+1}: I\subset J_n} \frac{|a_I|^2}{|I|}1_I \big)^{\frac{p}{2}}\lesssim \lambda_n^p \sum_{J_n\in \mathscr{J}_n} |J_n|.
\end{equation}
On the other hand, \eqref{Section4EquationInterpolation1} implies
\begin{equation}  \label{Section4EquationInterpolation5}
\sum_{J_n\in \mathscr{J}_n} |J_n|\leq \frac{\|(a_I)_{I\in\mathscr{I}}\|_2^2}{\lambda_{n+1}^2}\sim \frac{\|(a_I)_{I\in\mathscr{I}}\|_2^2}{\lambda_n^2}
\end{equation}
Plug \eqref{Section4EquationInterpolation5} into \eqref{Section4EquationInterpolation4}, and then plug \eqref{Section4EquationInterpolation4} into \eqref{Section4EquationInterpolation3}. Summing over $n$ completes the proof.
\end{proof}
\begin{definition}
Let $\mathbf{F}_j$ be a family of $j\text{-trees}$. Define the associated counting function:
\begin{equation*}
N_{\mathbf{F}_j}(x):=\sum_{\mathscr{T}_j\in \mathbf{F}_j} 1_{I_{\mathscr{T}_j}}(x).
\end{equation*}
Similarly, let $\mathbf{F}$ be a family of vector trees. Define the associated counting function:
\begin{equation*}
N_{\mathbf{F}}(x):=\sum_{\mathscr{T}\in \mathbf{F}} 1_{I_{\mathscr{T}}}(x).
\end{equation*}
\end{definition}
We are ready to divide $\{\mathcal{R}_j: \mathcal{R}\in \mathscr{R}\}$ into $j\text{-trees}$ with well controlled counting functions.
\begin{proposition}  \label{Section4PropositionMassSelection1}
Let $j\in [n]$, $\mathscr{P}_j\subset \{\mathcal{R}_j: \mathcal{R}\in \mathscr{R}\}$, and $\lambda_j>0$. Suppose $\mathbf{M}_j(\mathscr{P}_j)\leq \lambda_j$. Then there exists a partition $\mathscr{P}_j=\mathscr{P}_j^{'}\uplus \mathscr{P}_j^{''}$ such that:
\begin{enumerate}
\item The collection $\mathscr{P}_j^{'}$ is the disjoint union of a family of $j\text{-trees}$. Denote this family by $\mathbf{F}_j(\lambda_j):=\{\mathscr{T}_j\}$. Then the counting function $N_{\mathbf{F}_j(\lambda_j)}$ satisfies
\begin{align} \label{Section4EquationCountingFunctionEstimate}
\|N_{\mathbf{F}_j(\lambda_j)}\|_{p_j} & \lesssim \lambda_j^{-2}|E_j|^{\frac{1}{p_j}},\ \forall 1\leq p_j<\infty,  \\  \label{Section4EquationCountingFunctionEstimateBelow1}
\|N_{\mathbf{F}_j(\lambda_j)}\|_{p_j} & \lessapprox \lambda_j^{-(1+\frac{1}{p_j})}|E_j|^{\frac{1}{p_j}},\ \forall 0<p_j<1.
\end{align}
\item The collection $\mathscr{P}_j^{''}$ satisfies
\begin{equation}  \label{Section4EquationMassDecrease}
\mathbf{M}_j(\mathscr{P}_j^{''})\leq \frac{\lambda_j}{2}.
\end{equation}
\end{enumerate}
\end{proposition}
\begin{remark}
Inequality \eqref{Section4EquationCountingFunctionEstimate} appeared in \cite{KTZKTriangularHilbertDyadicPartialResult}. It improves the counting function estimate in \cite{DPTFractionalRank}. Inequality \eqref{Section4EquationCountingFunctionEstimateBelow1} appeared in \cite{LTBHT2}. Here we give a simpler proof. We only need \eqref{Section4EquationCountingFunctionEstimate} in the proof of Theorem \ref{MainTheorem1}, while \eqref{Section4EquationCountingFunctionEstimateBelow1} will be used in Section \ref{Section8} to simplify our argument for certain Lebesgue space exponents.
\end{remark}
Iterate Proposition \ref{Section4PropositionMassSelection1} with an initial dyadic number $\lambda_j\sim \mathbf{M}_j(\mathscr{P}_j)$. We obtain the following:
\begin{corollary}  \label{Section4PropositionMassSelection2}
Let $j\in [n]$ and $\mathscr{R}_j:=\{\mathcal{R}_j: \mathcal{R}\in \mathscr{R}\}$. We can write
\begin{equation*}
\mathscr{R}_j=\biguplus_{\substack{\lambda_j\text{\ dyadic},\\ \lambda_j\leq \mathbf{M}_j(\mathscr{R}_j)}}\mathscr{R}_j(\lambda_j):=\biguplus_{\substack{\lambda_j\text{\ dyadic},\\ \lambda_j\leq \mathbf{M}_j(\mathscr{R}_j)}}\biguplus_{\mathscr{T}_j\in \mathbf{F}_j(\lambda_j)} \mathscr{T}_j,
\end{equation*}
such that $\mathbf{M}_j\big(\mathscr{R}_j(\lambda_j) \big)\leq \lambda_j$ and \eqref{Section4EquationCountingFunctionEstimate} hold for each $\lambda_j$.
\end{corollary}
The rest of this section is devoted to proving Proposition \ref{Section4PropositionMassSelection1}. The arguments are classical \cite{MTTMultilinearSingularMultipliers}\cite{DPTFractionalRank}\cite{KTZKTriangularHilbertDyadicPartialResult}. For this part, we omit the index $j$ for simplicity. Hence tree means $j\text{-tree}$, $\mathscr{P}$ means $\mathscr{P}_j$, $|f|\leq 1_E$ for some bounded set $E$, etc. Some results below hold for general $f\in L^2(\mathbb{R}^d)$, but working with $|f|\leq 1_E$ is sufficient for our purpose. We begin with two further definitions.
\begin{definition}
Let $k\in [d]$. A lacunary tree $\mathscr{T}$ is called a $(k, +)\text{-lacunary}$ tree if there exits $\xi\in \mathbb{R}^d$ such that $\xi\in C_2\circ \Xi(\mathcal{R})$ and
\begin{equation*}
\xi_k\notin 10\Xi(\mathcal{R})_k,\ \xi_k<\mathfrak{c}\big( \Xi(\mathcal{R})_k \big)
\end{equation*}
holds for any $\mathcal{R}\in \mathscr{T}$. Similarly, we define $(k, -)\text{-lacunary}$ tree by replacing $\xi_k<\mathfrak{c}\big( \Xi(\mathcal{R})_k \big)$ with $\xi_k>\mathfrak{c}\big( \Xi(\mathcal{R})_k \big)$.
\end{definition}
Any lacunary tree can be divided into the disjoint union of $(k, \pm)\text{-lacunary}$ trees.
\begin{definition}
A family of disjoint lacunary trees $\mathbf{F}$ is strongly disjoint if for any $\mathscr{T}\neq \mathscr{T}^{'}\in \mathbf{F}, \mathcal{R}\in \mathscr{T}, \mathcal{R}^{'}\in \mathscr{T}^{'}$ such that $s \big(\Xi(\mathcal{R}) \big)<s\big(\Xi(\mathcal{R}^{'}) \big)$, the property $\Xi(\mathcal{R})\cap \Xi(\mathcal{R}^{'})\neq \emptyset$ implies $I(\mathcal{R}^{'})\cap I_{\mathscr{T}}=\emptyset$.
\end{definition}
We can deduce Proposition \ref{Section4PropositionMassSelection1} from the following two results.
\begin{proposition}  \label{Section4PropositionMassSelectionAlgorithm}
Let $\mathscr{P}$ and $\lambda$ as in Proposition \ref{Section4CorollaryMassEstimate1}. Let $k\in [d]$. We can take disjoint trees $\mathscr{T}_1, \widetilde{\mathscr{T}}_1, \cdots, \mathscr{T}_L, \widetilde{\mathscr{T}}_L\subset \mathscr{P}$ such that 
\begin{enumerate}
\item We have that $\mathscr{T}_1, \mathscr{T}_2, \cdots, \mathscr{T}_L$ are $(k, +)\text{-lacunary}$ trees. And
\begin{equation}  \label{Section4EquationTreeHasLargeMass}
\frac{1}{|I_{\mathscr{T}}|} \big{\|}\big(\sum_{\mathcal{R}\in \mathscr{T}} \frac{|\langle f|\varphi_{\mathcal{R}} \rangle|^2}{|I(\mathcal{R})|}1_{I(\mathcal{R})} \big)^{\frac{1}{2}}\big{\|}_{1, \infty}>c_d \lambda
\end{equation}
holds for $\mathscr{T}=\mathscr{T}_l, \forall 1\leq l\leq L$.
\item The family $\mathbf{F}:=\{\mathscr{T}_1, \mathscr{T}_2, \cdots, \mathscr{T}_L\}$ is strongly disjoint.
\item The relation $I_{\widetilde{\mathscr{T}}_l}\subset I_{\mathscr{T}_l}$ holds for any $1\leq l\leq L$.
\item The remaining collection $\mathscr{P}\backslash \Big(\bigcup_{l=1}^L(\mathscr{T}_l\cup \widetilde{\mathscr{T}}_l) \Big)$ doesn't contain any $(k, +)\text{-lacunary}$ tree $\mathscr{T}$ such that \eqref{Section4EquationTreeHasLargeMass} holds.
\end{enumerate}
Here $c_d>0$ is a fixed small constant. The same statement holds for $(k, -)\text{-lacunary}$ trees.
\end{proposition}
\begin{proposition}  \label{Section4PropositionMassSelectionBesselTypeEstimate}
Suppose $\varphi_{\mathcal{R}}$'s satisfy all the requirements in Definition \ref{Section3DefinitionWavePacket} but $N_1$ in \eqref{Section3EquationWavePacketDecay} replaced by $10d$. Let $\mathbf{F}$ be a family of disjoint lacunary trees. Assume $\mathbf{F}$ is strongly disjoint and
\begin{equation}  \label{Section4EquationRegularMassAssumption}
\sup_{\mathcal{R}\in \bigcup_{\mathscr{T}\in \mathbf{F}} \mathscr{T}} \frac{|\langle f|\varphi_{\mathcal{R}} \rangle|}{|I(\mathcal{R})|^{\frac{1}{2}}}\lesssim \inf_{\mathscr{T}\in \mathbf{F}} \frac{1}{|I_{\mathscr{T}}|^{\frac{1}{2}}} (\sum_{\mathcal{R}\in \mathscr{T}} |\langle f|\varphi_{\mathcal{R}} \rangle|^2)^{\frac{1}{2}}.
\end{equation}
Then we have the Bessel type estimate:
\begin{equation}  \label{Secteion4EquationFundamentalBesselTypeEstimate}
\sum_{\mathscr{T}\in \mathbf{F}}\sum_{\mathcal{R}\in \mathscr{T}} |\langle f|\varphi_{\mathcal{R}} \rangle|^2\lesssim \|f\|_2^2.
\end{equation}
\end{proposition}
Since the conclusion holds under the weaker decay condition, we immediately get its local version by the same argument in the proof of Corollary \ref{Section4CorollaryMassEstimate1}.
\begin{corollary}  \label{Section4CorollaryMassSelectionBesselTypeEstimateLocalVersion}
Suppose $\varphi_{\mathcal{R}}$'s are wave packets satisfying all the requirements in Definition \ref{Section3DefinitionWavePacket}. Let $\mathbf{F}$ be a family of disjoint lacunary trees. Assume $\mathbf{F}$ is strongly disjoint and \eqref{Section4EquationRegularMassAssumption} holds. Then for any dyadic interval $J\subset \mathbb{R}^d$, we have
\begin{equation*}
\sum_{\mathscr{T}\in \mathbf{F}: I_{\mathscr{T}}\subset J}\sum_{\mathcal{R}\in \mathscr{T}} |\langle f|\varphi_{\mathcal{R}} \rangle|^2\lesssim \int_{\mathbb{R}^d} |f|^2\big(1+\frac{d(x, J)}{s(J)} \big)^{-20d}.
\end{equation*}
\end{corollary}
We first prove Proposition \ref{Section4PropositionMassSelection1} using Proposition \ref{Section4PropositionMassSelectionAlgorithm}, Proposition \ref{Section4PropositionMassSelectionBesselTypeEstimate}, and Corollary \ref{Section4CorollaryMassSelectionBesselTypeEstimateLocalVersion}.
\begin{proof}
Apply Proposition \ref{Section4PropositionMassSelectionAlgorithm} for each type of lacunary trees: $(1, \pm), (2, \pm), \cdots, (d, \pm)$. (The order doesn't matter.) For each type, we use Proposition \ref{Section4PropositionMassSelectionAlgorithm} to remove some trees $\mathscr{T}_1, \widetilde{\mathscr{T}}_1, \cdots, \mathscr{T}_L, \widetilde{\mathscr{T}}_L$ and renew $\mathscr{P}$ into $\mathscr{P}\backslash \Big(\bigcup_{l=1}^L(\mathscr{T}_l\cup \widetilde{\mathscr{T}}_l) \Big)$ for the next type. After applying Proposition \ref{Section4PropositionMassSelectionAlgorithm} for each type once, denote the union of all trees taken out by $\mathscr{P}^{'}$ and the remaining collection by $\mathscr{P}^{''}$. Then \eqref{Section4EquationMassDecrease} holds. Otherwise we 
divide the tree violating \eqref{Section4EquationMassDecrease} into the disjoint union of $(k, \pm)\text{-lacunary}$ trees. Then we find some $k$ and a $(k, +)\text{-lacunary}$ or $(k, -)\text{-lacunary}$ tree $\mathscr{T}\subset \mathscr{P}^{''}$ satisfying \eqref{Section4EquationTreeHasLargeMass}, a contradiction. Let $\mathscr{P}^{'}:=\mathscr{P}\backslash \mathscr{P}^{''}$. To show \eqref{Section4EquationCountingFunctionEstimate} and \eqref{Section4EquationCountingFunctionEstimateBelow1}, we only need to prove the same estimate for each $\mathbf{F}$ in the second property of Proposition \ref{Section4PropositionMassSelectionAlgorithm}. The inequality \eqref{Section4EquationTreeHasLargeMass} implies
\begin{equation*}
\sum_{l=1}^L |I_{\mathscr{T}_l}|\lesssim \frac{1}{\lambda^2} \sum_{l=1}^L \sum_{\mathcal{R}\in \mathscr{T}_l} |\langle f|\varphi_{\mathcal{R}} \rangle|^2.
\end{equation*}
Combining this with Proposition \ref{Section4PropositionMassSelectionBesselTypeEstimate} gives
\begin{equation}  \label{Section4EquationProvingCountingFunctionEstimate1}
\sum_{l=1}^L |I_{\mathscr{T}_l}|\lesssim \lambda^{-2}\|f\|_2^2\leq \lambda^{-2}|E|,
\end{equation}
where we used $|f|\leq 1_E$. Similarly, use Corollary \ref{Section4CorollaryMassSelectionBesselTypeEstimateLocalVersion} instead. We get
\begin{equation} \label{Section4EquationProvingCountingFunctionEstimate2}
\frac{1}{|J|}\sum_{l: I_{\mathscr{T}_l}\subset J} |I_{\mathscr{T}_l}|\lesssim \lambda^{-2}\frac{1}{|J|} \int_{\mathbb{R}^d} |f|^2\big(1+\frac{d(x, J)}{s(J)} \big)^{-20d}\lesssim \lambda^{-2}.
\end{equation}
Let $\mathscr{I}:=\{I_{\mathscr{T}_l}: 1\leq l\leq L\}$. We can write
\begin{equation*}
\sum_{l: I_{\mathscr{T}_l}\subset J}1_{\mathscr{T}_l}(x)=\sum_{I\in \mathcal{I}: I\subset J} \frac{|a_I|^2}{|I|}1_I(x)
\end{equation*}
with $a_I:=(|I|\#\{1\leq l\leq L: I_{\mathscr{T}_l}=I\})^{\frac{1}{2}}$. Thus \eqref{Section4EquationProvingCountingFunctionEstimate1} and \eqref{Section4EquationProvingCountingFunctionEstimate1} imply
\begin{gather*}
\|(a_I)_{I\in \mathscr{I}}\|_2 \lesssim \lambda^{-1}|E|^{\frac{1}{2}}, \\
\|(a_I)_{I\in \mathscr{I}}\|_{bmo} \lesssim \lambda^{-1}.
\end{gather*}
Combining these inequalities and Lemma \ref{Section4LemmaInterpolation} gives \eqref{Section4EquationCountingFunctionEstimate}.

Now consider the $0<p<1$ case \eqref{Section4EquationCountingFunctionEstimateBelow1}. Define
\begin{equation*}
A_{l}:=\frac{1}{|I_{\mathscr{T}_l}|} \big{\|}\big(\sum_{\mathcal{R}\in \mathscr{T}_l} \frac{|\langle f|\varphi_{\mathcal{R}} \rangle|^2}{|I(\mathcal{R})|}1_{I(\mathcal{R})} \big)^{\frac{1}{2}}\big{\|}_{1, \infty},\ 1\leq l\leq L.
\end{equation*}
Let $2<r<\infty$ to be determined. We shall estimate the quantity:
\begin{equation*}
\int_{\mathbb{R}^d} \big( \sum_{l=1}^L(A_l 1_{I_{\mathscr{T}_l}})^r \big)^p
\end{equation*}
in two ways. On the one hand, \eqref{Section4EquationTreeHasLargeMass} implies
\begin{equation}  \label{Section8EquationProvingCountingFunctionEstimateBelow1}
\int_{\mathbb{R}^d} \big( \sum_{l=1}^L(A_l 1_{I_{\mathscr{T}_l}})^r \big)^p\gtrsim \lambda^{rp} \int_{\mathbb{R}^d} (\sum_{l=1}^L 1_{I_{\mathscr{T}_l}})^p.
\end{equation}
On the other hand, write $\frac{1}{r}=\frac{\theta}{2}+\frac{1-\theta}{\infty}$ so that $\theta=\frac{2}{r}$. Applying H\"olider's inequality, Corollary \ref{Section4CorollaryMassEstimate2}, H\"olider's inequality, and \eqref{Secteion4EquationFundamentalBesselTypeEstimate} successively gives
\begin{equation}  \label{Section8EquationProvingCountingFunctionEstimateBelow2}
\begin{split}
\int_{\mathbb{R}^d} \big( \sum_{l=1}^L(A_l 1_{I_{\mathscr{T}_l}})^r \big)^p & \leq \int_{\mathbb{R}^d} \big( \sum_{l=1}^L(A_l 1_{I_{\mathscr{T}_l}})^2 \big)^{\frac{\theta rp}{2}}(\sup_{1\leq l\leq L} A_l 1_{I_{\mathscr{T}_l}})^{(1-\theta)rp}  \\
& \lesssim \int_{\mathbb{R}^d} \big( \sum_{l=1}^L(A_l 1_{I_{\mathscr{T}_l}})^2 \big)^{\frac{\theta rp}{2}}(Mf)^{(1-\theta)rp}  \\
& =\int_{\mathbb{R}^d} \big( \sum_{l=1}^L(A_l 1_{I_{\mathscr{T}_l}})^2 \big)^p(Mf)^{(r-2)p}  \\
& \leq \big( \int_{\mathbb{R}^d} \sum_{l=1}^L(A_l 1_{I_{\mathscr{T}_l}})^2 \big)^p \big( \int_{\mathbb{R}^d} (Mf)^{\frac{(r-2)p}{1-p}}\big)^{1-p}  \\
& \lesssim |E|^p\cdot |E|^{1-p}
\end{split}
\end{equation}
provided $\frac{(r-2)p}{1-p}>1$. Take $r:=1+\frac{1}{p}+\epsilon$. Combining \eqref{Section8EquationProvingCountingFunctionEstimateBelow1} and \eqref{Section8EquationProvingCountingFunctionEstimateBelow2} proves \eqref{Section4EquationCountingFunctionEstimateBelow1}.
\end{proof}
It remains to prove Proposition \ref{Section4PropositionMassSelectionAlgorithm} and Proposition \ref{Section4PropositionMassSelectionBesselTypeEstimate}. We prove the former first.
\begin{proof}
It suffices to prove the $(k, +)$ version by symmetry. To each $(k, +)\text{-lacunary}$ tree $\mathscr{T}\subset \mathscr{P}$, we fix one $\xi_{\mathscr{T}}\in \mathbb{R}^d$ such that $\xi_{\mathscr{T}}\in C_2\circ \Xi(\mathcal{R})$ and
\begin{equation*}
(\xi_{\mathscr{T}})_k\notin 10\Xi(\mathcal{R})_k,\ (\xi_{\mathscr{T}})_k<\mathfrak{c}\big( \Xi(\mathcal{R})_k \big)
\end{equation*}
holds for any $\mathcal{R}\in \mathscr{T}$. Run the following algorithm for $l\geq 1$:
\begin{description}
\item[Step 1] Take one $(k, +)\text{-lacunary}$ tree $\mathscr{T}\subset \mathscr{P}$ such that \eqref{Section4EquationTreeHasLargeMass} holds and $(\xi_{\mathscr{T}})_k$ is the largest. (If not unique, take arbitrarily one.) Denote $\mathscr{T}$ by $\mathscr{T}_l$.
\item[Step 2] Define the accompanying tree
\begin{equation*}
\widetilde{\mathscr{T}}_l:=\{\mathcal{R}\in \mathscr{P}\backslash \mathscr{T}_l: \mathcal{R}\lesssim (I_{\mathscr{T}_l}, \xi_{\mathscr{T}_l})\}.
\end{equation*}
\item[Step 3] Renew $\mathscr{P}$ into $\mathscr{P}\backslash (\mathscr{T}_l\cup \widetilde{\mathscr{T}}_l)$. Then return to Step 1 to proceed the $(l+1)\text{-th}$ iteration.
\end{description}
The algorithm terminates when there is no $(k, +)\text{-lacunary}$ tree satisfying \eqref{Section4EquationTreeHasLargeMass}. Assume it terminates after $L$ iterations. We need to prove the second property in Proposition \ref{Section4PropositionMassSelectionAlgorithm}. Suppose $\mathbf{F}$ is not strongly disjoint. Then there exists $l\neq l^{'}, \mathcal{R}\in \mathscr{T}_l, \mathcal{R}^{'}\in \mathscr{T}_{l^{'}}$ such that $s \big(\Xi(\mathcal{R}) \big)<s \big(\Xi(\mathcal{R}^{'}) \big), \Xi(\mathcal{R})\cap \Xi(\mathcal{R}^{'})\neq \emptyset, I(\mathcal{R}^{'})\cap I_{\mathscr{T}_l}\neq \emptyset$. The sparse property implies $s \big(\Xi(\mathcal{R}^{'}) \big)>C_3\big(\Xi(\mathcal{R}) \big)$.

On the one hand, the above conditions imply $\xi_{\mathscr{T}_l}\in C_2\circ\Xi(\mathcal{R})\subset C_2\circ \Xi(\mathcal{R}^{'})$ and $I(\mathcal{R}^{'})\subset I_{\mathscr{T}_l}$ by the grid property. Thus
\begin{equation}  \label{Section4EquationProvingMassSelectionAlgorithm}
\mathcal{R}^{'}\lesssim (I_{\mathscr{T}_l}, \xi_{\mathscr{T}_l}).
\end{equation}
On the other hand, we have
\begin{align*}
(\xi_{\mathscr{T}_l})_k-\mathfrak{c} \big(\Xi(\mathcal{R})_k \big) & \geq -2C_2s \big(\Xi(\mathcal{R}) \big), \\
(\xi_{\mathscr{T}_{l^{'}}})_k-\mathfrak{c} \big(\Xi(\mathcal{R}^{'})_k \big) & \leq -5s \big(\Xi(\mathcal{R}^{'}) \big)
\end{align*}
because $\mathscr{T}_l$ is a tree and $\mathscr{T}_{l^{'}}$ is a $(k, +)\text{-lacunary}$ tree. Then $s \big(\Xi(\mathcal{R}^{'}) \big)>C_3\big(\Xi(\mathcal{R}) \big)$ implies $(\xi_{\mathscr{T}_{l^{'}}})_k<(\xi_{\mathscr{T}_l})_k$. Hence $l<l^{'}$ and $\mathscr{T}_l$ is selected first, which contradicts \eqref{Section4EquationProvingMassSelectionAlgorithm} and Step 2. This finishes the proof.
\end{proof}
Below we show Proposition \ref{Section4PropositionMassSelectionBesselTypeEstimate}. The main idea is utilizing a $T^*T$ argument to explore orthogonality in the time-frequency plane.
\begin{proof}
Let $T: f\mapsto (\langle f|\varphi_{\mathcal{R}})_{\mathcal{R}}$. Note that
\begin{equation*}
\|Tf\|_2^2=\langle f|T^{*}Tf \rangle\leq \|f\|_2\|T^*T\|_{2},
\end{equation*}
so it suffices to show $\|T^*T\|_{2}\lesssim \|Tf\|_2$. Square the $L^2$ norms on both sides and expand them. We only need to prove
\begin{equation}  \label{Section4EquationProvingMassSelectionBesselTypeEstimate1}
\sum_{\mathscr{T}, \mathscr{T}^{'}} \sum_{\mathcal{R}\in \mathscr{T}, \mathcal{R}^{'}\in \mathscr{T}^{'}} |\langle f|\varphi_{\mathcal{R}} \rangle||\langle f|\varphi_{\mathcal{R}^{'}} \rangle||\langle \varphi_{\mathcal{R}}|\varphi_{\mathcal{R}^{'}} \rangle|\lesssim \sum_{\mathscr{T}} \sum_{\mathcal{R}\in \mathscr{T}} |\langle f|\varphi_{\mathcal{R}} \rangle|^2.
\end{equation}
First consider those terms satisfying $s\big( \Xi(\mathcal{R}) \big)=s\big( \Xi(\mathcal{R}^{'}) \big)$. We must have $ \Xi(\mathcal{R})\cap  \Xi(\mathcal{R}^{'})$ for non-zero contribution. Hence $\Xi(\mathcal{R})=\Xi(\mathcal{R}^{'})$ by the sparse condition. Write
\begin{equation*}
\sum_{\mathscr{T}, \mathscr{T}^{'}} \sum_{\substack{\mathcal{R}\in \mathscr{T}, \mathcal{R}^{'}\in \mathscr{T}^{'},\\ s\big( \Xi(\mathcal{R}) \big)=s\big( \Xi(\mathcal{R}^{'}) \big)}} \cdots=\sum_{\mathscr{T}}\sum_{\mathcal{R}\in \mathscr{T}}\sum_{\substack{\mathscr{T}^{'}, \mathcal{R}^{'}\in \mathscr{T}^{'},\\ \Xi(\mathcal{R}^{'})=\Xi(\mathcal{R})}} \cdots.
\end{equation*}
By Schur's test, it suffices to show
\begin{equation}  \label{Section4EquationProvingMassSelectionBesselTypeEstimate2}
\sum_{\substack{\mathscr{T}^{'}, \mathcal{R}^{'}\in \mathscr{T}^{'},\\ \Xi(\mathcal{R}^{'})=\Xi(\mathcal{R})}} |\langle \varphi_{\mathcal{R}}|\varphi_{\mathcal{R}^{'}} \rangle|\lesssim 1
\end{equation}
for fixed $\mathscr{T}$ and $\mathcal{R}\in \mathscr{T}$. But all those $\mathcal{R}^{'}$ in \eqref{Section4EquationProvingMassSelectionBesselTypeEstimate2} are different rectangles sharing the same frequency interval. Hence their physical intervals, having the same scale, are disjoint. Applying \eqref{Dection4EquationSupportPerturbationLemma} then gives \eqref{Section4EquationProvingMassSelectionBesselTypeEstimate2}.

Then consider those terms in LHS of \eqref{Section4EquationProvingMassSelectionBesselTypeEstimate1} satisfying $s\big( \Xi(\mathcal{R}) \big)\neq s\big( \Xi(\mathcal{R}^{'}) \big)$. Then $\Xi(\mathcal{R})\cap  \Xi(\mathcal{R}^{'})\neq \emptyset$ implies $\mathscr{T}\neq \mathscr{T}^{'}$ by the lacunary property. We only need to tackle the $s\big( \Xi(\mathcal{R}) \big)<s\big( \Xi(\mathcal{R}^{'}) \big)$ case by symmetry. Now \eqref{Section4EquationRegularMassAssumption} implies
\begin{gather*}
\frac{|\langle f|\varphi_{\mathcal{R}} \rangle|}{|I(\mathcal{R})|^{\frac{1}{2}}}\lesssim \frac{1}{|I_{\mathscr{T}}|^{\frac{1}{2}}} (\sum_{\mathcal{R}\in \mathscr{T}} |\langle f|\varphi_{\mathcal{R}} \rangle|^2)^{\frac{1}{2}}, \\
\frac{|\langle f|\varphi_{\mathcal{R}^{'}} \rangle|}{|I(\mathcal{R}^{'})|^{\frac{1}{2}}}\lesssim \frac{1}{|I_{\mathscr{T}}|^{\frac{1}{2}}} (\sum_{\mathcal{R}\in \mathscr{T}} |\langle f|\varphi_{\mathcal{R}} \rangle|^2)^{\frac{1}{2}}.
\end{gather*}
Plug the above two estimates into
\begin{equation*}
\sum_{\mathscr{T}, \mathscr{T}^{'}: \mathscr{T}\neq \mathscr{T}^{'}} \sum_{\substack{\mathcal{R}\in \mathscr{T}, \mathcal{R}^{'}\in \mathscr{T}^{'},\\ s\big( \Xi(\mathcal{R}) \big)<s\big( \Xi(\mathcal{R}^{'}) \big)}} |\langle f|\varphi_{\mathcal{R}} \rangle||\langle f|\varphi_{\mathcal{R}^{'}} \rangle||\langle \varphi_{\mathcal{R}}|\varphi_{\mathcal{R}^{'}} \rangle|.
\end{equation*}
We get it is bounded by
\begin{equation*}
\sum_{\mathscr{T}, \mathscr{T}^{'}: \mathscr{T}\neq \mathscr{T}^{'}} \sum_{\substack{\mathcal{R}\in \mathscr{T}, \mathcal{R}^{'}\in \mathscr{T}^{'},\\ s\big( \Xi(\mathcal{R}) \big)<s\big( \Xi(\mathcal{R}^{'}) \big)}} \frac{1}{|I_{\mathscr{T}}|}|I(\mathcal{R})|^{\frac{1}{2}}|I(\mathcal{R}^{'})|^{\frac{1}{2}}|\langle \varphi_{\mathcal{R}}|\varphi_{\mathcal{R}^{'}} \rangle| (\sum_{\mathcal{P}\in \mathscr{T}} |\langle f|\varphi_{\mathcal{P}} \rangle|^2).
\end{equation*}
It then suffices to prove
\begin{equation}  \label{Section4EquationProvingMassSelectionBesselTypeEstimate3}
\sum_{\mathcal{R}\in \mathscr{T}}\sum_{\substack{\mathscr{T}^{'}, \mathcal{R}^{'}\in \mathscr{T}^{'},\\ s\big(\Xi(\mathcal{R}^{'})\big)>s\big(\Xi(\mathcal{R})\big)}} |I(\mathcal{R})|^{\frac{1}{2}}|I(\mathcal{R}^{'})|^{\frac{1}{2}}|\langle \varphi_{\mathcal{R}}|\varphi_{\mathcal{R}^{'}} \rangle|\lesssim |I_{\mathscr{T}}|
\end{equation}
for fixed $\mathscr{T}$. Fix $\mathcal{R}$ as well. Let $\mathcal{R}^{'}, \mathcal{R}^{''}$ be two different rectangles such that $s\big(\Xi(\mathcal{R}^{'})\big)>s\big(\Xi(\mathcal{R})\big), s\big(\Xi(\mathcal{R}^{''})\big)>s\big(\Xi(\mathcal{R})\big), \Xi(\mathcal{R}^{'})\cap \Xi(\mathcal{R})\neq \emptyset, \Xi(\mathcal{R}^{''})\cap \Xi(\mathcal{R})\neq \emptyset$. We show $I(\mathcal{R}^{'})\cap I(\mathcal{R}^{''})$ by considering the following three cases:
\begin{enumerate}
\item If $\mathcal{R}^{'}, \mathcal{R}^{''}$ belong to the same tree, then $\Xi(\mathcal{R}^{'})=\Xi(\mathcal{R}^{''})$. Otherwise $d\big( \Xi(\mathcal{R}^{'}), \Xi(\mathcal{R}^{''})\big)>\max\big{\{}s\big( \Xi(\mathcal{R}^{'})\big), s\big( \Xi(\mathcal{R}^{''})\big)\big{\}}$, which contradicts $\Xi(\mathcal{R}^{'})\cap \Xi(\mathcal{R})\neq \emptyset, \Xi(\mathcal{R}^{''})\cap \Xi(\mathcal{R})\neq \emptyset$. Hence $I(\mathcal{R}^{'})\cap I(\mathcal{R}^{''})=\emptyset$.
\item If $\mathcal{R}^{'}, \mathcal{R}^{''}$ belong to different trees such that $s\big(\Xi(\mathcal{R}^{'})\big)=s\big(\Xi(\mathcal{R}^{''})\big)$, then we still get $\Xi(\mathcal{R}^{'})=\Xi(\mathcal{R}^{''})$ by the sparse property. Thus $I(\mathcal{R}^{'})\cap I(\mathcal{R}^{''})=\emptyset$ as above.
\item If $\mathcal{R}^{'}, \mathcal{R}^{''}$ belong to different trees such that $s\big(\Xi(\mathcal{R}^{'})\big)\neq s\big(\Xi(\mathcal{R}^{''})\big)$, then $\Xi(\mathcal{R}\subset \Xi(\mathcal{R}^{'}\cap \Xi(\mathcal{R}^{''}$ implies $I(\mathcal{R}^{'})\cap I(\mathcal{R}^{''})=\emptyset$ by the strongly disjoint property.
\end{enumerate}
Moreover, the strongly disjoint property also implies $I(\mathcal{R}^{'})\cap I_{\mathscr{T}}=\emptyset$ for any $\mathcal{R}$ such that $s\big(\Xi(\mathcal{R}^{'})\big)>s\big(\Xi(\mathcal{R})\big), \Xi(\mathcal{R}^{'})\cap \Xi(\mathcal{R})\neq \emptyset$. For each $\mathcal{R}$, We can estimate
\begin{equation*}
\begin{split}
\sum_{\substack{\mathscr{T}^{'}, \mathcal{R}^{'}\in \mathscr{T}^{'},\\ s\big(\Xi(\mathcal{R}^{'})\big)>s\big(\Xi(\mathcal{R})\big)}} |I(\mathcal{R})|^{\frac{1}{2}}|I(\mathcal{R}^{'})|^{\frac{1}{2}}|\langle \varphi_{\mathcal{R}}|\varphi_{\mathcal{R}^{'}} \rangle| & \lesssim \sum_{\text{all involved\ }\mathcal{R}^{'}} \Big(1+\frac{d\big(I(\mathcal{R}), I(\mathcal{R}^{'})\big)}{s\big(I(\mathcal{R}) \big)} \Big)^{-5d}|I(\mathcal{R}^{'})| \\
&\leq \sum_{\text{all involved\ }\mathcal{R}^{'}} \int_{I(\mathcal{R}^{'})} \Big(1+\frac{d\big(x, I(\mathcal{R})\big)}{s\big(I(\mathcal{R}) \big)} \Big)^{-5d} dx \\
&\leq \int_{(I_{\mathscr{T}})^c} \Big(1+\frac{d\big(x, I(\mathcal{R})\big)}{s\big(I(\mathcal{R}) \big)} \Big)^{-5d} dx \\
&\lesssim \Big(1+\frac{d\big(I(\mathcal{R}), (I_{\mathscr{T}})^c \big)}{s\big(I(\mathcal{R}) \big)} \Big)^{-2d} s\big(I(\mathcal{R}) \big)^d.
\end{split}    
\end{equation*}
Plug this into \eqref{Section4EquationProvingMassSelectionBesselTypeEstimate3}. Note that the map $\mathcal{R}\mapsto I(\mathcal{R})$ is bijective inside $\mathscr{T}$. We can finally bound LHS of \eqref{Section4EquationProvingMassSelectionBesselTypeEstimate3} by
\begin{equation*}
\begin{split}
\sum_{I\subset I_{\mathscr{T}}} \Big(1+\frac{d\big(I, (I_{\mathscr{T}})^c \big)}{s(I)} \Big)^{-2d} s(I)^d & =\sum_{2^k\leq s(I_{\mathscr{T}})}2^{kd}\sum_{I\subset I_{\mathscr{T}}: s(I)=2^k}\Big(1+\frac{d\big(I, (I_{\mathscr{T}})^c \big)}{2^k} \Big)^{-2d} \\
&\lesssim \sum_{2^k\leq s(I_{\mathscr{T}})}2^{kd} \Big( \big(\sum_{l\geq 1}\sum_{\substack{I\subset I_{\mathscr{T}}: s(I)=2^k,\\ d\big(I, (I_{\mathscr{T}})^c \big)\sim 2^{k+l}}}2^{-2dl} \big)+ \big(\frac{s(I_{\mathscr{T}})}{2^k} \big)^{d-1} \Big) \\
&\lesssim \sum_{2^k\leq s(I_{\mathscr{T}})}2^{kd} \Big( \big(\sum_{l\geq 1} \frac{s(I_{\mathscr{T}})^{d-1}\cdot2^{k+l}}{2^{kd}}2^{-2dl} \big)+ \big(\frac{s(I_{\mathscr{T}})}{2^k} \big)^{d-1} \Big)  \\
&\lesssim s(I_{\mathscr{T}})^d.
\end{split}
\end{equation*}
This completes the proof of Proposition \ref{Section4PropositionMassSelectionBesselTypeEstimate}.
\end{proof}


\section{Reorganization for Vector Trees}  \label{Section5}
From now on, we add the index $j$ back to distinguish vector rectangles and rectangles, etc. We will divide $\mathscr{R}$ into the disjoint union of vector trees and estimate their counting functions.
\begin{definition}
Consider the decomposition of $\mathscr{R}_j=\{\mathcal{R}_j: \mathcal{R}\in \mathscr{R}\}$ given by Corollary \ref{Section4PropositionMassSelection2}. For a vector of dyadic numbers $\lambda=(\lambda_1, \cdots, \lambda_n)$ such that $\lambda_j\leq \mathbf{M}_j(\mathscr{R}_j), \forall j\in [n]$, define
\begin{equation*}
\mathscr{R}(\lambda):=\{\mathcal{R}\in \mathscr{R}: \mathcal{R}_j\in \mathscr{R}_j(\lambda_j), \forall j\in [n]\}.
\end{equation*}
\end{definition}
Fix $\lambda$. We organize $\mathscr{R}(\lambda)$ into vector trees as follows: Denote the collection of all maximal vector rectangles in $\mathscr{R}(\lambda)$ by $\mathscr{P}(\lambda)$.
\begin{description}
\item[Step 1] Choose one $\mathcal{P}\in \mathscr{P}(\lambda)$ such that $s\big(\Xi(\mathcal{P}) \big)$ is the smallest. (If not unique, take arbitrarily one.) Put the vector tree $\mathscr{T}_{\mathcal{P}}:=\{\mathcal{R}\in \mathscr{R}(\lambda): \mathcal{R}\lesssim \mathcal{P}\}$ into $\mathbf{F}(\lambda)$.
\item[Step 2] Renew $\mathscr{R}(\lambda)$ into $\mathscr{R}(\lambda)\backslash \mathscr{T}_{\mathcal{P}}$ and $\mathscr{P}$ into $\mathscr{P}\backslash\{\mathcal{P}\}$. Then return to Step 1 to iterate.
\end{description}
In this way, we get
\begin{equation*}
\mathscr{R}(\lambda)=\bigcup_{\mathcal{P}\in \mathscr{P}(\lambda)} \mathscr{T}_{\mathcal{P}}=\bigcup_{\mathscr{T}\in \mathbf{F}(\lambda)} \mathscr{T}.
\end{equation*}
We further write $\mathbf{F}(\lambda)=\mathbf{F}^{\geq 2}(\lambda)\cup \mathbf{F}^{=1}(\lambda)$, where
\begin{gather*}
\mathbf{F}^{\geq 2}(\lambda):=\{\mathscr{T}\in \mathbf{F}(\lambda): \#\mathscr{T}\geq 2\}, \\
\mathbf{F}^{=1}(\lambda):=\{\mathscr{T}\in \mathbf{F}(\lambda): \#\mathscr{T}=1\}.
\end{gather*}
The rest of this section is devoted to controlling $N_{\mathbf{F}^{\geq 2}(\lambda)}$ and $N_{\mathbf{F}^{=1}(\lambda)}$. Consider $N_{\mathbf{F}^{\geq 2}(\lambda)}$ first.
\begin{proposition}  \label{Section5PropositionCountingFunctionEstimate1}
Let $\lambda=(\lambda_1, \cdots, \lambda_n)$ be a vector of dyadic numbers with $\lambda_j\leq \mathbf{M}_j(\mathscr{R}_j), \forall j\in [n]$. We can estimate $N_{\mathbf{F}^{\geq 2}(\lambda)}$ in terms of $N_{\mathbf{F}_{j}(\lambda_j)}, j\in [n]$ as follows:
\begin{enumerate}
\item Assume we are in the first case of Theorem \ref{MainTheorem1}. Then
\begin{equation*}
N_{\mathbf{F}^{\geq 2}(\lambda)}\lesssim \prod_{j\in \Upsilon} N_{\mathbf{F}_j(\lambda_j)},\ \forall \Upsilon\subset [n]: \#\Upsilon=\lceil \frac{m}{d} \rceil.
\end{equation*}
\item Assume we are in the second case of Theorem \ref{MainTheorem1}. Then
\begin{equation*}
N_{\mathbf{F}^{\geq 2}(\lambda)}\lesssim \prod_{j\in \Upsilon} N_{\mathbf{F}_j(\lambda_j)}^{\frac{1}{2}},\ \forall \Upsilon\subset [n]: \#\Upsilon=n-1.
\end{equation*}
\item Assume we are in the third case of Theorem \ref{MainTheorem1}. Then
\begin{equation*}
N_{\mathbf{F}^{\geq 2}(\lambda)}\lesssim \prod_{j\in \Upsilon} N_{\mathbf{F}_j(\lambda_j)}^{\frac{m}{d(n-1)}},\ \forall \Upsilon\subset [n]: \#\Upsilon=n-1.
\end{equation*}
\end{enumerate}
\end{proposition}
To prove Proposition \ref{Section5PropositionCountingFunctionEstimate1}, we fix $\lambda, x$ and construct certain injective maps.
\begin{definition}
Let $x\in \mathbb{R}^d$. Let $\mathbf{F}$ be a collection of vector trees. We use $\mathbf{F}$ together with the superscript $x$ to denote the collection of $\mathscr{T}\in \mathbf{F}$ satisfying $x\in I_{\mathscr{T}}$. Use such notations for collections of $j\text{-trees}$ as well. For example,
\begin{equation*}
\mathbf{F}^{\geq 2, x}(\lambda):=\{\mathscr{T}\in \mathbf{F}^{\geq 2}(\lambda): x\in I_{\mathscr{T}}\}.
\end{equation*}
\end{definition}
\begin{remark}
By definition, we have $N_{\mathbf{F}}(x)=\#\mathbf{F}^x$.
\end{remark}
The construction of injective maps replies on the following observation. This is the only place we use the central grid property.
\begin{lemma}  \label{Section5LemmaKeyobservationForInjection1}
Let $\mathscr{T}, \widetilde{\mathscr{T}}\in \mathbf{F}^{\geq 2, x}(\lambda)$ correspond to the maximal vector rectangles $\mathcal{P}, \widetilde{\mathcal{P}}$ respectively. Then $\mathscr{T}=\widetilde{\mathscr{T}}$ if
\begin{equation*}
d\big( \Xi(\mathcal{P}), \Xi( \widetilde{\mathcal{P}}) \big)\lesssim_{C_1, C_2} \max\big{\{} s\big( \Xi(\mathcal{P})\big), s\big(\Xi( \widetilde{\mathcal{P}}) \big)\big{\}}.
\end{equation*}
\end{lemma}
\begin{proof}
Suppose $\mathscr{T}\neq \widetilde{\mathscr{T}}$. Then the sparse condition and the assumption on $d\big( \Xi(\mathcal{P}), \Xi( \widetilde{\mathcal{P}}) \big)$ imply $s\big( \Xi(\mathcal{P})\big)\neq s\big(\Xi( \widetilde{\mathcal{P}}) \big)$. Assume $s\big( \Xi(\mathcal{P})\big)>s\big(\Xi( \widetilde{\mathcal{P}}) \big)$ without loss of generality. Since $\#\mathscr{T}\geq 2$, we can take $\mathcal{R}\in \mathscr{T}$ such that $\mathcal{R}\neq \mathcal{P}$. We must have $s\big( \Xi(\mathcal{R})\big)>s\big( \Xi(\mathcal{P})\big)$, which implies $C_2\circ \Xi(\mathcal{P}_j)\varsubsetneqq C_2\circ \Xi(\mathcal{R}_j), \forall j\in [n]$. The fact that $\{\Xi(\mathcal{R}_j): \mathcal{R}\in \mathscr{R}\}$ is a $C_3^{\frac{1}{3}}\text{-central}$ grid (Lemma \ref{Section4LemmaCentralGrid}) implies $C_3^{\frac{1}{3}}\big( C_2\circ \Xi(\mathcal{P}) \big)\subset C_2\circ \Xi(\mathcal{R})$. Then the assumption on $d\big( \Xi(\mathcal{P}), \Xi( \widetilde{\mathcal{P}}) \big)$ gives
\begin{equation*}
C_2\circ \Xi( \widetilde{\mathcal{P}})\subset C_3^{\frac{1}{3}}\big( C_2\circ \Xi(\mathcal{P}) \big)\subset C_2\circ \Xi(\mathcal{R}).
\end{equation*}
Combing this with $x\in I_{\mathscr{T}}\cap I_{\widetilde{\mathscr{T}}}$, we see $\mathcal{R}\lesssim \widetilde{\mathcal{P}}$. Recall $s\big( \Xi(\mathcal{P})\big)>s\big(\Xi( \widetilde{\mathcal{P}}) \big)$. Hence $\widetilde{\mathscr{T}}$ should be selected first in our algorithm, which contradicts the fact $\mathcal{R}\in \mathscr{T}$. This finishes the proof.
\end{proof}
We are now ready to construct the injective map for the first case of Theorem \ref{MainTheorem1}.
\begin{definition}
Let $\mathscr{T}\in \mathbf{F}^{\geq 2, x}(\lambda)$ correspond to the maximal vector rectangle $\mathcal{P}$. Then for any $j\in [n]$, $\mathcal{P}_j\in \mathscr{R}_j$ belongs to a unique $\mathscr{T}_j\in \mathbf{F}_j(\lambda_j)$. Denote this $j\text{-tree}$ by $\sigma_j(\mathscr{T})$. Let $\Upsilon\subset [n]$. We also write $\sigma_{\Upsilon}({\mathscr{T}}):=\big(\sigma_j(\mathscr{T}) \big)_{j\in \Upsilon}$.
\end{definition}
\begin{proposition}  \label{Section5PropositionInjectiveMap1}
Under Type \uppercase\expandafter{\romannumeral1} non-degenerate condition for $\Gamma$, the restricted map $\sigma_{\Upsilon}\vert_{\mathbf{F}^{\geq 2, x}(\lambda)}: $
\begin{equation*}
\mathbf{F}^{\geq 2, x}(\lambda)\rightarrow \prod_{j\in \Upsilon} \mathbf{F}_j^x(\lambda_j)
\end{equation*}
is an injection for any $\Upsilon\in [n]: \#\Upsilon=\lceil \frac{m}{d}\rceil$.
\end{proposition}
\begin{proof}
Let $\mathscr{T}, \widetilde{\mathscr{T}}\in \mathbf{F}^{\geq 2, x}(\lambda)$ correspond to the maximal vector rectangles $\mathcal{P}, \widetilde{\mathcal{P}}$ respectively. Assume $\sigma_{\Upsilon}(\mathscr{T})=\sigma_{\Upsilon}(\widetilde{\mathscr{T}})$. Then
\begin{equation} \label{Section5EquationProvingPropositionInjectiveMap1}
d\big( \Xi(\mathcal{P}_j), \Xi(\widetilde{\mathcal{P}}_j) \big)\lesssim C_2\max\big{\{} s\big( \Xi(\mathcal{P})\big), s\big(\Xi( \widetilde{\mathcal{P}}) \big)\big{\}}
\end{equation}
for any $j\in [n]$. Take $\tau, \tilde{\tau}\in \Gamma$ such that
\begin{gather*}
d\big( \tau, \Xi(\mathcal{P}) \big)\lesssim C_1 s\big( \Xi(\mathcal{P})\big), \\
d\big( \tilde{\tau}, \Xi(\widetilde{\mathcal{P}}) \big)\lesssim C_1 s\big( \Xi(\widetilde{\mathcal{P}})\big).
\end{gather*}
Combining these two inequalities and \eqref{Section5EquationProvingPropositionInjectiveMap1} gives
\begin{equation}  \label{Section5EquationProvingPropositionInjectiveMap2}
d(\tau_{\Upsilon\times [d]}, \tilde{\tau}_{\Upsilon\times [d]})\lesssim_{C_1, C_2} \max\big{\{} s\big( \Xi(\mathcal{P})\big), s\big(\Xi( \widetilde{\mathcal{P}}) \big)\big{\}}.
\end{equation}
Then Type \uppercase\expandafter{\romannumeral1} non-degenerate condition implies
\begin{equation}  \label{Section5EquationProvingPropositionInjectiveMap3}
d(\tau, \tilde{\tau})\lesssim_{C_1, C_2} \max\big{\{} s\big( \Xi(\mathcal{P})\big), s\big(\Xi( \widetilde{\mathcal{P}}) \big)\big{\}},
\end{equation}
which in turn implies
\begin{equation}  \label{Section5EquationProvingPropositionInjectiveMap4}
d\big( \Xi(\mathcal{P}), \Xi( \widetilde{\mathcal{P}}) \big)\lesssim_{C_1, C_2} \max\big{\{} s\big( \Xi(\mathcal{P})\big), s\big(\Xi( \widetilde{\mathcal{P}}) \big)\big{\}}.  
\end{equation}
Applying Lemma \ref{Section5LemmaKeyobservationForInjection1} gives the desired result.
\end{proof}
Proposition \ref{Section5PropositionInjectiveMap1} implies the first case of Proposition \ref{Section5PropositionCountingFunctionEstimate1}. To handle other cases, we need to following lemma by Katz and Tao \cite{KTEntropyInequality}. We include their proof for completeness.
\begin{lemma}  \label{Section5LemmaKatz-Tao}
Let $L\in \mathbb{N}_+$. Let $\Omega, Z_1, \cdots, Z_{L-1}$ be finite sets and $h_l: \Omega\rightarrow Z_l, 1\leq l\leq L-1$ be maps. There exists a constant $c_L>0$ such that
\begin{equation*}
\#\{(\omega_1, \cdots, \omega_L)\in \Omega^L: h_l(\omega_j)=h_l(\omega_{l+1}), \forall 1\leq l\leq L-1\}\geq c_L\frac{(\#\Omega)^L}{\prod_{l=1}^{L-1} \#Z_l}.
\end{equation*}
\end{lemma}
\begin{proof}
We prove by induction on $L$. When $L=1$, both sides are $\#\Omega$. Assume the inequality is true for $L$, we show it is also true for $L+1$. Call an element $z_L\in Z_L$ popular if
\begin{equation*}
\#\big(h_L^{-1}(z_L)\big)\geq \frac{1}{2}\frac{\#\Omega}{\#Z_L}.
\end{equation*}
Define
\begin{equation*}
\Omega^{'}:=\{\omega\in \Omega: h_L(\omega)\text{\ is popular.}\}    
\end{equation*}
Then
\begin{equation*}
\#\Omega^{'}\geq \#\Omega-\#Z_L\cdot\frac{1}{2}\frac{\#\Omega}{\#Z_L}=\frac{\#\Omega}{2}.
\end{equation*}
Apply the induction hypothesis to $\Omega^{'}, Z_1, \cdots, Z_{L-1}$. We get
\begin{equation*}
\#\{(\omega_1, \cdots, \omega_L)\in (\Omega^{'})^L: h_l(\omega_l)=h_l(\omega_{l+1}), \forall 1\leq l\leq L-1\}\geq c_L\frac{(\#\Omega^{'})^L}{\prod_{l=1}^{L-1} \#Z_l}
\end{equation*}
Now add $\omega_{L+1}$ to those $(\omega_1, \cdots, \omega_L)$'s in the above set. For each fixed $L\text{-tuple}$, we have at least $\frac{1}{2}\frac{\#\Omega}{\#Z_L}$ choices of $\omega_{L+1}$ such that $h_L(\omega_L)=h_L(\omega_{L+1})$. Hence the inequality for $L+1$ follows with $c_{L+1}:=\frac{c_L}{2^L}$.
\end{proof}
\begin{remark}
One can improve $c_L$ to $1$ by the tensor power trick \cite{KTEntropyInequality}. The explicit value of $c_L$ is not important for our application.
\end{remark}
Recall $A, B^{(k)}$ are given in Definition \ref{Section1Non-DegenerateCondition2}. Let $L=2$ for the second case and $L=d$ for the third case of Proposition \ref{Section5PropositionCountingFunctionEstimate1}. Apply Lemma \ref{Section5LemmaKatz-Tao} with $\Omega:=\mathbf{F}^{\geq 2, x}(\lambda)$, $Z_1=\cdots=Z_{L-1}:=\prod_{j\in A} \mathbf{F}_j^x(\lambda_j)$, and $h_1=\cdots=h_{L-1}:=\sigma_{A}\vert_{\mathbf{F}^{\geq 2, x}(\lambda)}$. Define
\begin{equation*}
\mathbf{G}:=\big{\{}(\mathscr{T}^{(1)}, \cdots, \mathscr{T}^{(L)})\in \big(\mathbf{F}^{\geq 2, x}(\lambda) \big)^L: \sigma_{A}(\mathscr{T}^{(1)})=\cdots=\sigma_A(\mathscr{T}^{(L)})\big{\}}.
\end{equation*}
Then we get
\begin{equation}  \label{Section5EquationApplyingKatz-Tao}
(\#\mathbf{F}^{\geq 2, x})^L\lesssim \#\mathbf{G}\big(\prod_{j\in A} \#\mathbf{F}_j^x(\lambda_j)\big)^{L-1}.
\end{equation}
We shall prove
\begin{proposition}  \label{Section5PropositionInjectiveMap2}
Under Type \uppercase\expandafter{\romannumeral1} and Type \uppercase\expandafter{\romannumeral2} non-degenerate conditions for $\Gamma$, the restricted map $\prod_{k=1}^L\sigma_{B^{(k)}}\vert_{\mathbf{G}}:$
\begin{equation*}
\mathbf{G}\rightarrow \prod_{k=1}^L\prod_{j\in B^{(k)}}\mathbf{F}_j^x(\lambda_j)
\end{equation*}
is an injection for any $A, B^{(k)}$ given in Definition \ref{Section1Non-DegenerateCondition2}.
\end{proposition}
Combining \eqref{Section5EquationApplyingKatz-Tao} and Proposition \ref{Section5PropositionInjectiveMap2} immediately proves the second case of Proposition \ref{Section5PropositionCountingFunctionEstimate1}. For the third case, we obtain an estimate of the type:
\begin{equation*}
\#\mathbf{F}^{\geq 2, x}\lesssim \prod_{j\in A\cup B^{(1)}\cup \cdots \cup B^{(L)}} \big(\#\mathbf{F}_j^x(\lambda_j)\big)^{\alpha_j},
\end{equation*}
where
\begin{equation*}
\sum_{j\in A\cup B^{(1)}\cup \cdots \cup B^{(d)}} \alpha_j=\frac{d-1}{d}\cdot\#A+\frac{1}{d}\cdot d\cdot\#B^{(1)}=\frac{m}{d}.
\end{equation*}
Note that we always have $\# (A\cup B^{(1)}\cup \cdots \cup B^{(d)})\leq n-1$. Hence taking a geometric average over different choices of $(A, B^{(1)}, \cdots, B^{(d)})$ proves the third case of Proposition \ref{Section5PropositionCountingFunctionEstimate1}. Now we show Proposition \ref{Section5PropositionInjectiveMap2}.
\begin{proof}
Let $(\mathscr{T}^{(1)}, \cdots, \mathscr{T}^{(L)}), (\widetilde{\mathscr{T}^{(1)}}, \cdots, \widetilde{\mathscr{T}^{(L)}})\in \mathbf{G}$ correspond to the tuples of maximal vector rectangles $(\mathcal{P}^{(1)}, \cdots, \mathcal{P}^{(L)}), (\widetilde{\mathcal{P}^{(1)}}, \cdots, \widetilde{\mathcal{P}^{(L)}})$. Assume
\begin{equation*}
\prod_{k=1}^L\sigma_{B^{(k)}}\big( (\mathscr{T}^{(1)}, \cdots, \mathscr{T}^{(L)})\big)=\prod_{k=1}^L\sigma_{B^{(k)}}\big( (\widetilde{\mathscr{T}^{(1)}}, \cdots, \widetilde{\mathscr{T}^{(L)}})\big).
\end{equation*}
Then
\begin{equation} \label{Section5EquationProvingPropositionSection5PropositionInjectiveMap21}
d\big( \Xi(\mathcal{P}_j^{(k)}), \Xi(\widetilde{\mathcal{P}^{(k)}}_j) \big)\lesssim C_2 S
\end{equation}
for any $1\leq k\leq L, j\in B^{(k)}$. Here
\begin{equation*}
S:=\max\big{\{} s\big( \Xi(\mathcal{P}^{(1)})\big), \cdots, s\big( \Xi(\mathcal{P}^{(L)})\big), s\big(\Xi( \widetilde{\mathcal{P}^{(1)}}), \cdots, s\big(\Xi( \widetilde{\mathcal{P}^{(L)}}) \big)\big{\}}
\end{equation*}
For each $1\leq k\leq L$, choose $\xi^{(k)}, \widetilde{\xi^{(k)}}\in \Gamma$ such that
\begin{equation}  \label{Section5EquationProvingPropositionSection5PropositionInjectiveMap22}
d\big( \Xi(\mathcal{P}^{(k)}), \xi^{(k)} \big)\lesssim C_1 S,\ d\big( \Xi(\widetilde{\mathcal{P}^{(k)}}), \widetilde{\xi^{(k)}} \big)\lesssim C_1 S.
\end{equation}
Let $w^{(k)}:=\xi^{(k)}-\widetilde{\xi^{(k)}}, 1\leq k\leq L$. The definition of $\mathbf{G}$ and \eqref{Section5EquationProvingPropositionSection5PropositionInjectiveMap22} imply
\begin{equation}  \label{Section5EquationProvingPropositionSection5PropositionInjectiveMap23}
\begin{split}
|P_{A}(w^{(1)})-P_{A}(w^{(k)})| & =|\big(\xi^{(1)}_{A\times [d]}-\widetilde{\xi^{(1)}}_{A\times [d]}\big)-\big(\xi^{(k)}_{A\times [d]}-\widetilde{\xi^{(k)}}_{A\times [d]} \big)|  \\
&=|\big(\xi^{(1)}_{A\times [d]}-\xi^{(k)}_{A\times [d]}\big)-\big(\widetilde{\xi^{(1)}}_{A\times [d]}-\widetilde{\xi^{(k)}}_{A\times [d]} \big)| \\
&\lesssim_{C_1, C_2} S.
\end{split}
\end{equation}
for any $2\leq k\leq L$. Moreover, \eqref{Section5EquationProvingPropositionSection5PropositionInjectiveMap21} and \eqref{Section5EquationProvingPropositionSection5PropositionInjectiveMap22} imply
\begin{equation}  \label{Section5EquationProvingPropositionSection5PropositionInjectiveMap24}
\begin{split}
|P_{B^{(k)}}(w^{(k)})| & =|\xi^{(k)}_{B^{(k)}\times [d]}-\widetilde{\xi^{(k)}}_{B^{(k)}\times [d]}|  \\
&\lesssim_{C_1, C_2} S.
\end{split}
\end{equation}
for any $1\leq k\leq L$. Now consider the linear map $\mathfrak{L}$ given by \eqref{Section1EquationTheCrucialMap}. On the one hand, \eqref{Section5EquationProvingPropositionSection5PropositionInjectiveMap23} and \eqref{Section5EquationProvingPropositionSection5PropositionInjectiveMap24} means
\begin{equation}  \label{Section5EquationProvingPropositionSection5PropositionInjectiveMap25}
|\mathfrak{L}\big( (w^{(1)}, \cdots, w^{(L)}) \big)|\lesssim_{C_1, C_2} S.
\end{equation}
On the other hand, Type \uppercase\expandafter{\romannumeral2} non-degenerate condition tells us that $\ker(\mathfrak{L})=\{0\}$. Hence $\mathfrak{L}$ is a invertible linear map from $\Gamma^L$ to its image. Thus we obtain
\begin{equation}  \label{Section5EquationProvingPropositionSection5PropositionInjectiveMap26}
|(w^{(1)}, \cdots, w^{(L)})|\lesssim_{C_1, C_2} S
\end{equation}
from \eqref{Section5EquationProvingPropositionSection5PropositionInjectiveMap25}.
Pick $k$ such that
\begin{equation*}
\max\big{\{} s\big( \Xi(\mathcal{P}^{(k)})\big), s\big(\Xi( \widetilde{\mathcal{P}^{(k)}}) \big{\}}=S.
\end{equation*}
Combining \eqref{Section5EquationProvingPropositionSection5PropositionInjectiveMap26} and Lemma \ref{Section5LemmaKeyobservationForInjection1} gives $\mathscr{T}^{(k)}=\widetilde{\mathscr{T}^{(k)}}$. This and the definition of $\mathbf{G}$ further imply
\begin{equation*}
\sigma_{A\cup B^{(k^{'})}}(\mathscr{T}^{k^{'}})=\sigma_{A\cup B^{(k^{'})}}(\widetilde{\mathscr{T}^{k^{'}}}),\ \forall k^{'}\neq k.
\end{equation*}
Applying Proposition \ref{Section5PropositionInjectiveMap1} then gives $\mathscr{T}^{k^{'}}=\widetilde{\mathscr{T}^{k^{'}}}$ for each $k^{'}\neq k$. This finishes the proof.
\end{proof}
It remains to estimate $N_{\mathbf{F}^{=1}(\lambda)}$. Since each vector tree in $\mathbf{F}^{=1}(\lambda)$ has cardinality one, we will not distinguish $\mathscr{T}\in \mathbf{F}^{=1}(\lambda)$ from its maximal vector rectangle $\mathcal{P}$. The main difficulty is that Lemma \ref{Section5LemmaKeyobservationForInjection1} is not valid for $\mathbf{F}^{=1, x}(\lambda)$. Instead, we have the following weaker version:
\begin{lemma}  \label{Section5LemmaKeyobservationForInjection2}
Let $\mathcal{P}, \widetilde{\mathcal{P}}\in \mathbf{F}^{=1, x}(\lambda)$. Then $\mathcal{P}=\widetilde{\mathcal{P}}$ if
\begin{equation*}
d\big( \Xi(\mathcal{P}), \Xi( \widetilde{\mathcal{P}}) \big)\lesssim_{C_1} \max\big{\{} s\big( \Xi(\mathcal{P})\big), s\big(\Xi( \widetilde{\mathcal{P}}) \big)\big{\}}.
\end{equation*}
\end{lemma}
\begin{proof}
Suppose $\mathcal{P}\neq \widetilde{\mathcal{P}}$. Then the sparse condition and the assumption on $d\big( \Xi(\mathcal{P}), \Xi( \widetilde{\mathcal{P}}) \big)$ imply $s\big( \Xi(\mathcal{P})\big)\neq s\big(\Xi( \widetilde{\mathcal{P}}) \big)$. Assume $s\big( \Xi(\mathcal{P})\big)>s\big(\Xi( \widetilde{\mathcal{P}}) \big)$ without loss of generality. We get $C_2\circ \Xi( \widetilde{\mathcal{P}})\subset C_2\circ \Xi(\mathcal{P})$ from the condition on $d\big( \Xi(\mathcal{P}), \Xi( \widetilde{\mathcal{P}}) \big)$. Thus $\mathcal{P}\lesssim \widetilde{\mathcal{P}}$, a contradiction.
\end{proof}
To bound $N_{\mathbf{F}^{=1, x}(\lambda)}$ using Lemma \ref{Section5LemmaKeyobservationForInjection2} as in the $\mathbf{F}^{\geq 2, x}(\lambda)$ case, we need to replace the order relation between rectangles $\lesssim$ with $\leq $. This means we have to split each $j\text{-tree}$ in $\mathbf{F}_j(\lambda_j)$ further to fulfill the stronger requirement $\mathcal{R}_j\leq (I_{\mathscr{T}_j}, \xi_{\mathscr{T}_j})$.
\begin{definition}
Let $j\in [n]$ and dyadic number $\lambda_j\leq \mathbf{M}_j(\mathscr{R}_j)$. For each dyadic number $\kappa_j\leq \lambda_j$, define
\begin{equation*}
\mathscr{R}_j(\lambda_j, \kappa_j):=\{\mathcal{R}_j\in \mathscr{R}_j(\lambda_j): \frac{\kappa_j}{2}<\frac{|\langle f_j|\varphi_{\mathcal{R}_j} \rangle|}{|I(\mathcal{R}_j)|^{\frac{1}{2}}}\leq \kappa_j\}.
\end{equation*}
And for two vectors of dyadic numbers $\lambda=(\lambda_1, \cdots, \lambda_n), \kappa=(\kappa_1, \cdots, \kappa_n)$ such that $\kappa_j\leq \lambda_j\leq \mathbf{M}_j(\mathscr{R}_j), \forall j\in [n]$, define
\begin{equation*}
\mathscr{R}(\lambda, \kappa):=\{\mathcal{R}\in \mathscr{R}(\lambda): \mathcal{R}_j\in \mathscr{R}_j(\lambda_j, \kappa_j), \forall j\in [n]\}.
\end{equation*}
Recall we identify each vector tree in $\mathbf{F}^{=1}(\lambda)$ with its unique vector rectangle. Write
\begin{equation*}
\mathbf{F}^{=1}(\lambda)=\bigcup_{\kappa: \kappa_j\leq \lambda_j, \forall j\in [n]} \big( \mathbf{F}^{=1}(\lambda)\cap \mathscr{R}(\lambda, \kappa)\big)=:\bigcup_{\kappa: \kappa_j\leq \lambda_j, \forall j\in [n]} \mathbf{F}^{=1}(\lambda, \kappa).
\end{equation*}
\end{definition}
\begin{definition}
Let $j\in [n]$. Let $\lambda_j, \kappa_j$ be dyadic numbers such that $\kappa_j\leq\lambda_j\leq \mathbf{M}_j(\mathscr{R}_j)$. Define $\mathscr{F}_j(\lambda_j, \kappa_j)$ to be the collection of all maximal elements in $\mathscr{R}_j(\lambda_j, \kappa_j)$ under the order relation $\leq$. We also consider $\mathscr{F}_j(\lambda_j, \kappa_j)$ a family of lacunary $j\text{-trees}$ by regarding each $(\mathcal{R}_{max})_j\in \mathscr{F}_j(\lambda_j, \kappa_j)$ as a single-element $j\text{-tree}$.
\end{definition}
\begin{proposition}  \label{Section5PropositionCountingFunctionEstimateForSingle-ElementTrees}
Let $j\in [n]$. Let $\lambda_j, \kappa_j$ be dyadic numbers such that $\kappa_j\leq\lambda_j\leq \mathbf{M}_j(\mathscr{R}_j)$. Then the counting function $N_{\mathscr{F}_j(\lambda_j, \kappa_j)}$ satisfies
\begin{equation*}
\|N_{\mathscr{F}_j(\lambda_j, \kappa_j)}\|_{p_j}\lesssim \kappa_j^{-2}|E_j|^{\frac{1}{p_j}},\ \forall 1\leq p_j<\infty.
\end{equation*}
\begin{proof}
Note that $\mathscr{F}_j(\lambda_j, \kappa_j)$, considered as a family of lacunary $j\text{-trees}$, is strongly disjoint. Thus we can apply Proposition \ref{Section4PropositionMassSelectionBesselTypeEstimate} and Corollary \ref{Section4CorollaryMassSelectionBesselTypeEstimateLocalVersion}. Then using Lemma \ref{Section4LemmaDyadicJohn-Nirenberg} as in the proof of \eqref{Section4EquationCountingFunctionEstimate} gives the desired estimate.
\end{proof}
\end{proposition}
Morally, we split each $j\text{-tree}\in \mathbf{F}_j(\lambda_j)$ into new trees, each corresponding to one element in $\mathscr{F}_j(\lambda_j, \kappa_j)$. But these new trees won't appear explicitly. We can then use $N_{\mathscr{F}_j(\lambda_j, \kappa_j)}, j\in [n]$ to estimate $N_{\mathbf{F}^{=1}(\lambda, \kappa)}$.
\begin{proposition}  \label{Section5PropositionCountingFunctionEstimate2}
Let $\lambda=(\lambda_1, \cdots, \lambda_n), \kappa=(\kappa_1, \cdots, \kappa_n)$ be two vectors of dyadic numbers with $\kappa_j\leq \lambda_j\leq \mathbf{M}_j(\mathscr{R}_j), \forall j\in [n]$. We can estimate $N_{\mathbf{F}^{=1}(\lambda, \kappa)}$ in terms of $N_{\mathscr{F}(\lambda_j, \kappa_j)}, j\in [n]$ as follows:
\begin{enumerate}
\item Assume we are in the first case of Theorem \ref{MainTheorem1}. Then
\begin{equation*}
N_{\mathbf{F}^{=1}(\lambda, \kappa)}\lesssim \prod_{j\in \Upsilon} N_{\mathscr{F}_j(\lambda_j, \kappa_j)},\ \forall \Upsilon\subset [n]: \#\Upsilon=\lceil \frac{m}{d} \rceil.
\end{equation*}
\item Assume we are in the second case of Theorem \ref{MainTheorem1}. Then
\begin{equation*}
N_{\mathbf{F}^{=1}(\lambda, \kappa)}\lesssim \prod_{j\in \Upsilon} N_{\mathscr{F}_j(\lambda_j, \kappa_j)}^{\frac{1}{2}},\ \forall \Upsilon\subset [n]: \#\Upsilon=n-1.
\end{equation*}
\item Assume we are in the third case of Theorem \ref{MainTheorem1}. Then
\begin{equation*}
N_{\mathbf{F}^{=1}(\lambda, \kappa)}\lesssim \prod_{j\in \Upsilon} N_{\mathscr{F}_j(\lambda_j, \kappa_j)}^{\frac{m}{d(n-1)}},\ \forall \Upsilon\subset [n]: \#\Upsilon=n-1.
\end{equation*}
\end{enumerate}
\end{proposition}
The proof of Proposition \ref{Section5PropositionCountingFunctionEstimate2} is almost identical to that of Proposition \ref{Section5PropositionCountingFunctionEstimate1}. We first define the ``projection maps''.
\begin{definition}
Let $\lambda=(\lambda_1, \cdots, \lambda_n), \kappa=(\kappa_1, \cdots, \kappa_n)$ be two vectors of dyadic numbers such that $\kappa_j\leq \lambda_j\leq \mathbf{M}_j(\mathscr{R}_j), \forall j\in [n]$. Let $\mathcal{P}\in \mathbf{F}^{=1}(\lambda, \kappa)$. For any $j\in [n]$, there exists at least one $(\mathcal{R}_{max})_j\in \mathscr{F}_j(\lambda_j, \kappa_j)$ satisfying $\mathcal{P}_j\leq (\mathcal{R}_{max})_j$. Fix one choice of $(\mathcal{R}_{max})_j$ and denote it by $\sigma_j^{'}(\mathcal{P})$. Let $\Upsilon\subset [n]$. We also write $\sigma_{\Upsilon}^{'}(\mathcal{P}):=\big(\sigma_j^{'}(\mathcal{P}) \big)_{j\in \Upsilon}$.
\end{definition}
Then we have the following analogue of Proposition \ref{Section5PropositionInjectiveMap1}, which immediately proves the first case of Proposition \ref{Section5PropositionCountingFunctionEstimate2}.
\begin{proposition}  \label{Section5PropositionInjectiveMap1ForSingle-ElementTrees}
Under Type \uppercase\expandafter{\romannumeral1} non-degenerate condition for $\Gamma$, the restricted map $\sigma_{\Upsilon}^{'}\vert_{\mathbf{F}^{=1, x}(\lambda, \kappa)}: $
\begin{equation*}
\mathbf{F}^{=1, x}(\lambda, \kappa)\rightarrow \prod_{j\in \Upsilon} \mathscr{F}_j^x(\lambda_j, \kappa_j)
\end{equation*}
is an injection for any $\Upsilon\in [n]: \#\Upsilon=\lceil \frac{m}{d}\rceil$.
\end{proposition}
\begin{proof}
The proof is the same as that of Proposition \ref{Section5PropositionInjectiveMap1}. This time, the implicit constants in \eqref{Section5EquationProvingPropositionInjectiveMap1}\eqref{Section5EquationProvingPropositionInjectiveMap2}\eqref{Section5EquationProvingPropositionInjectiveMap3}\eqref{Section5EquationProvingPropositionInjectiveMap4} don't depend on $C_2$. And we apply Lemma \ref{Section5LemmaKeyobservationForInjection2} instead.
\end{proof}
Moreover, Let $L=2$ for the second case and $L=d$ for the third case of Proposition \ref{Section5PropositionCountingFunctionEstimate2}. Apply Lemma \ref{Section5LemmaKatz-Tao} with $\Omega:=\mathbf{F}^{=1, x}(\lambda, \kappa)$, $Z_1=\cdots=Z_{L-1}:=\prod_{j\in A} \mathscr{F}_j^x(\lambda_j, \kappa_j)$, and $h_1=\cdots=h_{L-1}:=\sigma_{A}^{'}\vert_{\mathbf{F}^{=1, x}(\lambda, \kappa)}$. Define
\begin{equation*}
\mathbf{G}^{'}:=\big{\{}(\mathcal{P}^{(1)}, \cdots, \mathcal{P}^{(L)})\in \big( \mathbf{F}^{=1, x}(\lambda, \kappa) \big)^L: \sigma_{A}^{'}(\mathcal{P}^{(1)})=\cdots=\sigma_A^{'}(\mathcal{P}^{(L)})\big{\}}.
\end{equation*}
Arguing as before, the second and the third cases of Proposition \ref{Section5PropositionCountingFunctionEstimate2} follow from the analogue of Proposition \ref{Section5PropositionInjectiveMap2} below.
\begin{proposition}  \label{Section5PropositionInjectiveMap2ForSingle-ElementTrees}
Under Type \uppercase\expandafter{\romannumeral1} and Type \uppercase\expandafter{\romannumeral2} non-degenerate conditions for $\Gamma$, the restricted map $\prod_{k=1}^L\sigma_{B^{(k)}}^{'}\vert_{\mathbf{G}^{'}}:$
\begin{equation*}
\mathbf{G}^{'}\rightarrow \prod_{k=1}^L\prod_{j\in B^{(k)}}\mathscr{F}_j^x(\lambda_j, \kappa_j)
\end{equation*}
is an injection for any $A, B^{(k)}$ given in Definition \ref{Section1Non-DegenerateCondition2}.
\end{proposition} 
\begin{proof}
The proof is the same as that of Proposition \ref{Section5PropositionInjectiveMap2}. This time, the implicit constants in \eqref{Section5EquationProvingPropositionSection5PropositionInjectiveMap21}\eqref{Section5EquationProvingPropositionSection5PropositionInjectiveMap23}\eqref{Section5EquationProvingPropositionSection5PropositionInjectiveMap24}\eqref{Section5EquationProvingPropositionSection5PropositionInjectiveMap25}\eqref{Section5EquationProvingPropositionSection5PropositionInjectiveMap26} don't depend on $C_2$. And we apply Lemma \ref{Section5LemmaKeyobservationForInjection2}, Proposition \ref{Section5PropositionInjectiveMap1ForSingle-ElementTrees} instead.
\end{proof}
Thus we complete the proofs of Proposition \ref{Section5PropositionCountingFunctionEstimate1} and Proposition \ref{Section5PropositionCountingFunctionEstimate2}.


\section{Boundedness of Multilinear Singular Operators} \label{Section6}
Back to \eqref{Section3EquationMainGoalDiscreteVersion}, the tools developed in previous sections imply the following estimate:
\begin{proposition}  \label{Section6PropositionGenerivEstimate}
Let $\mathscr{R}$ and $\varphi_{\mathcal{R}_j}, \mathcal{R}\in \mathscr{R}, j\in [n]$ be as in Theorem \ref{MainTheorem1Discrete}. Assume $E_j, j\in [n]$ are bounded sets and $f_j, j\in [n]$ are functions satisfying $|f_j|\leq 1_{E_j}, \forall j\in [n]$. We write the estimates derived from Proposition \ref{Section5PropositionCountingFunctionEstimate1} and Proposition \ref{Section5PropositionCountingFunctionEstimate2} into a unified form:
\begin{equation}  \label{Section6EquationCountingFunctionEstimateUnified}
N_{\mathbf{F}}\lesssim \prod_{j=1}^n N_{\mathbf{F}_j}^{\alpha_j}.
\end{equation}
Here $(\alpha_1, \cdots, \alpha_n)\in (0, \frac{1}{2})^n$ is a fixed tuple. We will determine its precise value in each case later. Suppose $\theta_j\in (0, \alpha_j], j\in [n]$ satisfy $\sum_{j=1}^n \theta_j=1$. Then we have
\begin{equation*}
\sum_{\mathcal{R}\in \mathscr{R}} |I(\mathcal{R})|^{1-\frac{n}{2}}\prod_{j=1}^n |\langle f_j | \varphi_{\mathcal{R}_j} \rangle|\lesssim \prod_{j=1}^n \mathbf{M}_j(\mathscr{R}_j)^{1-2\alpha_j}|E_j|^{\theta_j}.
\end{equation*}
\end{proposition}
\begin{proof}
Organize $\mathscr{R}$ into vector trees as in the previous section:
\begin{equation}  \label{Section6EquationOrganizingTheCollection}
\begin{split}
\sum_{\mathcal{R}\in \mathscr{R}} |I(\mathcal{R})|^{1-\frac{n}{2}}\prod_{j=1}^n |\langle f_j | \varphi_{\mathcal{R}_j} \rangle| & =\sum_{\lambda}\sum_{\mathcal{R}\in \mathscr{R}(\lambda)} |I(\mathcal{R})|^{1-\frac{n}{2}}\prod_{j=1}^n |\langle f_j | \varphi_{\mathcal{R}_j} \rangle|  \\
&=\sum_{\lambda}\sum_{\mathscr{T}\in \mathbf{F}^{\geq 2}(\lambda)}\sum_{\mathcal{R}\in \mathscr{T}} \cdots+\sum_{\lambda}\sum_{\kappa}\sum_{\mathcal{R}\in \mathbf{F}^{=1}(\lambda, \kappa)} \cdots,
\end{split}   
\end{equation}
where the vectors of dyadic numbers $\lambda=(\lambda_1, \cdots, \lambda_n), \kappa=(\kappa_1, \cdots, \kappa_n)$ satisfy $\kappa_j\leq \lambda_j\leq \mathbf{M}_j(\mathscr{R}_j), \forall j\in [n]$. Apply Lemma \ref{Section4LemmaTreeEstimate}, Proposition \ref{Section4PropositionMassSelection1}, \eqref{Section6EquationCountingFunctionEstimateUnified}, H\"older's inequality, and Proposition \ref{Section4PropositionMassSelection1} successively. We can estimate the first term on the right of \eqref{Section6EquationOrganizingTheCollection} as follows:
\begin{equation} \label{Section6EquationEstimatingTheOrganizedCollection1}
\begin{split}
\sum_{\lambda}\sum_{\mathscr{T}\in \mathbf{F}^{\geq 2}(\lambda)}\sum_{\mathcal{R}\in \mathscr{T}} \cdots & \lesssim \sum_{\lambda}\sum_{\mathscr{T}\in \mathbf{F}^{\geq 2}(\lambda)} |I_{\mathscr{T}}|\prod_{j=1}^n \lambda_j=\sum_{\lambda} (\prod_{j=1}^n \lambda_j)\|N_{ \mathbf{F}^{\geq 2}(\lambda)}\|_1  \\
&\lesssim \sum_{\lambda} (\prod_{j=1}^n \lambda_j)\|N_{\mathbf{F}_j(\lambda_j)}^{\alpha_j}\|_1 \leq \sum_{\lambda} (\prod_{j=1}^n \lambda_j)(\prod_{j=1}^n\|N_{\mathbf{F}_j(\lambda_j)}\|_{\frac{\alpha_j}{\theta_j}}^{\alpha_j})  \\
&\lesssim \sum_{\lambda} (\prod_{j=1}^n \lambda_j)(\prod_{j=1}^n \lambda_j^{-2\alpha_j}|E_j|^{\theta_j}) \lesssim \prod_{j=1}^n \mathbf{M}_j(\mathscr{R}_j)^{1-2\alpha_j}|E_j|^{\theta_j}.
\end{split}
\end{equation}
Similarly, apply the definition of $\mathbf{F}^{=1}(\lambda, \kappa)$, \eqref{Section6EquationCountingFunctionEstimateUnified}, H\"older's inequality, and Proposition \ref{Section5PropositionCountingFunctionEstimateForSingle-ElementTrees} successively. We get the second term on the right of \eqref{Section6EquationOrganizingTheCollection} is bounded by the same quantity:
\begin{equation}  \label{Section6EquationEstimatingTheOrganizedCollection2}
\sum_{\lambda}\sum_{\kappa}\sum_{\mathcal{R}\in \mathbf{F}^{=1}(\lambda, \kappa)} \cdots\lesssim \prod_{j=1}^n \mathbf{M}_j(\mathscr{R}_j)^{1-2\alpha_j}|E_j|^{\theta_j}.
\end{equation}
Combining \eqref{Section6EquationOrganizingTheCollection}, \eqref{Section6EquationEstimatingTheOrganizedCollection1}, and \eqref{Section6EquationEstimatingTheOrganizedCollection2} gives the desired estimate.
\end{proof}
To handle $f_j\in L^{p_j}$ with a small $p_j$, we remove the peaks of $f_j$ and carefully exploit the gain from such removal as in \cite{MTTMultilinearSingularMultipliers}.
\begin{proposition}  \label{Section6PropositionMainTheoremRestrictedWeak}
Keep the conditions of Proposition \ref{Section6PropositionGenerivEstimate}. Let $C>0$ be a constant. Further take $f_n:=1_{E_n\backslash \Omega}$ with $\Omega$ defined by
\begin{equation}  \label{Section6PropositionMainTheoremRestrictedWeakDefiningTheExceptionalSet}
\Omega:=\{x\in \mathbb{R}^d: M\big( \frac{1_{E_j}}{|E_j|}\big)(x)>C|E_n|^{-1} \text{\ for some $j\in [n-1]$.}\}
\end{equation}
Then we have
\begin{equation*}
\sum_{\mathcal{R}\in \mathscr{R}} |I(\mathcal{R})|^{1-\frac{n}{2}} \big(\prod_{j=1}^{n-1} |\langle f_j | \varphi_{\mathcal{R}_j} \rangle| \big)|\langle 1_{E_n\backslash \Omega} | \varphi_{\mathcal{R}_n} \rangle|\lesssim \big(\prod_{j=1}^{n-1} |E_j|^{1-2\alpha_j+\theta_j} \big)|E_n|^{1-\sum_{j=1}^{n-1} (1-2\alpha_j+\theta_j)}.
\end{equation*}
\end{proposition}
\begin{proof}
For $l\in \mathbb{N}$, define
\begin{equation*}
\mathscr{R}^{(l)}:=\big{\{} \mathcal{R}\in \mathscr{R}: 2^l\leq 1+\frac{d\big( I(\mathcal{R}), \Omega^c \big)}{s\big( I(\mathcal{R}) \big)}<2^{l+1} \big{\}}.
\end{equation*}
We shall apply Proposition \ref{Section6PropositionGenerivEstimate} to each $\mathscr{R}^{(l)}$ instead of $\mathscr{R}$. Corollary \ref{Section4CorollaryMassEstimate2} and \eqref{Section6PropositionMainTheoremRestrictedWeakDefiningTheExceptionalSet} imply
\begin{gather*}
\mathbf{M}_j(\mathscr{R}_j^{(l)})\lesssim 2^{dl} |E_j||E_n|^{-1},\ \forall j\in [n-1],  \\
\mathbf{M}_n(\mathscr{R}_n^{(l)})\lesssim 2^{-N_{\epsilon}l}.
\end{gather*}
for each $l\in \mathbb{N}$. Thus combining these two estimates and Proposition \ref{Section6PropositionGenerivEstimate} gives
\begin{equation*}
\sum_{\mathcal{R}\in \mathscr{R}^{(l)}} |I(\mathcal{R})|^{1-\frac{n}{2}} \big(\prod_{j=1}^{n-1} |\langle f_j | \varphi_{\mathcal{R}_j} \rangle| \big)|\langle 1_{E_n\backslash \Omega} | \varphi_{\mathcal{R}_n} \rangle|\lesssim 2^{-l}\big(\prod_{j=1}^{n-1} |E_j|^{1-2\alpha_j+\theta_j} \big)|E_n|^{1-\sum_{j=1}^{n-1} (1-2\alpha_j+\theta_j)}
\end{equation*}
provided we take $N_{\epsilon}$ sufficiently large depending on $\alpha_n$ (which is defined in terms of $\epsilon$ below). Summing the above estimate over $l$ completes the proof.
\end{proof}
How to use the $\alpha_j, \theta_j, j\in [n]$ in Proposition \ref{Section6PropositionMainTheoremRestrictedWeak} to generate all allowed tuples $(\frac{1}{p_1}, \cdots, \frac{1}{p_{n-1}}, \frac{1}{q^{'}})$ in Theorem \ref{MainTheorem1}? We apply the following lemma from \cite{MTTMultilinearSingularMultipliers}. The proof there lacks some details, so we include a proof.
\begin{lemma} \label{Section6LemmaConvexHull}
Let $a_1, \cdots, a_n\in \mathbb{R}$ satisfy $a_1\geq \cdots\geq a_n$. Then the closed convex hull of all permutations of $(a_1, \cdots, a_n)$ is
\begin{equation*}
\mathcal{Z}:=\{(x_1, \cdots, x_n)\in \mathbb{R}^n: \sum_{j\in \Upsilon} x_j\leq \sum_{j=1}^{\#\Upsilon} a_j, \forall \Upsilon\subsetneqq [n]\text{\ and\ } \sum_{j=1}^n x_j=\sum_{j=1}^n a_j\}.    
\end{equation*}
\end{lemma}
\begin{proof}
It's easy to see the convex hull is contained in $\mathcal{Z}$. We prove the reverse inclusion. By the Krein-Milman theorem, we only need to show that any the extreme point of $\mathcal{Z}$ is a permutation of $a:=(a_1, \cdots, a_n)$. Given any $x=(x_1, \cdots, x_n)\in \mathbb{R}^n$, define
\begin{equation*}
S_l(x):=\sum_{k=1}^l x_{j_k},\ 1\leq l\leq n,
\end{equation*}
where $x_{j_1}\geq \cdots\geq x_{j_n}$. Let $y=(y_1, \cdots, y_n)$ be an extreme point of $\mathcal{Z}$ such that $y_1\geq \cdots\geq y_n$. it suffices to prove
\begin{equation*}
S_l(y)=S_l(a),\ \forall 1\leq l\leq n-1.
\end{equation*}
Assume this is not true. Then there exists $l$ such that $S_l(y)<S_l(a)$. Let $L$ be the largest number having this property. Then $y_{L+1}>a_{L+1}$ and $y_j=a_j$ for any $L+2\leq j\leq n$. Below we prove the point $\title{y}:=(y_1, \cdots, y_{L-1}, y_L+t, y_{L+1}-t, y_{L+2}, \cdots, y_n)\in \mathcal{Z}$ provided $|t|$ is sufficiently small, which contradicts the fact that $y$ is an extreme point.

Let $L^{'}\leq L$ be the integer such that $y_{L^{'}-1}>y_{L^{'}}=\cdots=y_{L}$. First consider the $t>0$ case. Since $|t|$ is very small, we have
\begin{equation*}
S_l(\title{y})=S_l(y)\leq S_l(a),\ \forall 1\leq l\leq L^{'}-1.
\end{equation*}
For $L^{'}\leq l\leq L$, note that $S_l(y)<S_l(a)$ must hold. Otherwise we have $S_{l_1}(y)=S_{l_1}(a)$ for some $L^{'}\leq l_1\leq L$, which implies $a_{l_1}\leq y_{l_1}$. On the other hand, $S_L(y)<S_L(a)$ implies there exists some $l_2: l_1<l_2\leq L$ such that $a_{l_2}>y_{l_2}$. Thus
\begin{equation*}
a_{l_1}\leq y_{l_1}=y_{l_2}<a_{l_2},
\end{equation*}
a contradiction. Now since $S_l(y)<S_l(a)$ for any $L^{'}\leq l\leq L$, we can guarantee 
\begin{equation*}
S_l(\title{y})\leq S_l(a),\ \forall L^{'}\leq l\leq L.
\end{equation*}
by taking $|t|$ very small. For $L+1\leq l\leq n$, we use $y_{L+1}>a_{L+1}\geq a_{L+2}=y_{L+2}$. Hence for sufficiently small $t$, we have
\begin{equation*}
S_l(\title{y})=S_{L+1}(y)+(y_{L+2}+\cdots+y_{l})\leq S_{L+1}(a)+(a_{L+2}+\cdots+a_{l})=S_{l}(a).
\end{equation*}
We have proved $\tilde{y}\in \mathcal{Z}$ in the $t>0$ case. Now we consider the $t<0$ case. If $y_L>y_{L+1}$, we still have $y_1\geq \cdots \geq y_{L}+t>y_{L+1}-t\geq \cdots\geq y_n$ provided $|t|$ is sufficiently small. And the conclusion follows. If If $y_L=y_{L+1}$, change the position of $y_L$ and $y_{L+1}$. Then we are back to the $t>0$ case. This finishes the proof of Lemma \ref{Section6LemmaConvexHull}.
\end{proof}
\begin{corollary}  \label{Section6CorollaryFinalBoundRestrictedWeak}
Let $\mathscr{R}$ and $\varphi_{\mathcal{R}_j}, \mathcal{R}\in \mathscr{R}, j\in [n]$ be as in Theorem \ref{MainTheorem1Discrete}. Assume $E_j, j\in [n]$ are bounded sets and $f_j, j\in [n]$ are functions satisfying $|f_j|\leq 1_{E_j}, \forall j\in [n-1]$ and $f_n:=1_{E_n}\backslash \Omega$ with $\Omega$ defined by \eqref{Section6PropositionMainTheoremRestrictedWeakDefiningTheExceptionalSet}. For each case of Theorem \ref{MainTheorem1}, let $(p_1, \cdots, p_{n-1}, q)$ be any tuple in the corresponding range. Then we have
\begin{equation}  \label{Section6EquationFinalBoundRestrictedWeak}
\sum_{\mathcal{R}\in \mathscr{R}} |I(\mathcal{R})|^{1-\frac{n}{2}} \big(\prod_{j=1}^{n-1} |\langle f_j | \varphi_{\mathcal{R}_j} \rangle| \big)|\langle 1_{E_n\backslash \Omega} | \varphi_{\mathcal{R}_n} \rangle|\lesssim \big(\prod_{j=1}^{n-1} |E_j|^{\frac{1}{p_j}} \big)|E_n|^{\frac{1}{q^{'}}}.
\end{equation}
\end{corollary}
\begin{proof}
Let $0<\epsilon<\frac{\min\{\frac{n}{2}-\frac{m}{d}, 1\}}{100n}$ be a constant. In each case of Theorem \ref{MainTheorem1}, define $(\alpha_1, \cdots, \alpha_n)$ to be the following tuple or its permutation:
\begin{enumerate}
\item In the first case of Theorem \ref{MainTheorem1}, let
\begin{equation*}
\alpha_1=\cdots=\alpha_{n-2\lceil \frac{m}{d} \rceil}:=\frac{2\lceil \frac{m}{d} \rceil \epsilon}{n-2\lceil \frac{m}{d} \rceil}
\end{equation*}
and $\alpha_{n-2\lceil \frac{m}{d} \rceil+1}=\cdots=\alpha_n:=\frac{1}{2}-\epsilon$.
\item In the second case of Theorem \ref{MainTheorem1}, let $\alpha_1:=(n-1)\epsilon$ and $\alpha_2=\cdots=\alpha_n:=\frac{1}{2}-\epsilon$.
\item In the third case of Theorem \ref{MainTheorem1}, let $\alpha_1:=\frac{m}{d}-\frac{n-1}{2}+(n-1)\epsilon$ and $\alpha_2=\cdots=\alpha_n:=\frac{1}{2}-\epsilon$.
\end{enumerate}
We need to show that our choices of $(\alpha_1, \cdots, \alpha_n)$ make \eqref{Section6EquationCountingFunctionEstimateUnified} hold. We argue as follows:
\begin{enumerate}
\item In the first case of Theorem \ref{MainTheorem1}, Proposition \ref{Section5PropositionCountingFunctionEstimate1} and Proposition \ref{Section5PropositionCountingFunctionEstimate2} imply we can let $(\alpha_1, \cdots, \alpha_n)$ have $\lceil \frac{m}{d} \rceil$ many $1$'s and $n-\lceil \frac{m}{d} \rceil$ many $0$'s. On the one hand, we take a geometric average between two disjoint $\Upsilon$. Then we get $(\alpha_1, \cdots, \alpha_n)$ having $2\lceil \frac{m}{d} \rceil$ many $(\frac{1}{2})$'s and $n-2\lceil \frac{m}{d} \rceil$ many $0$'s. On the other hand, we take the geometric average over all possible $\Upsilon$. Then we get $(\alpha_1, \cdots, \alpha_n)$ with all elements equal to $\frac{\lceil \frac{m}{d} \rceil}{n}$. Finally, taking the geometric average between these two new types of $(\alpha_1, \cdots, \alpha_n)$ gives the one stated in Proposition \ref{Section6PropositionGenerivEstimate}.
\item In the second case of Theorem \ref{MainTheorem1}, Proposition \ref{Section5PropositionCountingFunctionEstimate1} and Proposition \ref{Section5PropositionCountingFunctionEstimate2} imply we can let $(\alpha_1, \cdots, \alpha_n)$ have $n-1$ many $(\frac{1}{2})$'s and one $0$s. Then taking the geometric average over all possible $\Upsilon$ gives $(\alpha_1, \cdots, \alpha_n)$ with all elements equal to $\frac{n-1}{2n}$. Finally, take the average between this new type of $(\alpha_1, \cdots, \alpha_n)$ and the old type of $(\alpha_1, \cdots, \alpha_n)$ corresponding to one fixed $\Upsilon$. We get the one stated in Proposition \ref{Section6PropositionGenerivEstimate}.
\item In the third case of Theorem \ref{MainTheorem1}, we argue as in the second case.
\end{enumerate}
Then we use Proposition \ref{Section6PropositionMainTheoremRestrictedWeak} with $(\theta_1, \cdots, \theta_n)$ defined as follows:
\begin{enumerate}
\item In the first case of Theorem \ref{MainTheorem1}, let $\theta_j:=\alpha_j, 1\leq j\leq n-2\lceil \frac{m}{d} \rceil+1$, $\theta_{n-2\lceil \frac{m}{d} \rceil+2}:=\frac{1}{2}-(4\lceil \frac{m}{d} \rceil-3)\epsilon$, and $\theta_{n-2\lceil \frac{m}{d} \rceil+3}=\cdots=\theta_n=\epsilon$. Then
\begin{equation*}
1-2\alpha_j+\theta_j=
\begin{cases}
1-O(\epsilon),\ 1\leq j\leq n-2\lceil \frac{m}{d} \rceil, \\
\frac{1}{2}+O(\epsilon),\ j=n-2\lceil \frac{m}{d} \rceil+1, \\
\frac{1}{2}-O(\epsilon),\ j=n-2\lceil \frac{m}{d} \rceil+2, \\
O(\epsilon),\ n-2\lceil \frac{m}{d} \rceil+3\leq j\leq n-1.
\end{cases}
\end{equation*}
Here we only have the first two lines if $\lceil \frac{m}{d} \rceil=1$.
\item In the second case of Theorem \ref{MainTheorem1}, let $\theta_j:=\alpha_j, 1\leq j\leq 2$, $\theta_3:=\frac{1}{2}-(2n-5)\epsilon$, and $\theta_4=\cdots=\theta_n=\epsilon$. Then
\begin{equation*}
1-2\alpha_j+\theta_j=
\begin{cases}
1-O(\epsilon),\ j=1, \\
\frac{1}{2}+O(\epsilon),\ j=2, \\
\frac{1}{2}-O(\epsilon),\ j=3, \\
O(\epsilon),\ 4\leq j\leq n-1.
\end{cases}
\end{equation*}
\item In the third case of Theorem \ref{MainTheorem1}, let $(\theta_1, \cdots, \theta_n)$ be the same as in the second case. Then
\begin{equation*}
1-2\alpha_j+\theta_j=
\begin{cases}
\frac{n+1}{2}-\frac{m}{d}-O(\epsilon),\ j=1, \\
\frac{1}{2}+O(\epsilon),\ j=2, \\
\frac{n}{2}-\frac{m}{d}-O(\epsilon),\ j=3, \\
O(\epsilon),\ 4\leq j\leq n-1.
\end{cases}
\end{equation*}
\end{enumerate}
In each case, we see
\begin{equation*}
\sum_{j\in \Upsilon} \frac{1}{p_j}\leq \sum_{j=1}^{\#\Upsilon} (1-2\alpha_j+\theta_j),\ \forall \Upsilon\subset [n-1],
\end{equation*}
as long as $\epsilon$ is sufficiently small. And
\begin{equation*}
(\sum_{j=1}^{n-1} \frac{1}{p_j})+\frac{1}{q^{'}}=\big(\sum_{j=1}^{n-1} (1-2\alpha_j+\theta_j) \big)+\big(1-\sum_{j=1}^{n-1} (1-2\alpha_j+\theta_j) \big)=1.
\end{equation*}
Thus we can obtain the desired estimate for $(p_1, \cdots, p_{n-1}, q)$ by applying Lemma \ref{Section6LemmaConvexHull} and taking a geometric average over the estimates of Proposition \ref{Section6PropositionMainTheoremRestrictedWeak}.
\end{proof}
Recall the following classical lemma for weak $L^q$ norms \cite{MS2}:
\begin{lemma}  \label{Section6LemmaWeakLebesgueSpace}
Let $(X, \mu)$ be a $\sigma\text{-finite}$ space. Let $0<q<\infty$. Then for any function $f$ we have
\begin{equation*}
\|f\|_{q, \infty}\sim \sup_{0<\mu(E)<\infty} \inf_{E^{'}\subset E: \mu(E^{'})\sim \mu(E)} |\int_{E^{'}} f|\mu(E)^{-\frac{1}{q^{'}}}.
\end{equation*}
\end{lemma}
\begin{proof}
For the $\lesssim$ direction, assume $RHS=1$. Let $\lambda>0$. Take $E$ to be each of the following sets (intersecting with a finite-measure set if necessary):
\begin{equation*}
\{x: \Re f>\frac{1}{2}\lambda\},\ \{x: \Re f<-\frac{1}{2}\lambda\},\ \{x: \Im f>\frac{1}{2}\lambda\},\ \{x: \Im f<-\frac{1}{2}\lambda\}.
\end{equation*}
We get
\begin{equation*}
\lambda \mu(\{x: |f|>\lambda\})^{\frac{1}{q}}\lesssim 1
\end{equation*}
as desired. For the $\gtrsim$ direction, assume $LHS=1$. Given $E$, define
\begin{equation*}
\Omega:=\{x: |f|>C\mu(E)^{-\frac{1}{q}}\}    
\end{equation*}
with a large constant $C$ and take $E^{'}:=E\backslash\Omega$. We see $\mu(E^{'})\sim \mu(E)$ and
\begin{equation*}
|\int_{E^{'}} f|\lesssim \mu(E)^{\frac{1}{q^{'}}},
\end{equation*}
which finishes the proof.
\end{proof}
By Lemma \ref{Section6LemmaWeakLebesgueSpace}, we see \eqref{Section6EquationFinalBoundRestrictedWeak} is equivalent to
\begin{equation*}
\|T_{\mathscr{R}}^{dis}(f_1, \cdots, f_{n-1})\|_{q, \infty}\lesssim \prod_{j=1}^{n-1} |E_j|^{\frac{1}{p_j}}.
\end{equation*}
Applying multilinear interpolation then gives Theorem \ref{MainTheorem1Discrete}. Moreover, \eqref{Section6EquationFinalBoundRestrictedWeak} implies the similar estimate \eqref{Section3EquationTensorizedMultilinearFormNewBounds} for $\widetilde{\Lambda_{\mathfrak{m}}}$. Then follow the approximation procedure stated below \eqref{Section3EquationTensorizedMultilinearFormNewBounds}. We also completes the proof of Theorem \ref{MainTheorem1}.


\section{The Mildly Degenerate Case} \label{Section7}
To tackle Example \ref{Section1ExampleFractionalRankDegenerateMildly}, we do the entirely same argument as Section \ref{Section3} and Section \ref{Section4}. Recall its singular space $\Gamma$ in \eqref{Section1EquationMildlyDegenerateSingularSpace} is degenerate, so only partial estimates in Section \ref{Section5} are valid. First, $\Gamma$ satisfies Type \uppercase\expandafter{\romannumeral1} non-degenerate condition. Hence Proposition \ref{Section5PropositionInjectiveMap1} and Proposition \ref{Section5PropositionInjectiveMap1ForSingle-ElementTrees} are valid:
\begin{align}
N_{\mathbf{F}^{\geq 2}(\lambda)} & \lesssim \prod_{j\in \Upsilon} N_{\mathbf{F}_j(\lambda_j)},\ \forall \Upsilon\subset [n]: \#\Upsilon=2,  \label{Section7SmallRankTreeCounting1}  \\
N_{\mathbf{F}^{=1}(\lambda, \kappa)} & \lesssim \prod_{j\in \Upsilon} N_{\mathscr{F}_j(\lambda_j, \kappa_j)},\ \forall \Upsilon\subset [n]: \#\Upsilon=2. \label{Section7SmallRankTreeCounting2}
\end{align}
Moreover, we calculate the kernels of $P_{\{j\}}: \Gamma\to \mathbb{R}^2$:
\begin{align*}
\ker(P_{\{1\}}) & =\big{\{} \big( (0, 0), (0, \tau^{''}), (-\tau^{''}, 0), (\tau^{''}, -\tau^{''})\big): \tau^{''}\in \mathbb{R} \big{\}},  \\
\ker(P_{\{2\}}) & =\big{\{} \big( (\tau, 0), (0, 0), (0, -\tau), (-\tau, \tau)\big): \tau\in \mathbb{R} \big{\}},  \\
\ker(P_{\{3\}}) & =\big{\{} \big( (0, \tau^{'}), (-\tau^{'}, 0), (0, 0), (\tau^{'}, \tau^{'})\big): \tau^{'}\in \mathbb{R} \big{\}},  \\
\ker(P_{\{4\}}) & =\big{\{} \big( (\tau^{'}+\tau^{''}, \tau^{'}), (-\tau^{'}, \tau^{''}), (-\tau^{''}, -\tau^{'}-\tau^{''}), (0, 0)\big): \tau^{'}, \tau^{''}\in \mathbb{R} \big{\}}.
\end{align*}
We see $\ker(\mathfrak{L})=\{0\}$ for $A=\{3\}, B^{(1)}=\{1\}, B^{(2)}=\{2\}$. Hence the proofs of Proposition \ref{Section5PropositionInjectiveMap2} and Proposition \ref{Section5PropositionInjectiveMap2ForSingle-ElementTrees} are valid for $A=\{3\}, B^{(1)}=\{1\}, B^{(2)}=\{2\}$. And we get
\begin{align}
N_{\mathbf{F}^{\geq 2}(\lambda)} & \lesssim \prod_{j=1}^3 N_{\mathbf{F}_j(\lambda_j)}^{\frac{1}{2}}, \label{Section7MediumRankTreeCounting1} \\
N_{\mathbf{F}^{=1}(\lambda, \kappa)} & \lesssim \prod_{j=1}^3 N_{\mathscr{F}_j(\lambda_j, \kappa_j)}^{\frac{1}{2}}. \label{Section7MediumRankTreeCounting2}
\end{align}
We should only use \eqref{Section7SmallRankTreeCounting1} and \eqref{Section7SmallRankTreeCounting2} with $\Upsilon$ containing $4$ because we have the more efficient \eqref{Section7MediumRankTreeCounting1} and \eqref{Section7MediumRankTreeCounting2}. Taking geometric averages of \eqref{Section7SmallRankTreeCounting1} and \eqref{Section7MediumRankTreeCounting1}, \eqref{Section7SmallRankTreeCounting2} and \eqref{Section7MediumRankTreeCounting2} gives
\begin{align*}
N_{\mathbf{F}^{\geq 2}(\lambda)} & \lesssim (\prod_{j=1}^3 N_{\mathbf{F}_j(\lambda_j)}^{\frac{1}{2}\beta+\beta_j})N_{\mathbf{F}_4(\lambda_4)}^{\beta_1+\beta_2+\beta_3},  \\
N_{\mathbf{F}^{=1}(\lambda, \kappa)} & \lesssim (\prod_{j=1}^3 N_{\mathscr{F}_j(\lambda_j, \kappa_j)}^{\frac{1}{2}\beta+\beta_j})N_{\mathscr{F}_4(\lambda_4, \kappa_4)}^{\beta_1+\beta_2+\beta_3},
\end{align*}
where $0\leq \beta, \beta_1, \beta_2, \beta_3\leq 1$ satisfy $\beta+\beta_1+\beta_2+\beta_3=1$. We can then arrange each exponent smaller than $\frac{1}{2}$ and argue as in Section \ref{Section6}. This proves the boundedness of $T_{\mathfrak{m}}$. The range of $(p_1, p_2,p_3, q)$ is smaller than that in the second case of Theorem \ref{MainTheorem1}. We only write down the symmetric bounds for simplicity.
\begin{theorem}
Let $T_{\mathfrak{m}}$ be the operator given in Example \ref{Section1ExampleFractionalRankDegenerateMildly} and $2<p<\infty$. Then we have
\begin{equation*}
\|T_{\mathfrak{m}}(f_1, f_2,f_3)\|_{\frac{p}{3}}\lesssim \prod_{j=1}^3 \|f_j\|_{p},\ \forall f_1, f_2, f_3\in \mathcal{S}(\mathbb{R}^2).
\end{equation*}
\end{theorem}

\section{Final Remarks} \label{Section8}
\begin{enumerate}
\item In our discretization process Section \ref{Section3}, $\mathfrak{m}$ behaves like a bump function on each Whitney cube. This means $\mathfrak{m}$ has been tensorized on each cube and the following Fourier expansion procedure may not be necessary. We include this procedure because it is free and makes the operator have a concise form.
\item Our original discrete sum \eqref{Section3EquationAlmostFinalDiscreteModelSum} is a linear combination of \eqref{Section3EquationMainGoalDiscreteVersion}, so it's necessary to make the exceptional set \eqref{Section6PropositionMainTheoremRestrictedWeakDefiningTheExceptionalSet} only depend on $E_j$.
\item Our way to define Type \uppercase\expandafter{\romannumeral2} non-degenerate condition is not unique. It's possible to find similar conditions of the same spirit.
\item Further assume $(p_1, \cdots, p_{n-1}, q)\in (1, \infty)^n$ in each case of Theorem \ref{MainTheorem1}. Then we can simplify the proof using \eqref{Section4EquationCountingFunctionEstimateBelow1}: Write $p_n:=q^{'}$ as usual. We choose $\theta_j:=\frac{1}{p_j}, j\in [n]$ in \eqref{Section6EquationEstimatingTheOrganizedCollection1} and apply \eqref{Section4EquationCountingFunctionEstimateBelow1} if necessary. We can suitably select $\alpha_j$ to guarantee the exponents of $\lambda_j$ is positive. Then simply using $\mathbf{M}_j(\mathscr{R}_j)\lesssim 1$ gives
\begin{equation*}
\sum_{\lambda}\sum_{\mathscr{T}\in \mathbf{F}^{\geq 2}(\lambda)}\sum_{\mathcal{R}\in \mathscr{T}} \cdots\lesssim \prod_{j=1}^n |E_j|^{\frac{1}{p_j}}.  
\end{equation*}
Similarly, we get the estimate:
\begin{equation*}
\sum_{\lambda}\sum_{\kappa}\sum_{\mathcal{R}\in \mathbf{F}^{=1}(\lambda, \kappa)} \cdots\lesssim \prod_{j=1}^n |E_j|^{\frac{1}{p_j}}.
\end{equation*}
This finishes the proof without removing the exceptional set \eqref{Section6PropositionMainTheoremRestrictedWeakDefiningTheExceptionalSet}.
\item If Type \uppercase\expandafter{\romannumeral1} or Type \uppercase\expandafter{\romannumeral2} non-degenerate condition fails mildly, then we can argue like Section \ref{Section7} to show the boundedness of $T_{\mathfrak{m}}$ in a certain range. However, we chose to only discuss one example since it is tedious to completely describe which degenerate cases can be handled.
\item For Example \ref{Section1ExampleFractionalRankSmalln}, we can also anticipate its difficulty from a projection perspective. The following is morally a generalization of Theorem 1.17 and Theorem 1.18 in \cite{BDLClassificationOfTrilinear}. Let $d, \Delta$ be positive integers and $\mathfrak{L}_j: \mathbb{R}^{\Delta}\rightarrow \mathbb{R}^{d}, j\in [n-1]$ be linear operators. Let $f_j\in L^{\infty}(\mathbb{R}^{d})$ and $\rho\in L^{\infty}(\mathbb{R}^{\Delta})$ such that $\rho$ is compactly supported. Define
\begin{equation*}
T_{(\mathfrak{L}_j)_j, \rho}(f_1, \cdots, f_{n-1})(x):=\int_{\mathbb{R}^{\Delta}} \prod_{j=1}^{n-1} f_j(x+\mathfrak{L}_jt) \rho(t) dt.    
\end{equation*}
\begin{lemma}  \label{Section8LemmaProjectionLemma}
Let $d^{'}<d$ be positive integers and $A: \mathbb{R}^{d}\rightarrow \mathbb{R}^{d^{'}}$ be a linear surjection. Fix $0<p_j, q<\infty$ such that $\frac{1}{q}=\sum_{j=1}^{n-1}\frac{1}{p_j}$. Suppose
\begin{equation}  \label{Section8LemmaProjectionLemmaAssumption}
\|T_{(\mathfrak{L}_j)_j, \rho}(F_1, \cdots, F_{n-1})\|_{q}\leq C\prod_{j=1}^{n-1}\|F_j\|_{p_j},\ \forall F_j\in L^{\infty}(\mathbb{R}^d)\cap L^{p_j}(\mathbb{R}^d).
\end{equation}
Then we have
\begin{equation}  \label{Section8LemmaProjectionLemmaConclusion}
\|T_{(A\mathfrak{L}_j)_j, \rho}(f_1, \cdots, f_{n-1})\|_{q}\lesssim_{d, d^{'}, n, (\mathfrak{L})_j, A} C\prod_{j=1}^{n-1}\|f_j\|_{p_j},\ \forall f_j\in L^{\infty}(\mathbb{R}^{d^{'}})\cap L^{p_j}(\mathbb{R}^{d^{'}}).
\end{equation}
\begin{proof}
Denote the ball $\{x: |x|<R\}$ by $B_R$. Assume $\rho$ is supported in $B_{R_0}$ for some $R_0$. Let $R\gg R_0$. For given $f_j\in L^{\infty}(\mathbb{R}^{d^{'}})\cap L^{p_j}(\mathbb{R}^{d^{'}})$, take
\begin{equation*}
F_{j, R}(x):=f_{j}(Ax)1_{B_R}(x).    
\end{equation*}
We have $|x+\mathfrak{L}_jt|\leq R$ for $|x|\leq cR, |t|\leq R_0$, which implies
\begin{equation*}
T_{(\mathfrak{L}_j)_j, \rho}(F_{1, R}, \cdots, F_{n-1, R})(x)=T_{(A\mathfrak{L}_j)_j, \rho}(f_1, \cdots, f_{n-1})(Ax),\ \forall |x|\lesssim R.   
\end{equation*}
Write $x=y+z$ with $y\in \ker(A)^{\perp}, z\in \ker(A)$. Then we obtain
\begin{equation}  \label{Section8LemmaProjectionLemmaProof1}
\begin{split}
\int_{\mathbb{R}^d} |T_{(\mathfrak{L}_j)_j, \rho}(F_{1, R}, \cdots, & F_{n-1, R}) (x)|^q dx \geq \int_{B_{cR}} |T_{(A\mathfrak{L}_j)_j, \rho}(f_{1}, \cdots, f_{n-1})(Ax)|^q dx \\
& \geq \int_{\ker (A)\cap B_{\frac{c}{2}R}} \int_{\ker(A)^{\perp}\cap B_{\frac{c}{2}R}} |T_{(A\mathfrak{L}_j)_j, \rho}(f_{1}, \cdots, f_{n-1})(Ay)|^q dy dz \\
& \gtrsim R^{d-d^{'}} \int_{B_{c^{'}R}} T_{(A\mathfrak{L}_j)_j, \rho}(f_{1}, \cdots, f_{n-1})(x^{'})|^q dx^{'}.
\end{split}    
\end{equation}
Similarly, we have
\begin{equation}   \label{Section8LemmaProjectionLemmaProof2}
\int_{\mathbb{R}^d} |F_{j, R}|^{p_j} dx\lesssim R^{d-d^{'}} \int_{\mathbb{R}^{d^{'}}} |f_j|^{p_j} dx^{'}. 
\end{equation}
Apply \eqref{Section8LemmaProjectionLemmaAssumption} with the functions $F_{j, R}$, plug \eqref{Section8LemmaProjectionLemmaProof1} and \eqref{Section8LemmaProjectionLemmaProof2} into the inequality. We see the powers of $R$ on both sides cancel out. Letting $R\to \infty$ gives \eqref{Section8LemmaProjectionLemmaConclusion}.
\end{proof}
The assumptions that $\rho\in L^{\infty}$ and $\rho$ has a compact support are purely technical. Note that the implicit constant in \eqref{Section8LemmaProjectionLemmaConclusion} doesn't depend on $\rho$, so we can remove such technical assumptions in many cases by approximation. Now let $\rho$ be a singular integral kernel $K$. We expect, at least intuitively, that $T_{(A\mathfrak{L}_j)_j, K}$ is just a projection of $T_{(\mathfrak{L}_j)_j, K}$ and is easier to tackle than $T_{(\mathfrak{L}_j)_j, K}$.

From the operator in Example \ref{Section1ExampleFractionalRankSmalln}, we can choose $A$ suitably to recover all the six cases of the classification in \cite{DTTwoDimensionalBHT}. The first five cases are handled in \cite{DTTwoDimensionalBHT} using time-frequency analysis, but the last case is handled in \cite{KovacTwistedParaproduct} using a very different technique. Thus we believe tackling Example \ref{Section1ExampleFractionalRankSmalln} will require a non-trivial combination of the two very different techniques or a completely new idea.

\end{lemma}

\end{enumerate}


\bibliographystyle{plain}
\bibliography{ref}

\Address

\end{document}